\documentclass[11pt]{article}
\pdfoutput=1
\usepackage[T1]{fontenc}
\usepackage{latexsym,amssymb,amsmath,amsfonts,amsthm}
\usepackage{graphics}
\usepackage[section]{placeins}
\usepackage{graphicx}
\usepackage{mathrsfs}
\usepackage{subfigure}
\usepackage{float}
\usepackage{color}
\usepackage{url,epsfig,tikz}
\usepackage{subfigure}
\usepackage{pgfplots}
\usepackage{epstopdf}
\usetikzlibrary{shapes.geometric}

\newcommand{\R}{{\mathbb R}}

\newcommand{\N}{{\mathbb N}}

\newcommand{\be}{\begin{eqnarray}}
\newcommand{\ben}{\begin{eqnarray*}}
\newcommand{\en}{\end{eqnarray}}
\newcommand{\enn}{\end{eqnarray*}}

\newcommand{\pa}{\partial}

\newcommand{\ov}{\overline}
\newcommand{\real}{{\rm Re\,}}

\newcommand{\s}{\mathbb{S}}

\newcommand{\supp}{{\rm supp}}

\newtheorem{thm}{Theorem}[section]

\newtheorem{lem}{Lemma}[section]

\newtheorem{defn}{Definition}[section]

\newtheorem{rem}{Remark}[section]

\definecolor{rot}{rgb}{1,0,0}
\definecolor{hw}{rgb}{0,0,1}

\begin{document}
\renewcommand{\theequation}{\arabic{section}.\arabic{equation}}
\begin{titlepage}
\title{\bf Direct sampling method to inverse wave-number-dependent source problems (part I): determination of the support of a stationary source}
\author{ Hongxia Guo
 \thanks{School of Mathematical Sciences and LPMC, Nankai University, 300071 Tianjin, China. ({\tt hxguo$\_$math@163.com})}
 \and Guanghui Hu \thanks{School of Mathematical Sciences and LPMC, Nankai University, 300071 Tianjin, China. ({\tt ghhu@nankai.edu.cn})}  \and Mengjie Zhao \thanks{ (Corresponding author) School of Mathematical Sciences and LPMC, Nankai University, 300071 Tianjin, China. ({\tt 1120200028@mail.nankai.edu.cn })}
}
\date{}
\end{titlepage}
\maketitle


\begin{abstract}
This paper is concerned with a direct sampling method for imaging the support of a frequency-dependent source term embedded in a homogeneous and isotropic medium.
The source term is given by the Fourier transform of a time-dependent source whose radiating period in the time domain is known.
 The time-dependent source is supposed to be stationary in the sense that its compact support does not vary along the time variable.
 Via a multi-frequency direct sampling method, we show that the smallest strip containing the source support and perpendicular to the observation direction can be recovered from far-field patterns at a fixed observation angle. With multiple but sparse observations directions, the shape of the so-called $\Theta$-convex hull of the source support can be recovered. The frequency-domain analysis performed here can be used to handle inverse time-dependent source problems.
 Our algorithm has low computational overhead and is robust against noise. Numerical experiments in both two and three dimensions have proved our theoretical findings.

\vspace{.2in} {\bf Keywords: Inverse source problems, frequency-dependent source, multi-frequency data, Fourier transform, direct sampling method.
}
\end{abstract}

\section{Introduction}
\subsection{Mathematical formulation}
Let $0\leq k_{\min}<k_{\max}$ and $B_R=\{x\in \R^3: |x|<R\}$ for some $R>0$.
Consider a frequency-dependent radiating problem in a homogeneous background medium in three dimensions. The acoustic wave propagation can be governed by the inhomogeneous Helmholtz equation
\be\label{Helmholtz}\left\{\begin{array}{lll}
\Delta u(x, k)+k^2 u(x, k)=f(x, k), \quad & k\in (k_{\min}, k_{\max}),\; x\in \R^3,\\
\lim_{r\rightarrow\infty}r (\partial_r u-ik u)=0,\quad & r=|x|,
\end{array}\right.
\en
where $k>0$ is the wavenumber and $f$ is the frequency-dependent source term. The radiation condition at infinity in (\ref{Helmholtz}) is known as the
 Sommerfeld radiation condition, which holds uniformly in all directions $\hat{x}=x/|x|\in \s^2:=\{x\in \R^3: |x|=1\}$.
Throughout the paper we suppose that
 $f(\cdot, k)\in L^2(B_R)$, $\mbox{supp}f(\cdot, k)=D\subset B_R$  for all $k\in (k_{\min}, k_{\max})$. Here $D$ is supposed to be a bounded domain such that $\R^3\backslash\overline{D}$ is connected.
For every $k>0$,
it is well known that there exists a unique solution $u\in H^2(B_R)$ to \eqref{Helmholtz} with the explicit representation
\be\label{expression-w}
u(x,k)=\int_{D} \Phi^{(k)}(x,y)\, f(y,k)\,dy,\qquad x\in \R^3.
\en
Here, $\Phi^{(k)}$ is the fundamental solution to the Helmholtz equation $(\Delta+k^2)u=0$, given by
\ben
\Phi^{(k)}(x,y)=\frac{e^{ik |x-y|}}{4\pi |x-y|},\quad x\neq y.
\enn

The Sommerfeld radiation condition gives rise to the following asymptotic behavior of $u$ at infinity (see \cite{CK}):
\be\label{far-field}
u(x)=\frac{e^{ik|x|}}{4\pi|x|}\big\{u^\infty(\hat{x},k)+O(r^{-2})\big\}\quad\mbox{as}\quad|x|\rightarrow\infty,
\en
where $u^\infty(\cdot, k)\in C^\infty(\s^2)$ is defined as the
 far-field pattern (or scattering amplitude) of $u$. It is well known that the function  $\hat{x}\mapsto u^\infty(\hat{x}, k)$ is real-analytic on $\s^2$, where $\hat{x}\in \s^2$ is referred as the observation direction. By \eqref{expression-w}, the far-field pattern $u^\infty$ of $u$ can be expressed as
\be\label{u-infty}
u^\infty(\hat{x}, k)=\int_{D} e^{-ik\hat{x}\cdot y} f(y, k)\,dy,\quad \hat{x}\in \s^2,\quad k>0.
\en

In this paper we are interested in the following inverse problem: determine the position and shape of the support $D$ from knowledge of multi-frequency far-field patterns at sparse observation directions: $\{u^\infty(\hat{x}_j, k): k\in(k_{\min}, k_{\max}), j=1,2,\cdots, J\}$ for some $J\in \N$.
The aim of this paper is to explore a direct sampling method for imaging the support $D=\supp f(\cdot, k)$ from numerical point of view. As done in one of the authors' previous paper \cite{GGH}, the frequency-dependent source term is supposed to be a windowed Fourier transform of some time-dependent source term, i.e., there exists a time window $(t_{\min}, t_{\max})\subset \R_+$ such that
\be\label{fxt}
f(x,k)=\int_{t_{\min}}^{t_{\max}} F(x,t)e^{-ikt}\,dt, \quad x\in D,\;k\in(k_{\min}, k_{\max}).
\en
Here the function $F\in L^\infty(D\times (t_{\min}, t_{\max}))$ satisfies $\supp F(\cdot, t)=D$ for all $t\in(t_{\min}, t_{\max})$ and the time window $(t_{\min}, t_{\max})\subset \R^+$ is supposed to be available in advance. In this paper $F$
 is supposed to be real-valued  subject to the positivity condition
\be\label{F}
F(x, t)>0\qquad\mbox{a.e.}\; x\in D,\quad t\in (t_{\min}, t_{\max}).
\en
In particular, $F$ is allowed to vanish on the boundary of $D$ for $t\in (t_{\min}, t_{\max})$. It is obvious that the underlying source $F$ is stationary in the sense that the shape of its support does not vary with  the time.
By the assumption \eqref{fxt}, we have $f(x, -k)=\overline{f(x,k)}$ for all $k>0$ and thus $u^\infty(\hat{x}, -k)=\overline{u^\infty(\hat{x}, k)}$.

\subsection{Scientific context}
If $F(x,t)=f(x)\delta(t)$, it follows from \eqref{fxt} that $f(x,k)=f(x)$
is independent of frequencies. In this special case the far-field pattern $u^\infty$ given by \eqref{far-field} is nothing  else but the Fourier transform of the space-dependent function at the Fourier variable  $\xi=k\hat{x}\in \R^3$. For such space-dependent inverse source problems, it is well-known that measurement data at a single frequency is severely ill-posed and uniqueness is impossible due to the existence of non-radiating sources \cite{BC,EV09}. We refer to \cite{SK05,KS03,HL2020} for further discussions on recovering polygonal/polyhedral support of a source function with a single pair of data. To overcome the ill-posedness,
 a wide range of literatures are devoted to frequency-independent inverse problems from multi-frequency far-/near-field measurement data.
Uniqueness from the data of an interval of frequencies and observation angles follows straightforwardly, because
 the measurement data are analytic with respect to both the wavenumber $k>0$ and the observation angle. Consequently, a couple of iterative and non-iterative schemes (for example, sampling-type methods) have been proposed
for recovering the source function \cite{AKS, BLT10, CIL, EV09, ZG} as well as the shape of the support  \cite{AHLS,  Ik99, JLZ2019, GS, EH19, GH22}.
Concerning sampling-type methods to inverse obstacle scattering problems at a fixed energy, we refer the readers to the monographs \cite{CK,NP,KG08,P2001} for an overview. The direct sampling method \cite{CCH13,IJZ,L17,LMZ22} (which is also called orthogonal sampling method \cite{P10,G11}) has also drawn a lot of attentions, because it can be easily implemented and is rather robust against noise. 
The purpose of this paper is to explore a direct sampling scheme for imaging the support of a class of frequency-dependent source terms in the time-harmonic regime.

When $f$ depends on the wavenumber/frequency, the far-field pattern \eqref{u-infty} is no longer the Fourier transform of a source function, bringing essential difficulties in proving uniqueness and also in designing inversion schemes. The inverse medium scattering can be equivalently transformed into an inverse wave-number-dependent source problem.
To the best of the authors' knowledge, most existing methods cannot be carried over to this case for extracting information on the support and source functions. Although a direct sampling scheme \cite{AHLS} has been tested for recovering the shape of wave-number dependent sources,  the numerical examples there were not very satisfactory in comparison with the reconstructions for wave-number independent sources.
In \cite{GGH}, a factorization method has been established for recovering source terms of the form \eqref{fxt}, generalizing the earlier work of \cite{GS} to inverse wave-number-dependent source problems with a known radiating period $[t_{\min}, t_{\max}]$. As mentioned in the previous subsection, such kind of sources arises naturally from the Fourier transform of time-dependent source terms. Motivated by \cite{GGH} and \cite{AHLS}, we consider within this paper the same kind of wave-number-dependent source as \cite{GGH}. 
Our approach provides a tool for analyzing inverse time-dependent scattering problems in the frequency domain via Fourier transform.

We compare the factorization scheme of \cite{GGH} and the direct sampling method proposed in this paper as follows. Fixing an observation direction, one can apply both of them to recover the
the smallest strip containing the source support and perpendicular to the observation direction from the far-field data in two and three dimensions. Both of them are applicable to the near-field data at multi-frequencies in three dimensions.
The factorization method can be implemented with the data from any interval of frequencies, but the direct sampling scheme requires all frequencies in the positive axis. Like the multi-static case, the factorization scheme provides a sufficient and necessary condition for imaging the strip, while the direct sampling method yields only a sufficient condition. In this sense, the multi-frequency factorization method inherits the merits of its multi-static version, due to the solid analysis from  functional analysis. However, the direct sampling method can be more easily implemented, because it only involves inner product calculations. Further, the theoretical justification of the direct sampling, as done in this paper, can be performed under less restrictive mathematical conditions. For example, the positivity condition of the source function $F$ can be further relaxed (see Remark \ref{rem4.1}), but the source was required to be coercive in \cite{GGH, GS}.

In this paper the starting time point $t_{\min}$ and the terminal time point $t_{\max}$ of the source radiating period are both supposed to be a priori known. If one of them is given, we can also design a direct sampling scheme to recover the $\Theta$-convex hull of the connected support by increasing far-field observation directions. However, it remains unclear to us how to handle the case when neither of them is known. The direct sampling scheme explored within this paper, in particular the inversion theory from a single observation point, can be adapted to extract information on the trajectory of a moving source under additional constraints, which will be reported in our forthcoming papers.

The remaining part of this paper is organized as follows. In the subsequent two sections, we design a direct sampling scheme using multi-frequency far-field patterns at sparse observation direction in 2D. Numerical tests will be reported in Section \ref{sec:4} in the far-field case. In the final Section \ref{sec:5}, we extend the results of Sections \ref{sec:2} and \ref{sec:3} to the case of near-field measurement data in three dimensions.

\section{Inverse Fourier transform of far-field pattern}\label{sec:2}
Let $\hat{x}\in \mathbb{S}^2$ be an observation direction and $D\subset \mathbb{R}^3$ be a bounded domain with $C^2$-smooth boundary. Define
\ben
\hat{x}\cdot D=\{t\in \R: t=\hat{x}\cdot y\quad\mbox{for some}\quad y\in D\}.
\enn
Obviously, $\hat{x}\cdot D$ is an interval of $\R$.
Introduce the notations
\begin{equation}
      T:=t_{\max}-t_{\min}, \quad L:=\sup(\hat{x} \cdot D)-\inf(\hat{x} \cdot D),\quad \Lambda:=T+L.
\end{equation}
In this paper the one-dimensional Fourier and inverse transforms are defined respectively by
\ben
(\mathcal{F}g)(k):=\int_\R g(\xi)e^{-ik\xi}\,d\xi,\quad k\in \R,\\
(\mathcal{F}^{-1}f)(\xi):=\frac{1}{2\pi}\int_\R f(k)e^{ik\xi}\,dk,\quad \xi\in \R.\enn

\begin{defn}
The supporting interval of the function $f\in L^2(\R)$ is defined as the minimum interval such that $f$ vanishes almost every in the exterior of this interval.
\end{defn}

Below we describe the supporting interval of the inverse Fourier transform of the far-field pattern with respect to frequencies at a fixed observation direction.

\begin{lem}\label{21}
	Under the assumption $(\ref{F})$, the supporting interval of $(\mathcal{F}^{-1}u^{\infty}(\hat{x},k))(t)$ is $H:=\big(\inf(\hat{x}\cdot D)+t_{\min},\ \sup(\hat{x}\cdot D)+t_{\max}\big)$ and the function $t\longmapsto (\mathcal{F}^{-1}u^{\infty})(t)$ is positive in $H$.
\end{lem}

\begin{proof} Combining \eqref{u-infty} and \eqref{fxt},
	the far-field pattern $u^{\infty}$ of $u$ can be expressed as
	\ben
		u^{\infty}(\hat{x},k)
		=\int_{D}f(y,k)e^{-ik\hat{x}\cdot y}dy
		=\int_{t_{\min}}^{t_{\max}} \int_{D} F(y,t) e^{-ik(\hat{x}\cdot y+t)}dy dt .
		\enn
To write the above expression as a Fourier transform, we observe that
\ben
\int_D F(y,t)e^{-ik(\hat{x}\cdot y+t)}\,dy=\int_\R e^{-ik \xi}\left(\int_{\Gamma(\xi-t)}F(y,t)ds(y)\right)\, d\xi
\enn
where $\Gamma(t)\subset D$ is defined as
    \be
    \Gamma(t):=\left\lbrace y\in D:\hat{x}\cdot y=t\right\rbrace , \quad t\in \R. \nonumber
    \en
Note that when $\Gamma(\xi-t)=\emptyset$, the aforementioned integral over $\Gamma(\xi-t)$ is taken as zero. Consequently,
\be\label{fourier}
		u^{\infty}(\hat{x},k)=\int_{\R} e^{-ik\xi} g(\xi) d\xi =(\mathcal{F}g)(k),
		\en
with
	\be\label{g}
	    g(\xi):=\int_{t_{\min}}^{t_{\max}} \int_{\Gamma(\xi-t)}F(y,t)ds(y)dt=\int_{\xi-t_{\max}}^{\xi-t_{\min}} \int_{\Gamma(t)} F(y,\xi-t)ds(y)dt.
    \en
    Since $\Gamma(t)=\emptyset$ for $t<\inf(\hat{x}\cdot D)$ and $t>\sup(\hat{x} \cdot D)$, it is obvious that
 \ben
    g(\xi)=0 \qquad\qquad\mbox{if} \quad \xi<\inf(\hat{x}\cdot D)+t_{\min}\quad\mbox{or}\quad \xi>\sup(\hat{x} \cdot D)+t_{\max}.
    \enn
It remains to prove
\be\label{positivity}
g(\xi)>0\qquad\mbox{if}\quad \xi\in
\big(\inf(\hat{x}\cdot D)+t_{\min},\; \sup(\hat{x} \cdot D)+t_{\max}\big).
\en
For such $\xi$, we suppose that
\be\label{xi}
\xi=\inf(\hat{x}\cdot D)+t_{\min}+\epsilon\quad\mbox{for some}\quad \epsilon\in(0, \Lambda).
\en
Then it holds that
\ben
\xi-t\in (t_{\min}, t_{\min}+\epsilon),\quad\mbox{if}\quad t\in\big(\inf(\hat{x}\cdot D),\; \inf(\hat{x}\cdot D)+\epsilon\big).
\enn
Now we rewrite $g$ as
\ben
        g(\xi)=\int_{\inf(\hat{x}\cdot D)}^{\inf(\hat{x}\cdot D)+\varepsilon}\int_{\Gamma(t)}F(y,\xi-t)ds(y)dt.
    \enn
Observing for any $\epsilon\in(0, \Lambda)$ that
\ben
\big(\inf(\hat{x}\cdot D),\; \inf(\hat{x}\cdot D)+\epsilon\big)\cap \big(\inf(\hat{x}\cdot D), \sup(\hat{x}\cdot D)\big)\neq \emptyset,\\
\big(t_{\min}, t_{\min}+\epsilon\big) \cap
\big(t_{\min}, t_{\max}\big)\neq \emptyset,
\enn
we deduce from the positivity assumption of $F$  that
$g$ must be positive in $\big(\inf(\hat{x}\cdot D)+t_{\min},\; \sup(\hat{x} \cdot D)+t_{\max}\big)$.
\end{proof}

In Lemma \ref{lem2} below, we assume that $D=D_1\cup D_2\subset\mathbb{R}^3$ consists of two disconnected components $D_1$ and $D_2$. We make the following assumption on the distance between $D_1$ and $D_2$:
\begin{description}
   \item[Assumption (A):] It holds that
    \ben
    \sup(\hat{x}\cdot D_1)+t_{\max}<\inf(\hat{x}\cdot D_2)+t_{\min}\quad\mbox{or}\quad \sup(\hat{x}\cdot D_2)+t_{\max}<\inf(\hat{x}\cdot D_1)+t_{\min}.  \enn
In other words, $\inf(\hat{x}\cdot D_2)-\sup(\hat{x}\cdot D_1)>T$ or $\inf(\hat{x}\cdot D_1)-\sup(\hat{x}\cdot D_2)>T$. \end{description}
Physically, the Assumption (A) means that the wave signals radiated from $D_1$ and $D_2$ can be separated in the observation direction $\hat x$:
see Figs. \ref{fig:SatisfyA} and \ref{fig:NotA}. Otherwise, the support of the these waves can be overlapped along the propagating direction $\hat{x}$.

\begin{figure}
	\centering
	\begin{tikzpicture}
		\begin{axis}[
			axis lines=middle, xlabel={$x$},  ylabel={$y$},
			xmin=-6,xmax=6,
			ymin=-6,ymax=6,
			]
			\addplot [domain=0:2*pi,samples=100, line width=1,]({-3+cos (deg(x))+0.65*cos(2*deg(x))-0.65},{-3+1.5*sin (deg(x))});
			\addplot [domain=0:2*pi,samples=100, line width=1,]({2+1.5*cos (deg(x))},{2+0.5*sin (deg(x))});
			
		\end{axis}
		\draw [dashed, line width=1] (0.85, 0) -- (0.85, 6);
		\draw [dashed, line width=1] (2.3, 0) -- (2.3, 6);
		\fill [gray, opacity=0.2] (0.85, 0) -- (0.85, 6)-- (2.3, 6)--  (2.3, 0)--(0.85, 0);
		\node [below, font=\fontsize{6}{6}\selectfont] at (0.85, 0) {$\inf(\hat{x}\cdot D_1)$};
		\node [below, font=\fontsize{6}{6}\selectfont] at (2.3, 0) {$\sup(\hat{x}\cdot D_1)$};

		\draw [-, line width=1, color=blue] (1.15, 0.2) -- (1.15, 5.2);
		\draw [-, line width=1, color=red] (2.9, 0.2) -- (2.9, 5.2);
		\node [right, color=blue, font=\fontsize{6}{6}\selectfont,rotate=-90] at (1.3, 5) {$\inf(\hat{x}\cdot D_1)+t_{\min}$};
		\node [left, color=red, font=\fontsize{6}{6}\selectfont,rotate=-90] at (2.7, 2.65) {$\sup(\hat{x}\cdot D_1)+t_{\max}$};

		\draw [->, line width=1] (6.5,4 )--(7, 4);
		\node [right,font=\fontsize{6}{6}\selectfont,rotate=-90] at (7.3, 4.5) {$\hat{x}=(1,0)$};

		\draw [dashed, line width=1] (3.70, 0) -- (3.70, 6);
		\draw [dashed, line width=1] (5.45, 0) -- (5.45, 6);
		\fill [gray, opacity=0.8] (3.70, 0) -- (3.70, 6)-- (5.45, 6)--  (5.45, 0)--(3.70, 0);
		\node [below, font=\fontsize{6}{6}\selectfont] at (3.8, 0) {$\inf(\hat{x}\cdot D_2)$};
		\node [below, font=\fontsize{6}{6}\selectfont] at (5.5, 0) {$\sup(\hat{x}\cdot D_2)$};

		\draw [-, line width=1, color=blue] (4, 1) -- (4, 5.8);
		\draw [-, line width=1, color=red] (6.05, 1) -- (6.05, 5.8);
		\node [right, color=blue, font=\fontsize{6}{6}\selectfont,rotate=-90] at (4.2, 3.2) {$\inf(\hat{x}\cdot D_2)+t_{\min}$};
		\node [left, color=red, font=\fontsize{6}{6}\selectfont,rotate=-90] at (5.85, 0.85) {$\sup(\hat{x}\cdot D_2)+t_{\max}$};
		
		\node [above right] at (1.2, 1.1) {$D_1$};
		\node [above right] at (4.2, 3.45) {$D_2$};
	\end{tikzpicture}
	\caption{Examples of $D_1$ and $D_2$ satisfying  the Assumption (A). Here $t_{\min}=0.5$, $t_{\max}=1$, $D_1$ is kite-shaped, $D_2$ is elliptical and $\hat{x}=(1,0)$.
	} \label{fig:SatisfyA}
	
\end{figure}
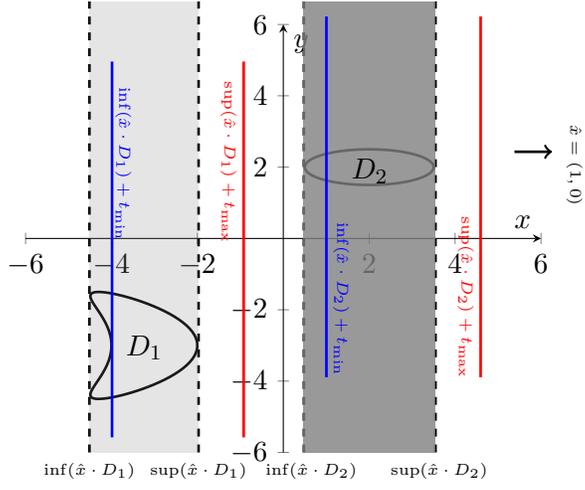

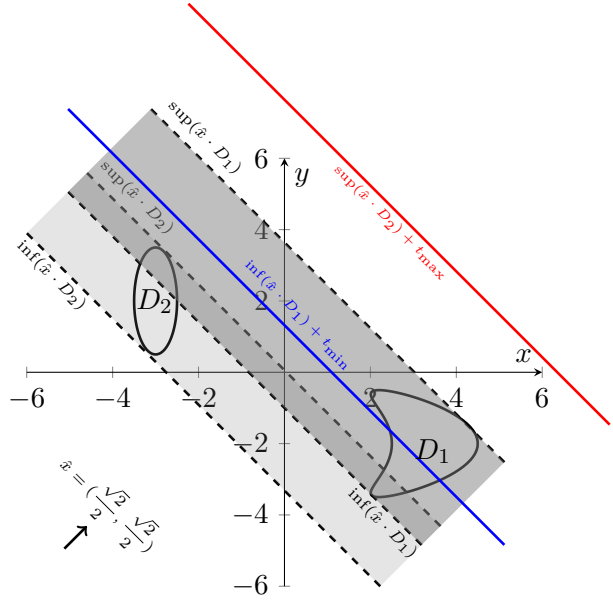
\begin{figure}
	\centering
	\begin{tikzpicture}
		\begin{axis}[
			axis lines=middle, xlabel={$x$},  ylabel={$y$},
			xmin=-6,xmax=6,
			ymin=-6,ymax=6,
			]
			\addplot [domain=0:2*pi,samples=100, line width=1,]({3.5+cos (deg(x))+0.65*cos(2*deg(x))-0.65},{-2+1.5*sin (deg(x))});
			\addplot [domain=0:2*pi,samples=100, line width=1,]({-3+0.5*cos (deg(x))},{2+1.5*sin (deg(x))});
			
		\end{axis}
		\draw [dashed, line width=1] (4.7, 0) -- (0, 4.7);
		\draw [dashed, line width=1] (5.5, 0.8) -- (0.8, 5.5);
		\fill [gray, opacity=0.2] (4.7, 0) -- (0, 4.7)-- (0.8, 5.5)--  (5.5, 0.8)--(4.7, 0);
		\node [below, font=\fontsize{6}{6}\selectfont,rotate=-45] at (0.5, 4.3) {$\inf(\hat{x}\cdot D_2)$};
		\node [above, font=\fontsize{6}{6}\selectfont,rotate=-45] at (1.3, 5) {$\sup(\hat{x}\cdot D_2)$};

		\draw [-, line width=1, color=red] (7.75, 2.15) -- (2.15, 7.75);
		\node [below, color=red, font=\fontsize{5}{6}\selectfont,rotate=-45] at (5, 5) {$\sup(\hat{x}\cdot D_2)+t_{\max}$};

		\draw [->, line width=1] (0.5,0.5 )--(0.8, 0.8);
		\node [above,font=\fontsize{6}{6}\selectfont,rotate=-45] at (0.8, 0.8) {$\hat{x}=(\dfrac{\sqrt{2}}{2},\dfrac{\sqrt{2}}{2})$};

		\draw [dashed, line width=1] (5.25, 0.55) -- (0.55, 5.25);
		\draw [dashed, line width=1] (6.35, 1.65) -- (1.65, 6.35);
		\fill [gray, opacity=0.5] (5.25, 0.55) -- (0.55, 5.25)-- (1.65, 6.35)--  (6.35, 1.65)--(5.25, 0.55);
		\node [below, font=\fontsize{6}{6}\selectfont,rotate=-45] at (4.9, 1) {$\inf(\hat{x}\cdot D_1)$};
		\node [above, font=\fontsize{6}{6}\selectfont,rotate=-45] at (2.2, 5.8) {$\sup(\hat{x}\cdot D_1)$};
		\draw [-, line width=1,color=blue] (6.35, 0.55) -- (0.55, 6.35);
		
		\node [above, color=blue, font=\fontsize{5}{6}\selectfont,rotate=-45] at (3.45, 3.45) {$\inf(\hat{x}\cdot D_1)+t_{\min}$};
		
		\node [above] at (1.7, 3.5) {$D_2$};
		\node [above] at (5.4, 1.5) {$D_1$};
	\end{tikzpicture}
	\caption{Examples of $D_1$ and $D_2$ which do not satisfy the assumption (A). Here $t_{\min}=1$, $t_{\max}=5$ and $\hat{x}=(\sqrt{2}/2, \sqrt{2}/2)$. The acoustic waves radiated from $D_1$ and $D_2$ can be overlapped.
	} \label{fig:NotA}

\end{figure}

\begin{lem}\label{lem2}
	The supporting interval of $(\mathcal{F}^{-1}u^{\infty}(\hat{x},k))(t)$ is $H_1\cup H_2$ where $H_j=\big(\inf(\hat{x}\cdot D_j)+t_{\min},\ \sup(\hat{x}\cdot D_j)+t_{\max}\big)$ and the function $t\longmapsto (\mathcal{F}^{-1}u^{\infty})(t)$ is positive in $H_1\cup H_2$. Further,
$H_1\cup H_2$ is a subset of $H=\big(\inf(\hat{x}\cdot D)+t_{\min},\ \sup(\hat{x}\cdot D)+t_{\max}\big)$ under the Assumption (A), and $H_1\cup H_2=H$ if the Assumption (A) does not hold.

\end{lem}
\begin{rem}
The assumption (A) ensures that the intervals $H_1$ and $H_2$ have no intersections, i.e., $H_1\cap H_2=\emptyset$. In particular, $({x}\cdot D_1) \cap   (\hat{x}\cdot D_2)=\emptyset$.
\end{rem}

\begin{proof}
	As done in the proof of Lemma \ref{21}, the far-field far-field takes the form 
	\be\label{2.12}
	  u^{\infty}(\hat{x},k) \nonumber
	   &=&\int_{t_{\min}}^{t_{\max}} \int_{D_1\cup D_2} F(y,t) e^{-ik(\hat{x}\cdot    y+t)}dy dt  \\
	   &=&\int_{\R} e^{-ik\xi} g(\xi) d\xi =(\mathcal{F}g)(k).
	\en
Here $g=g_1+g_2$ with the functions $g_j$ ($j=1,2$) defined as
	\be\label{gj}
	g_j(\xi):=\int_{t_{\min}}^{t_{\max}} \int_{\Gamma_j(\xi-t)}F(y,t)ds(y)dt=\int_{\xi-t_{\max}}^{\xi-t_{\min}} \int_{\Gamma_j(t)} F(y,\xi-t)ds(y)dt.
	\en
The set $\Gamma_j(t)\subset D_j$ is defined as
	\be
	\Gamma_j(t):=\left\lbrace y\in D_j:\hat{x}\cdot y=t\right\rbrace ,\quad t\in \R. \nonumber
	\en
Hence, the inverse Fourier transform of the multi-frequency far-field patterns \eqref{2.12} is $\mathcal{F}^{-1}u^{\infty}=g=g_1+g_2$.
It follows from Lemma \ref{21} that $\supp (g_j)=H_j$ and $g_j>0$ in $H_j$ for $j=1,2$.
This proves  $\supp (\mathcal{F}^{-1} u^\infty)=\supp (g)=H_1\cup H_2$ and $g$ is positive in $H_1\cup H_2$.

 Noting that  $H_1\cap H_2=\emptyset$ under the assumption (A), we obtain that $H_1\cup H_2$ is a subset of $H$. If $H_1\cap H_2\neq \emptyset$, it is obvious that $H_1\cup H_2=H$.

\end{proof}

\section{Indicator and test functions}\label{sec:3}
From Lemmas \ref{21} and \ref{lem2}, one can extract information on the interval $\big(\inf(\hat{x}\cdot D)+t_{\min},\ \sup(\hat{x}\cdot D)+t_{\max}\big)$
 by taking the inverse Fourier transform of $u^\infty (\hat{x}, k)$ for all $k>0$. Since $t_{\min}$ and $t_{\max}$ are both given, this obviously yields information on the strip $\hat{x}\cdot D$.
 In this section we shall design proper indicator functions to  image the support $\hat{x}\cdot D$ in a more straightforward manner.

 We first suppose that $D$ is connected. Given an observation direction $\hat{x}$, we want to design an indicator function $y\longmapsto I(y)$, $y\in \R^2$ such that
 \ben
 I(y)=\left\{\begin{array}{lll}
 0\quad&&\mbox{if} \quad\hat{x}\cdot y\in \hat{x}\cdot D,\\
\mbox{finite positive number} \quad&&\mbox{if \; otherwise}.
  \end{array}\right.
 \enn
Introduce two test functions
\ben
\phi^{(\hat{x})}_{1}(y,k):=e^{-ik\left( \hat{x}\cdot y+{t_{\min}}\right)},\quad
\phi^{(\hat{x})}_{2}(y,k):=e^{-ik\left( \hat{x}\cdot y+{t_{\max}}\right)},
\enn
and two auxiliary indicator functions
\be\label{Ind}
I^{(\hat{x})}_j(y) =\int_{\R}u^{\infty}(\hat{x},k) \ \overline{\phi_j^{(\hat{x})}(y,k)}\ dk, \quad j=1,2,
\en
where $y\in \R^2$ is referred to as the sampling variable. Below we characterize the support of $I^{(\hat{x})}_j$ ($j=1,2$). 

\begin{thm}\label{ID10} Let $\hat{x}\in \mathbb{S}^2$ be fixed.
We have
\be\label{In1}
I^{(\hat{x})}_1(y)
=\left\{\begin{array}{lll}
0\qquad\quad&&\mbox{for}\quad\hat{x}\cdot y \notin \big(\inf(\hat{x}\cdot D),\sup(\hat{x}\cdot D)+T\big),\\
 {\rm finite\; positive\; number} \qquad&&\mbox{for}\quad\hat{x}\cdot y \in \big(\inf(\hat{x}\cdot D),\sup(\hat{x}\cdot D)+T\big);
\end{array}\right. \\ \label{In2}
I^{(\hat{x})}_2(y)=\left\{\begin{array}{lll}
 0\qquad\quad&&\mbox{for}\quad
 \hat{x}\cdot y \notin \big(\inf(\hat{x}\cdot D)-T,\sup(\hat{x}\cdot D)\big),\\
 {\rm finite\; positive\; number}\qquad&&\mbox{for}\quad
 \hat{x}\cdot y \in \big(\inf(\hat{x}\cdot D)-T,\sup(\hat{x}\cdot D)\big).
 \end{array}\right.
 \en
\end{thm}

\begin{proof}
Denote by $\delta(\cdot)$ the Dirac delta function.
The test functions $\phi^{(\hat{x})}_{j}(y,k)$ ($j=1,2$) can be rephrased as
	\be\begin{split}\label{test}
	\phi^{(\hat{x})}_{1}(y,k)
	= \int_{\R} e^{-ik\xi}\delta(\xi-\hat{x}\cdot y-t_{\min})  d\xi
	= [\mathcal{F}(\delta(\xi-\hat{x}\cdot y-t_{\min}))](k),\\
	\phi^{(\hat{x})}_{2}(y,k)
	= \int_{\R} e^{-ik\xi}\delta(\xi-\hat{x}\cdot y-t_{\max})  d\xi
	= [\mathcal{F}(\delta(\xi-\hat{x}\cdot y-t_{\max}))](k).
	\end{split}
	\en
   Applying the Parserval's identity and properties of Fourier transform, we deduce from \eqref{fourier} and \eqref{Ind} that
	 \be \label{I1}
	I^{(\hat{x})}_1(y)&=&\int_{\R}u^{\infty}(\hat{x},k)  \overline{\phi^{(\hat{x})}_{1}(k)} dk \nonumber \\
	&=& \int_{\R}[\mathcal{F}g(\xi)](k) \; \overline{[\mathcal{F}(\delta(\xi-\hat{x}\cdot y-t_{\min}))](k)} dk \nonumber \\
	&=& \int_{\R} g(\xi)\; \delta(\xi-\hat{x}\cdot y-t_{\min}) d \xi \nonumber \\
	&=& g(\hat{x}\cdot y+t_{\min}).
	\en
	where $g$ is defined by \eqref{g}.
	From Lemma \ref{21}, we know $ \supp\, g=\big(\inf(\hat{x}\cdot D)+t_{\min},\ \sup(\hat{x}\cdot D)+t_{\max}\big)$ and $g$ is positive in $\supp \,g$. Therefore,
	$g(\hat{x}\cdot y+t_{\min})>0$ for all $y\in \R^3$ such that
	$\hat{x}\cdot y \in \big(\inf(\hat{x}\cdot D),\sup(\hat{x}\cdot D)+T\big)$. On the other hand, it is also obvious that 	
$g(\hat{x}\cdot y+t_{\min})=0$ for all
	$\hat{x}\cdot y \notin \big(\inf(\hat{x}\cdot D),\sup(\hat{x}\cdot D)+T\big)$, which proves \eqref{In1}. The results in \eqref{In2} can be verified analogously.	
\end{proof}
The indicator function for imaging $\hat{x}\cdot D$ is defined as follows:
\ben
I^{(\hat{x})}(y):=\left[\frac{1}{I_1^{(\hat{x})}(y)}+  \frac{1}{I_2^{(\hat{x})}(y)}   \right]^{-1}=\frac{I_1^{(\hat{x})}(y)\; I_2^{(\hat{x})}(y)}{I_1^{(\hat{x})}(y)+I_2^{(\hat{x})}(y)}.
\enn
Now we state the indicating behavior of $I^{(\hat{x})}$.
\begin{thm}\label{ID1} Let $\hat{x}\in \mathbb{S}^2$ be fixed and assume that $D$ is connected. Then
\be\label{Indicator}
I^{(\hat{x})}(y)
=\left\{\begin{array}{lll}
0\qquad\quad&&\mbox{for}\quad\hat{x}\cdot y \notin \hat{x}\cdot D,\\
{\rm finite\; positive\; number} \qquad&&\mbox{for}\quad\hat{x}\cdot y \in \hat{x}\cdot D.
\end{array}\right.
\en
\end{thm}
\begin{proof}
It is obvious that
$$\hat{x}\cdot D = \big(\inf(\hat{x}\cdot D)-T,\sup(\hat{x}\cdot D)\big)\,\cap \, \big(\inf(\hat{x}\cdot D),\sup(\hat{x}\cdot D)+T\big).$$
Hence, if $\hat{x}\cdot y\in \hat{x}\cdot D$, one deduces from Theorem \ref{ID10} that $0<I_j^{(\hat{x})}(y)<\infty$ for $j=1,2$, implying  $0<I^{(\hat{x})}(y)<\infty$.

On the other hand, if $\hat{x}\cdot y\notin \hat{x}\cdot D$, we have either $\hat{x}\cdot y \notin \big(\inf(\hat{x}\cdot D)-T,\sup(\hat{x}\cdot D)\big)$ or
$\hat{x}\cdot y \notin \big(\inf(\hat{x}\cdot D),\sup(\hat{x}\cdot D)+T\big)$.
In the former case, we get $I_1^{(\hat{x})}(y)=0$ and in the latter case $I_2^{(\hat{x})}(y)=0$. Hence, there must hold  $I^{(\hat{x})}(y)=0$ for $\hat{x}\cdot y\notin \hat{x}\cdot D$.
\end{proof}

If $D$ consists of two components, we can get analogous results to Theorem \ref{ID1}. In the following theorem,
 the intervals $H_1$ and $H_2$ corresponding to $D_1$ and $D_2$ are defined as in Lemma \ref{lem2}.


\begin{thm}\label{ID2}
Let $D=D_1\cup D_2$ where $D_j\subset \R^3$ are  bounded domains such that $D_1\cap D_2=\emptyset$ and let $\hat{x}\in \mathbb{S}^2$ be fixed. Under the Assumption (A), we have
\be\label{Indicator2}
I^{(\hat{x})}(y)
=\left\{\begin{array}{lll}
0\qquad\quad&&\mbox{for}\quad\hat{x}\cdot y \notin \bigcup_{j=1,2}\{\hat{x}\cdot D_j\},\\
{\rm finite\; positive\; number} \qquad&&\mbox{for}\quad\hat{x}\cdot y \in \bigcup_{j=1,2}\{\hat{x}\cdot D_j\}.
\end{array}\right.
\en
If the Assumption (A) does not hold, then $I^{(\hat{x})}(y)=0$ for $\hat{x}\cdot y \notin \hat{x}\cdot D$.
\end{thm}

\begin{proof} Let $g_1$ and $g_2$ be defined as in \eqref{gj} with the support $\supp (g_j)=H_j$ for $j=1,2$. Analogously to the proof of Theorem \ref{ID1}, we obtain
\ben
&&I_1^{(\hat{x})}(y)=g_1(\hat{x}\cdot y+t_{\min})+g_2(\hat{x}\cdot y+t_{\min}), \\ &&I_2^{(\hat{x})}(y)=g_1(\hat{x}\cdot y+t_{\max})+g_2(\hat{x}\cdot y+t_{\max}).
\enn
Under the Assumption (A), it follows from Lemma \ref{lem2} that $\supp (g_1+g_2)=H_1\cup H_2$ and $g_1+g_2$ is positive in $H_1\cup H_2$. Hence,
\ben
I^{(\hat{x})}_1(y)=
0\quad\mbox{if}\quad\hat{x}\cdot y \notin \bigcup_{j=1,2}\big(\inf(\hat{x}\cdot D_j),\;\sup(\hat{x}\cdot D_j)+T\big), \\
0<I^{(\hat{x})}_1(y)<\infty
\quad\mbox{if}\quad\hat{x}\cdot y \in \bigcup_{j=1,2}\big(\inf(\hat{x}\cdot D_j),\;\sup(\hat{x}\cdot D_j)+T\big),\\
I^{(\hat{x})}_2(y)=
0\quad\mbox{if}\quad\hat{x}\cdot y \notin \bigcup_{j=1,2}\big(\inf(\hat{x}\cdot D_j)-T,\;\sup(\hat{x}\cdot D_j)\big), \\
0<I^{(\hat{x})}_1(y)<\infty
\quad\mbox{if}\quad\hat{x}\cdot y \in \bigcup_{j=1,2}\big(\inf(\hat{x}\cdot D_j)-T,\;\sup(\hat{x}\cdot D_j)\big).
\enn
By the definition of $I^{(\hat{x})}$, one can get the relations shown in \eqref{Indicator2} under the Assumption (A).  Without the Assumption (A), we deduce from the proof of
Theorem \ref{ID1}
 that
\ben
\supp (I_1^{(\hat{x})})\,\subset\, \big(\inf(\hat{x}\cdot D),\ \sup(\hat{x}\cdot D)+T\big),\\
\supp (I_2^{(\hat{x})})\,\subset\, \big(\inf(\hat{x}\cdot D)-T,\ \sup(\hat{x}\cdot D)\big)
\enn and $I_j^{(\hat{x})}>0$ in $\supp (I_1^{(\hat{x})})$ for $j=1,2$.
This yields
 $$\supp (I^{(\hat{x})})  = \supp (I_1^{(\hat{x})})\cap \supp (I_2^{(\hat{x})})= \big(\inf(\hat{x}\cdot D),\ \sup(\hat{x}\cdot D)\big) =\hat{x}\cdot D.$$ In other words, $I^{(\hat{x})}(y)=0$ for $\hat{x}\cdot y \notin \hat{x}\cdot D$.
\end{proof}


\section{Two-dimensional numerical experiments with multi-frequency far-field patterns}\label{sec:4}

In this section, we shall numerically reconstruct the source support $D\subset \R^2$ from the multi-frequency far-field data $\left\lbrace u^\infty(\hat{x}_{m},  k): k\in(0,K),m=1,2, \cdots , M \right\rbrace $,
where $\hat{x}_{m}\in \s^2$ and $K>0$ is a truncated number. Unless otherwise stated, we suppose that $D$ is connected.
Since the function $k\mapsto u^\infty(\hat{x},  k)$ can be analytically extended onto the negative axis by $u^\infty(\hat{x},  -k)=\overline{u^\infty(\hat{x},  k)}$ for $k>0$, the indicators functions \eqref{Ind} can be approximated by
\be\nonumber
I^{(\hat{x})}_j(y)&\approx& \int_{-K}^K u^{\infty}(\hat{x},k) \ \overline{\phi_j^{(\hat{x})}(y,k)}\ dk \\ \nonumber
&=& \int_{0}^K u^{\infty}(\hat{x},k) \ \overline{\phi_j^{(\hat{x})}(y,k)}\ dk+
\int_{-K}^0 u^{\infty}(\hat{x},k) \ \overline{\phi_j^{(\hat{x})}(y,k)}\ dk \\ \nonumber
&=&\int_{0}^K u^{\infty}(\hat{x},k) \ \overline{\phi_j^{(\hat{x})}(y,k)}\ dk+
\int_{0}^K \overline{u^{\infty}(\hat{x},k)} \ \phi_j^{(\hat{x})}(y,k)\ dk \\ \label{TI}
&=&2\real\left\{\int_{0}^K u^{\infty}(\hat{x},k) \ \overline{\phi_j^{(\hat{x})}(y,k)}\ dk\right\}
\en
for $j=1,2$. Using multiple observations, we shall utilize the reciprocal of the sum of $(I^{(\hat{x}_m)})^{-1}$ ($m=1,2, \cdots, M$) as a new indicator, that is,
\be
 I(y)=\left[\sum_{m=1}^{M}\dfrac{1}{I^{(\hat{x}_m)}(y)}\right]^{-1}=\left[\sum_{m=1}^{M}
 \frac{I_1^{(\hat{x}_{m})}(y)+I_2^{(\hat{x}_{m})}(y)}{I_1^{(\hat{x}_{m})}(y)\; I_2^{(\hat{x}_{m})}(y)}\right]^{-1}.
  \en
The convex hull of the source region $D$ (that is, the intersections of all half spaces containing $D$) can be approximated by the intersection of the strips $\mathcal{S}^{(\hat{x}_m)}:= \{y\in \R^2: \hat{x}\cdot y\in \hat{x}_m\cdot D\}$.
Define the $\Theta$-convex hull of $D$ (see \cite{SK05}) determined by the directions $\{\hat{x}_m: m=1,2,...,M\}$  as
\ben
\Theta_D:=\bigcap_{m=1}^M \mathcal{S}^{(\hat{x}_m)}.
\enn
Since $\Theta_D\subset \mathcal{S}^{(\hat{x}_m)}$ for all $m=1,2,...,   M$, one can reconstruct the $\Theta$-convex hull of $D$ from sparse observation directions by plotting the indicator function $I(y)$.
\begin{thm}
Let $D\subset \R^2$ be connected. We have
\be
I(y)=\left\{\begin{array}{lll}
0\qquad\quad&&\mbox{for}\quad y\notin \Theta_D,\\
{\rm finite\; positive\; number} \qquad&&\mbox{for}\quad y
\in \Theta_D.
\end{array}\right.
\en
\end{thm}
\begin{proof}
If $y\in\Theta_D$, then $y\in\mathcal{S}^{(\hat{x}_m)}$ for all $m=1,2,...,M$, yielding  that $\hat{x}_m \cdot y \in \hat{x}_m \cdot D$. Hence, one deduces from Theorem \ref{ID10} that $0<I^{(\hat{x}_m)}(y)<\infty$ for all $m=1,2,...,M$, implying $0<I(y)<\infty$.
On the other hand, if $y\notin\Theta_D$,  it holds that $y\notin \mathcal{S}^{(\hat{x}_{\ell})}$ for some $\hat{x}_{\ell}$. Again using Theorem \ref{ID10} we obtain $\frac{1}{I^{(\hat{x}_{\ell})}(y)}=\infty$. Hence, there must hold  $I(y)=0$ for $y\in\Theta_D$.
\end{proof}

\noindent To balance the indicator values for different observation directions, we normalize the indicator functions $I$ and $I^{(\hat{x}_m)}$ by
\ben
\tilde{I}^{(\hat{x}_m)}(y):=
\frac{I^{(\hat{x}_m)}(y)-\min_{y\in \Omega} I^{(\hat{x}_m)}(y)}{\max_{y\in \Omega} I^{(\hat{x}_m)}(y)-\min_{y\in \Omega} I^{(\hat{x}_m)}(y)}
\enn
\ben
I(y)=\left[\sum_{m=1}^{M}\dfrac{1}{\tilde{I}^{(\hat{x}_m)}(y)}\right]^{-1},\quad\tilde{I}(y):=
\frac{I(y)-\min_{y\in \Omega} I(y)}{\max_{y\in \Omega} I(y)-\min_{y\in \Omega} I(y)}.
\enn
where $\Omega \supset D$ is a search domain for imaging $D$.
The normalization ensures that the values of the indicators $\tilde{I}^{(\hat{x}_m)}$ and $\tilde{I}$ are positive with the maximum value one.
By Theorem \ref{ID1}, one can expect that $\tilde{I}$ should be relatively larger in $D$ than those values in the exterior of $D$.
Our numerical examples are tested by the normalized indicators.

With all observations $\hat{x}\in\mathbb{S}^2$, one can easily obtain the following uniqueness result.

\begin{thm} Suppose that the positivity condition \eqref{F} holds on the connected domain
$D\subset \R^3$ and the radiating period $(t_{\min}, t_{\max})$ is given. Then the multi-frequency far-field patterns $\left\lbrace u^\infty(\hat{x},  k): \forall\;\hat{x}\in \s^2,k\in(0,K)\right\rbrace $ with some $K>0$ uniquely determine the convex hull of $D$.
\end{thm}
Note that the knowledge of $u^\infty(\hat{x},  k)$ in the frequency interval $(0, K)$ is equivalent to the far-field data for all $k>0$, due to the analyticity of $u^\infty(\hat{x},  k)$ in $k>0$.

\begin{rem}\label{rem4.1}
We think that the positivity condition over $D$ can be further relaxed, for example, $F$ remains positive only in a neighborhood of $\partial \left(D\times (t_{\min}, t_{\max})\right)$. Even without the positivity assumption, one can get the same uniqueness if the source function $F$ and the boundary of the domain $D$ are both analytic.
\end{rem}

\begin{rem}
In the direct sampling scheme, it is essential to explain properties of the indicator function $I^{(\hat{x})}$ at a fixed observation point. In this paper the indicating behavior of the truncated indicator \eqref{TI} is interpreted as the approximation of the original indicators \eqref{Ind} whose properties are rigorously justified in Theorem \ref{ID10}. This is different from an alternative explanation made in \cite{AHLS} where the truncated indicator was shown to decay as $1/|y|$ when the sampling variable $y$ moves away from $D$.
\end{rem}

\subsection{Reconstructions from a single observation direction}
In this subsection, we illustrate numerical reconstructions of the strip $\mathcal{S}^{(\hat{x})}$ from multi-frequency far-field data $\left\lbrace u^\infty(\hat{x},  k_n): k_n=nK/N\in(0,K],\ n=1,2,\cdots , N\right\rbrace $ in $\R^2$ using the indicator $\tilde{I}^{(\hat{x})}$.
It is supposed that $\hat{x}=(\cos\theta,\ \sin\theta)$ with some fixed $\theta \in \left( 0,2\pi\right] $.
We select the source functions that are real-valued and are constrained by (\ref{F}). If not otherwise stated, we choose $K=20$ and $N=200$ as the default setting. In Figs \ref{fig:1dirI}  (a), (b) and (c), we set $F(x,t)=1$ supported on a kite-shaped domain, fix $\theta=\pi/2$ and plot the normalized indicators $\tilde{I}_j^{(\hat{x})}$ for $j=1,2$ and
$\tilde{I}^{(\hat{x})}$.
 To make inversion results more intuitive, we set a threshold $\varepsilon>0$ in Figs. (d), (e) and (f), respectively. It is obvious that
 $\tilde{I}_j^{(\hat{x})}$ ($j=1,2$) can be used to capture the strip $\hat{x}\cdot D$ with a shift $T=2$ on the upper and bottom boundaries. The strip $\mathcal{S}^{(\hat{x})}$ has been accurately recovered by the normalized indicator $\tilde{I}^{(\hat{x})}$.

\begin{figure}[H]
	\centering
	\subfigure[Indicator $\tilde{I}_1^{(\hat{x})}$]{
		\includegraphics[scale=0.3]{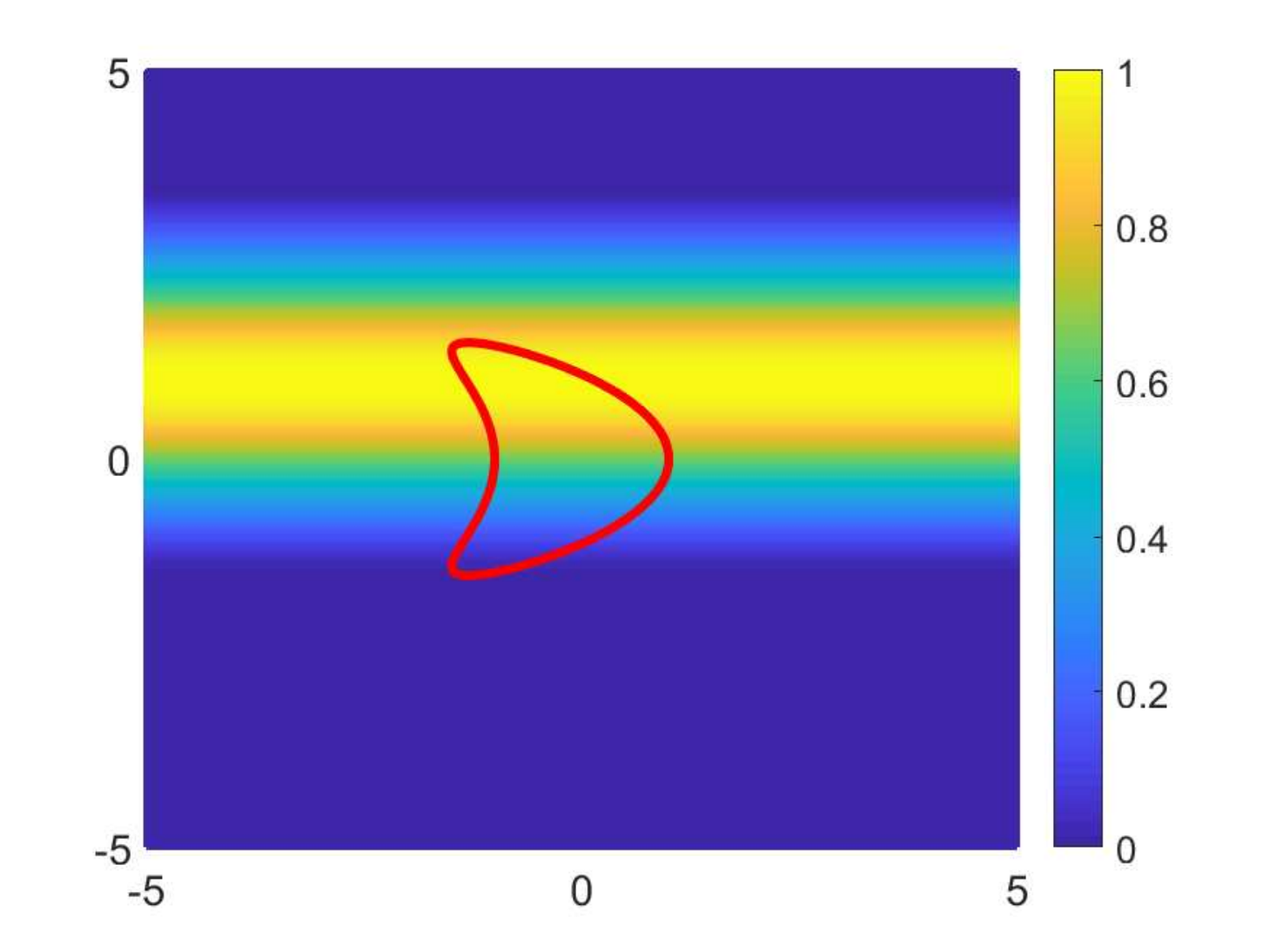}
	}	
	\subfigure[Indicator $\tilde{I}_2^{(\hat{x})}$]{
		\includegraphics[scale=0.3]{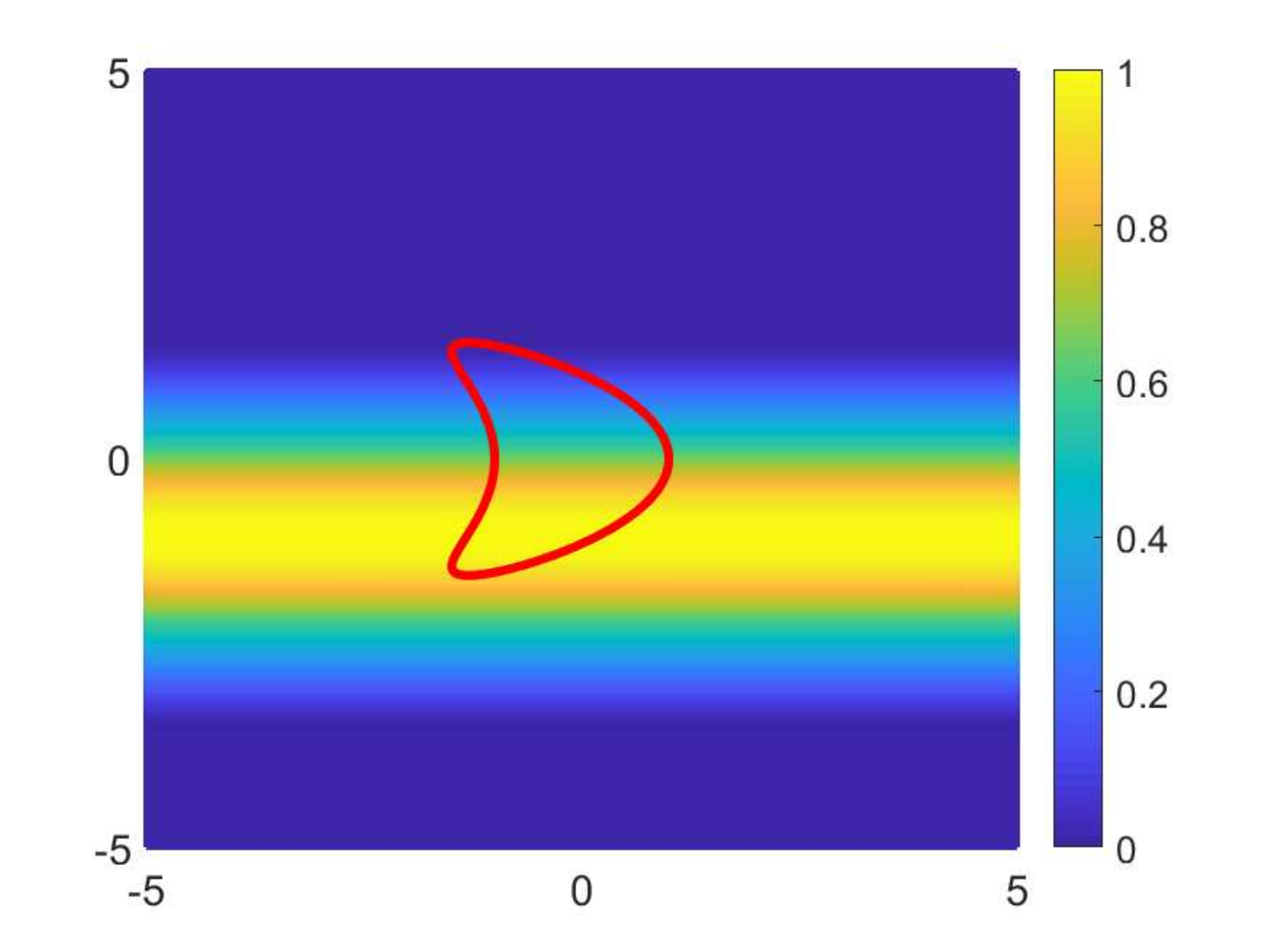}
	}
	\subfigure[Indicator $\tilde{I}^{(\hat{x})}$ ]{
		\includegraphics[scale=0.3]{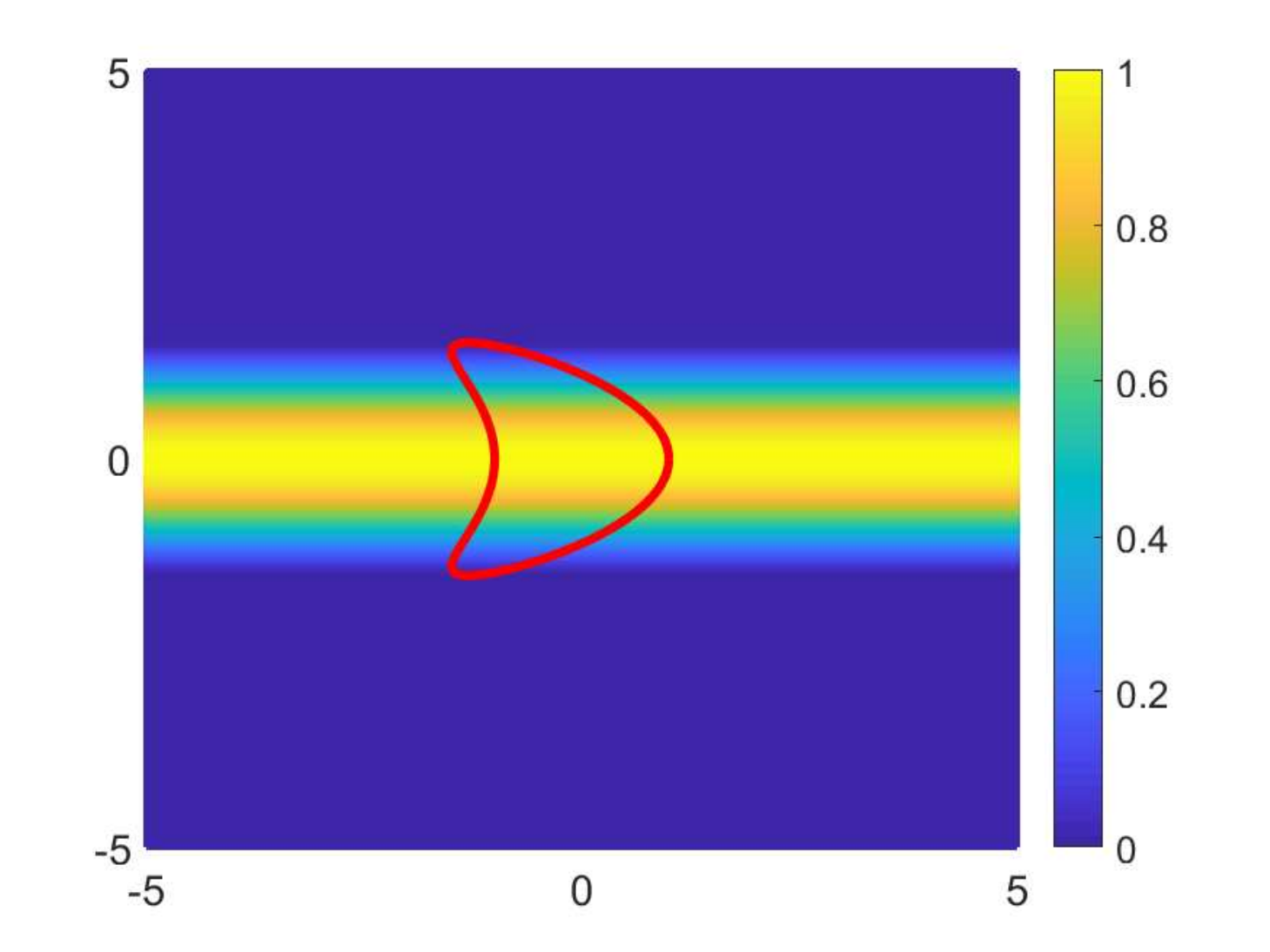}	
	}

	\subfigure[ $\tilde{I}_1^{(\hat{x})}$, $\varepsilon=0.007$]{
		\includegraphics[scale=0.3]{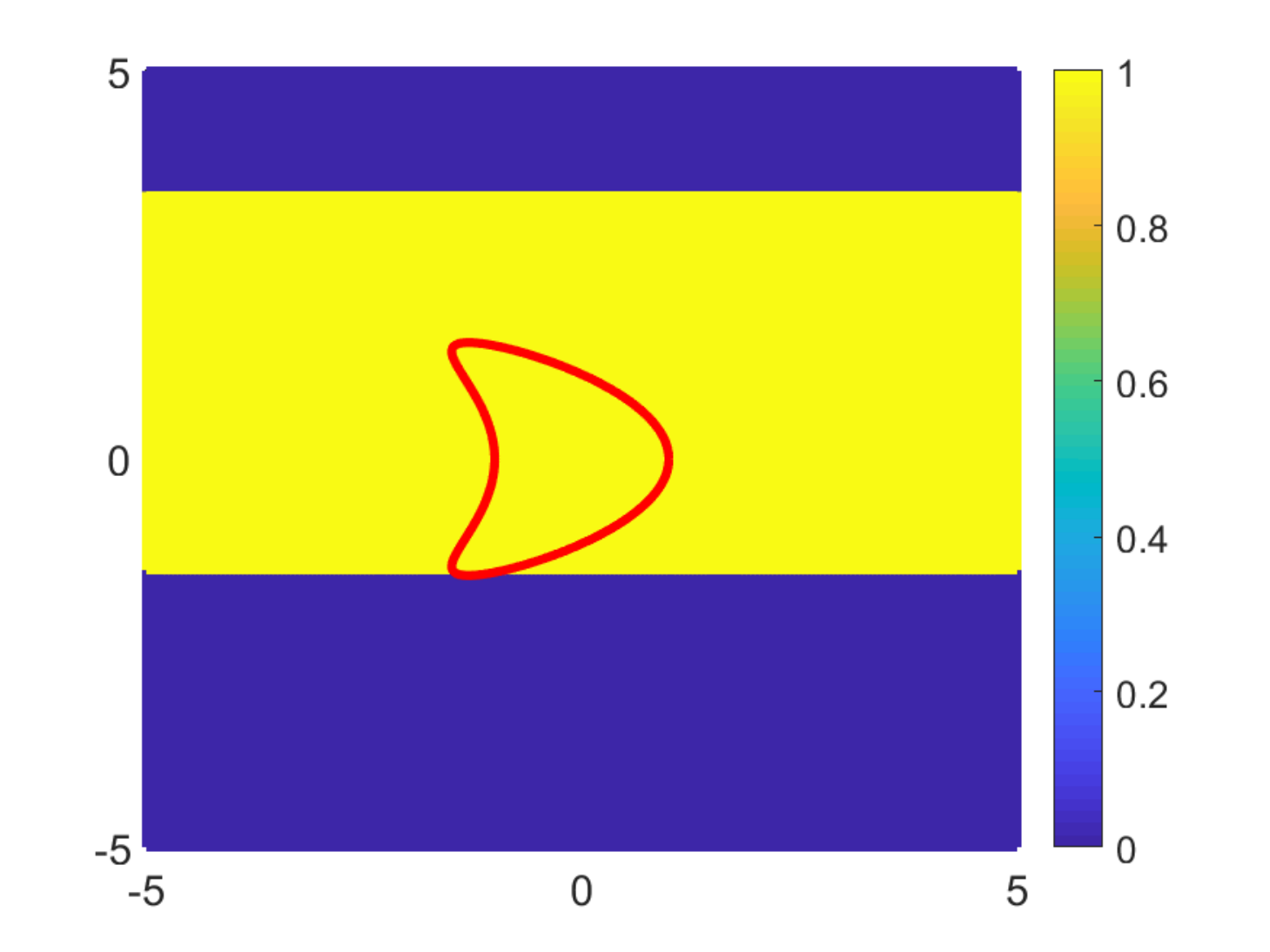}		
	}
	\subfigure[ $\tilde{I}_2^{(\hat{x})}$, $\varepsilon=0.007$]{
		\includegraphics[scale=0.3]{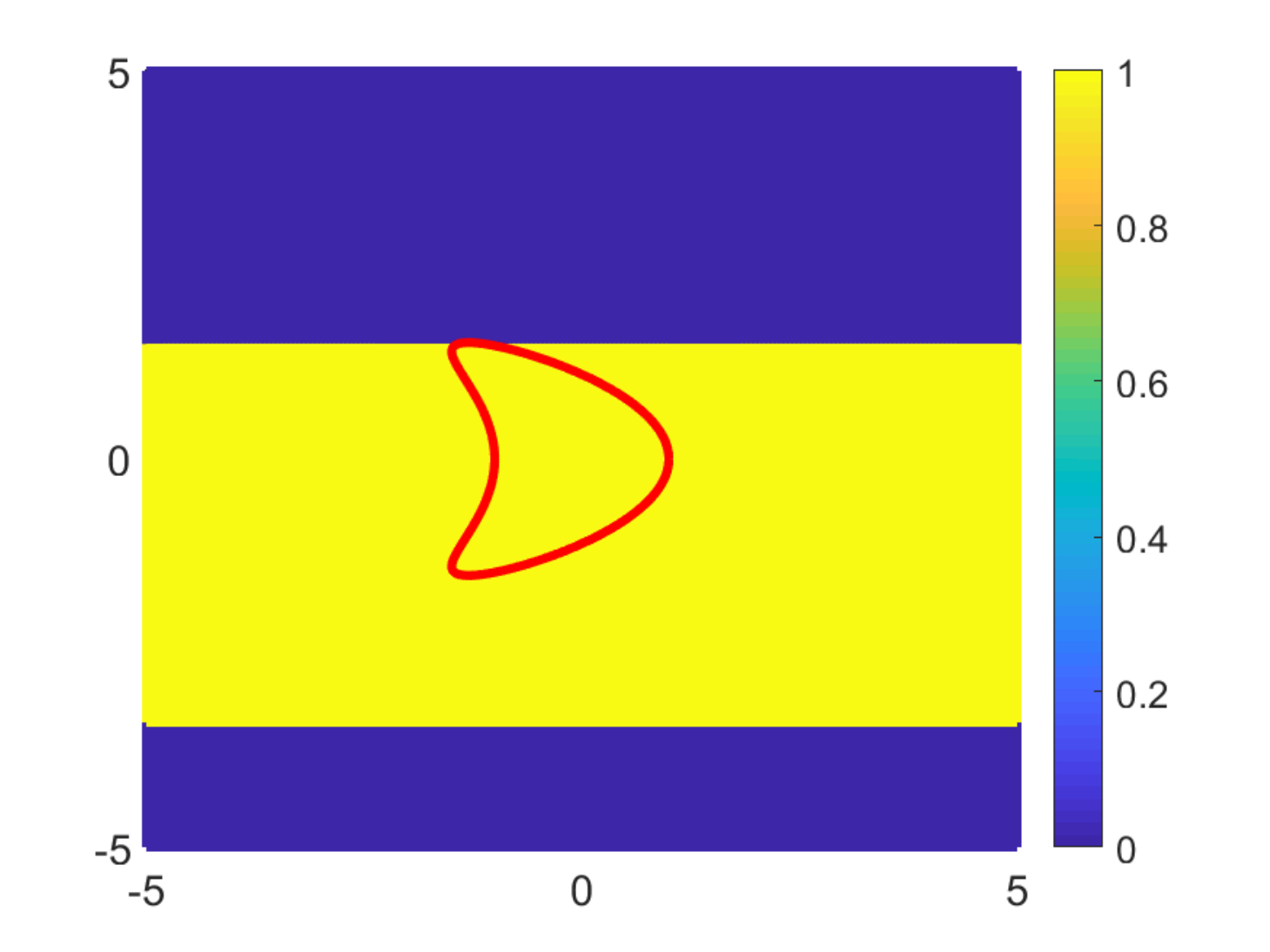}
	}
	\subfigure[ $\tilde{I}^{(\hat{x})}$, $\varepsilon=0.007$]{
		\includegraphics[scale=0.3]{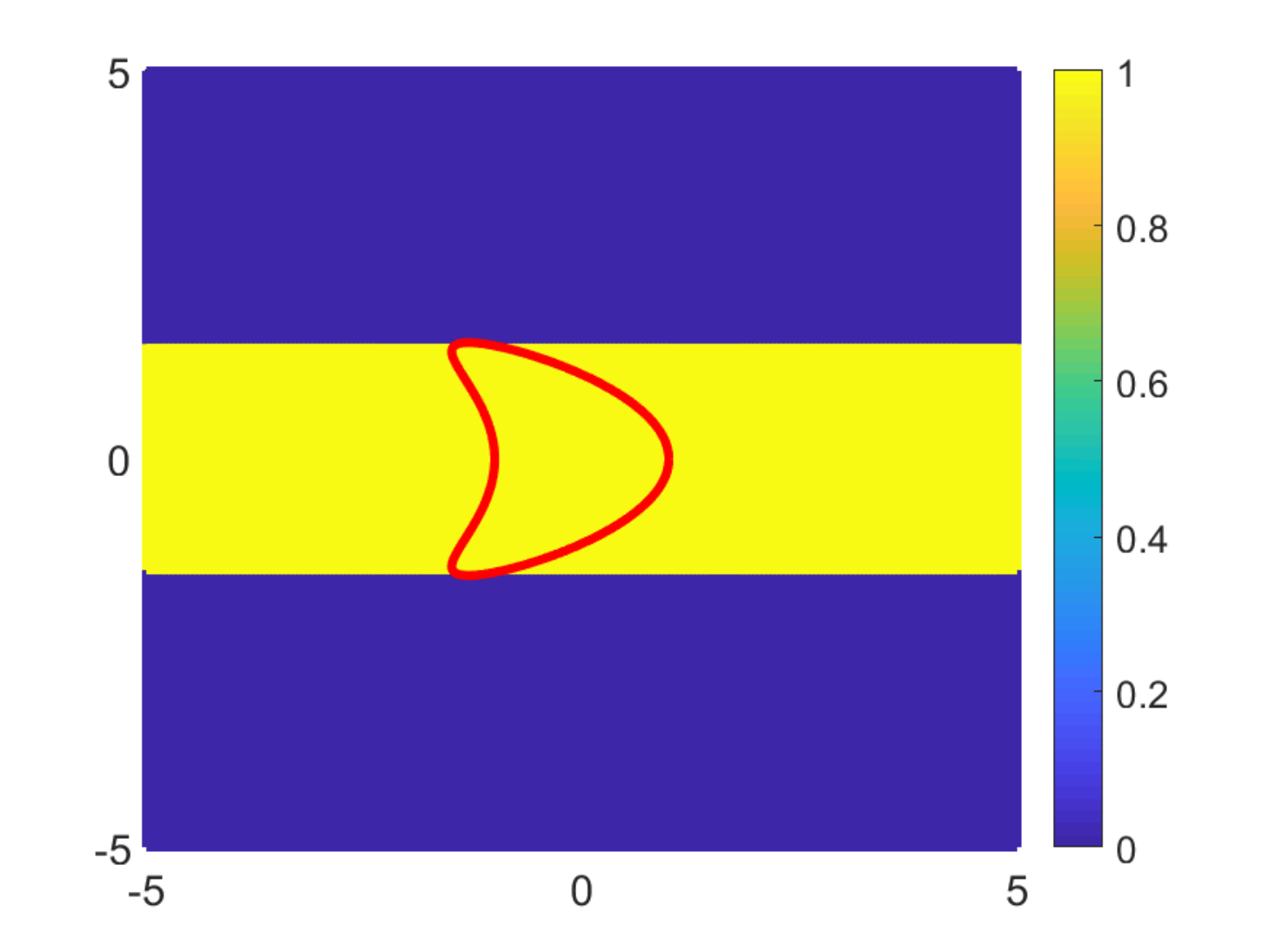}
	}

	\caption{Reconstructions of a kite-shaped support with different indicators defined as (\ref{In1}), (\ref{In2}) and (\ref{Indicator}). The Fourier transform window is $(t_{\min}, t_{\max})=(0, 2)$.	} \label{fig:1dirI}
\end{figure}

Three different source functions and three observation directions are selected in Fig. \ref{fig:1dirF}. We choose $F(x,t)=1$ and $\theta=\pi/4$ in Fig.\ref{fig:1dirF} (a); $F(x,t)=t$ and $\theta=3\pi /4$ in Fig.\ref{fig:1dirF} (b) and $F(x,t)=(x_1^2+x_2^2+10)t$ and $\theta=2\pi$ in Fig. \ref{fig:1dirF}. These figures are reconstructed by the normalized indicator $\tilde{I}^{(\hat{x})}$ and
 are not processed by a threshhold. One can conclude that strip $\mathcal{S}^{(\hat{x})}$ can be well recovered by using the multi-frequency data at a single observation direction.

\begin{figure}[H]
	\centering
	
	\subfigure[$F=1, \theta=\pi/4$]{
		\includegraphics[scale=0.3]{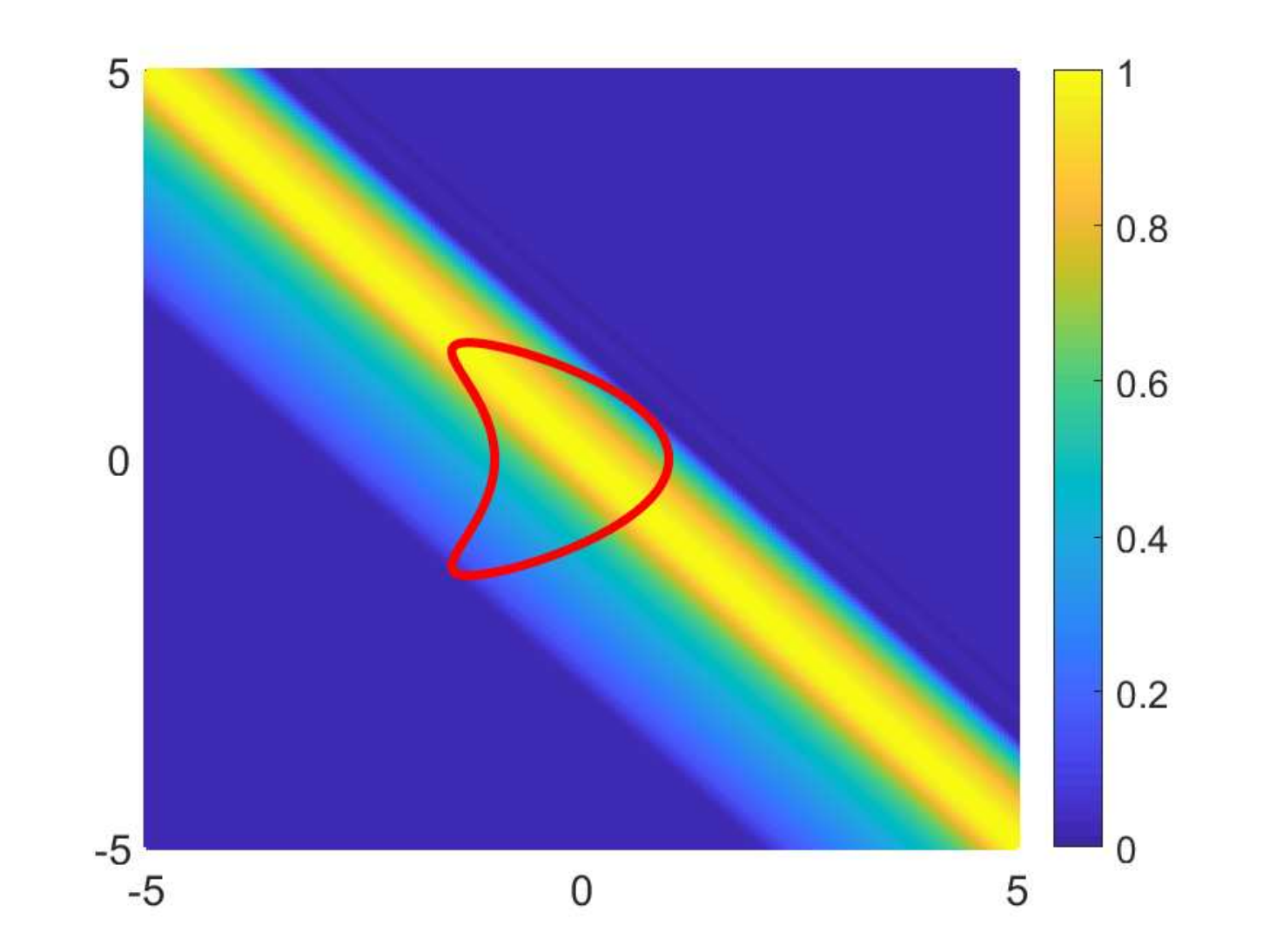}
	}
	\subfigure[$F=t, \theta=3\pi/4$ ]{
		\includegraphics[scale=0.3]{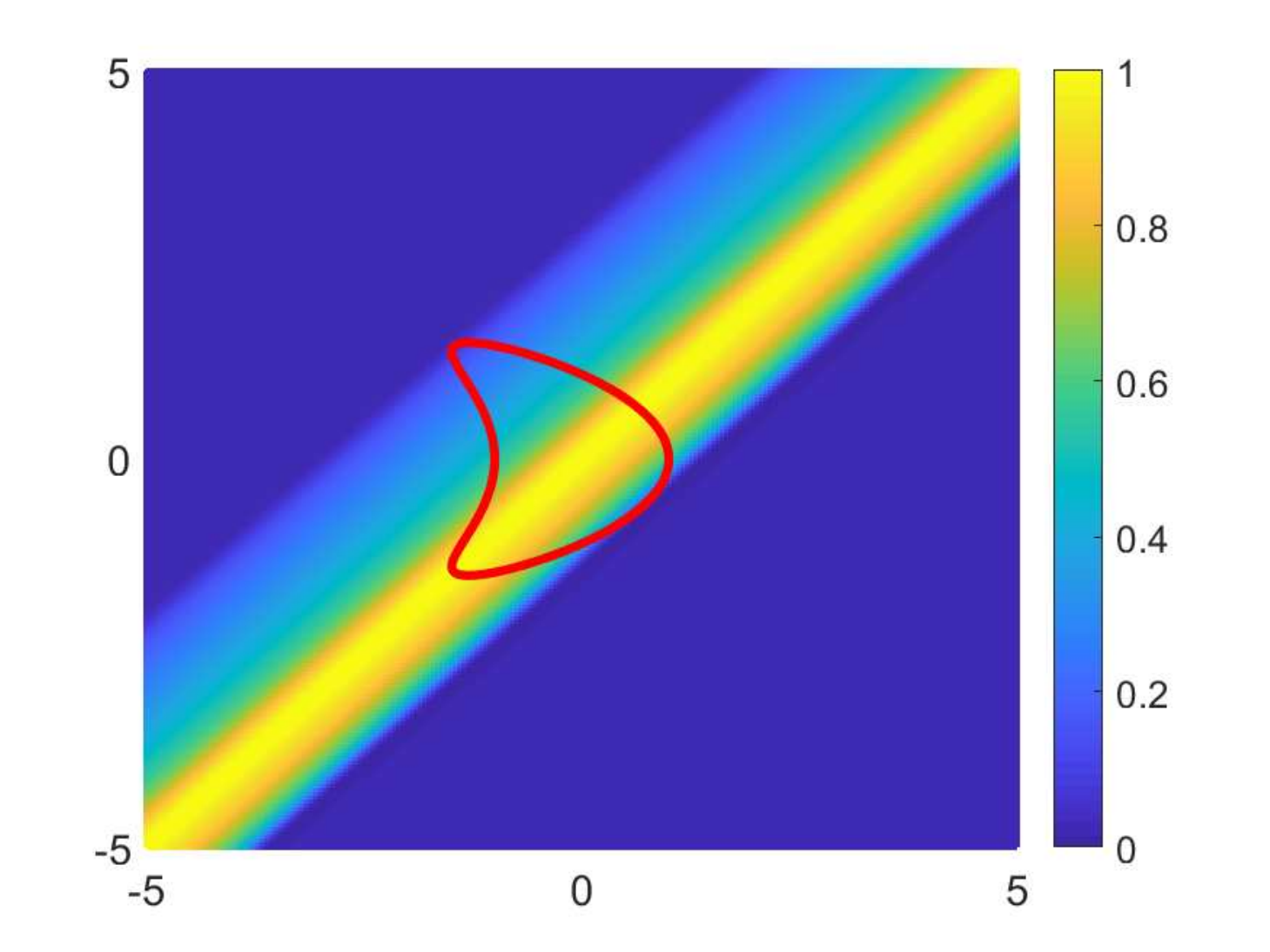}
		
	}
	\subfigure[$F=(x_1^2+x_2^2)t, \theta=2\pi$]{
		\includegraphics[scale=0.3]{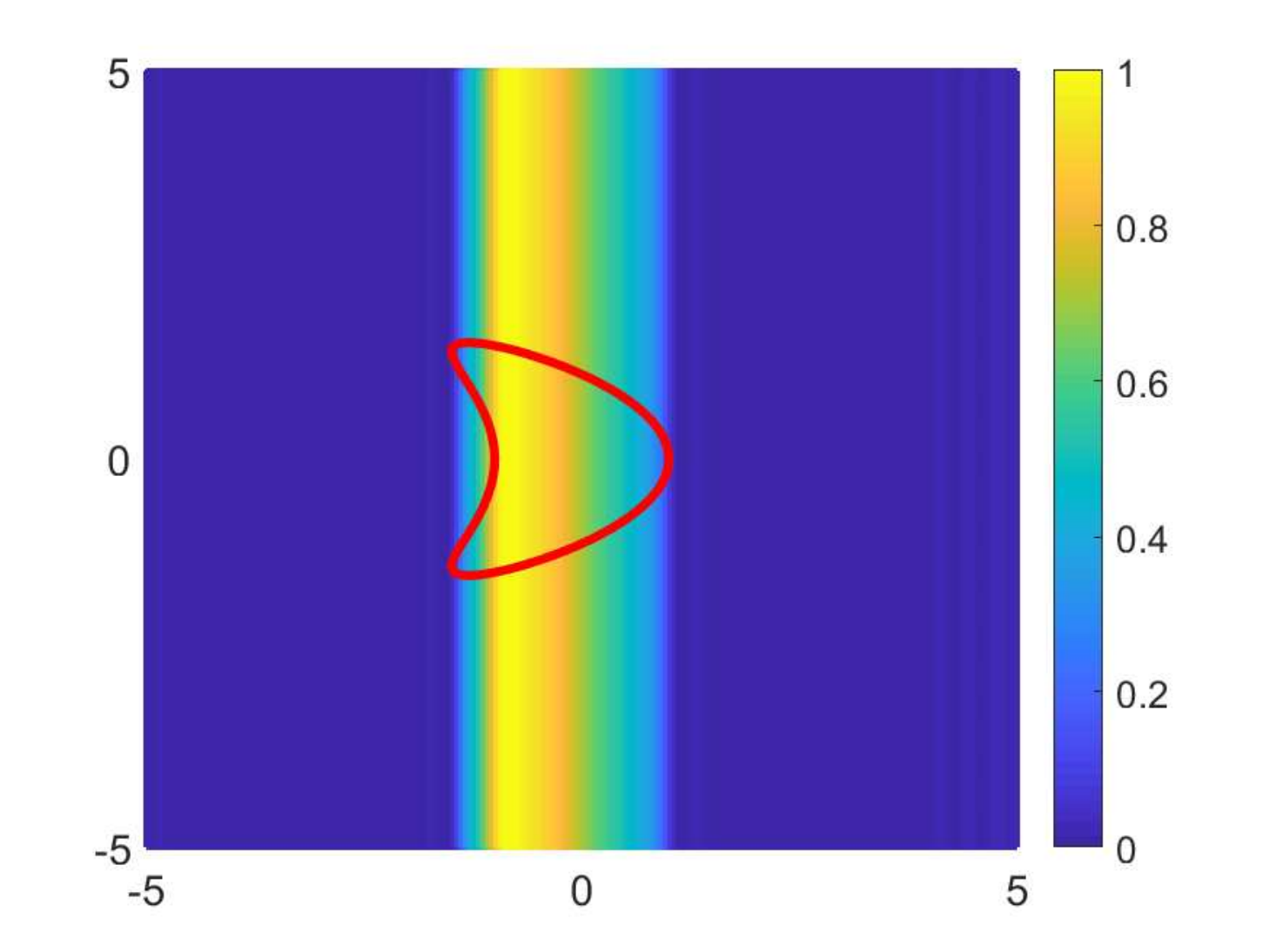}
		
	}

	\caption{Reconstructions for a kite-shaped support by using a single observation direction $\hat{x}=(\cos\theta,\sin\theta)$ and multi-frequency far-field data. The Fourier transform window is $(0, 0.1)$.
	} \label{fig:1dirF}
\end{figure}

Next, we perform numerical tests by using different Fourier transform windows $(t_{\min},\ t_{\max})$. We fix $\hat{x}=(\cos\theta,\sin\theta)$, $\theta=3\pi/4$ and fix the source function $F(x,t)=(x_1^2+x_2^2+10)t$ supported on the kite-shaped domain. 
We set $t_{\min}=0$ in Figure \ref{fig:1dirT}. Take $t_{\max}=1$, $N=200$ in Fig. \ref{fig:1dirT} (a); $t_{\max}=3$, $N=300$ in Fig. \ref{fig:1dirT}(b); $t_{\max}=8$,  $N=800$ in Fig.\ref{fig:1dirT} (c). It is obvious that the source support lies in the smallest strip perpendicular to the observation direction. Fixing $T=t_{\max}-t_{\min}=0.1$, we change $t_{\min}=1,2,3$ in Fig. \ref{fig:1dirt}.  The previous two groups of experiments have shown that the choice of the Fourier transform window $(t_{\min},t_{\max})$ does not effect the reconstructions at a single observation direction.

\begin{figure}[H]
	\centering
	\subfigure[$T=1$]{
		\includegraphics[scale=0.3]{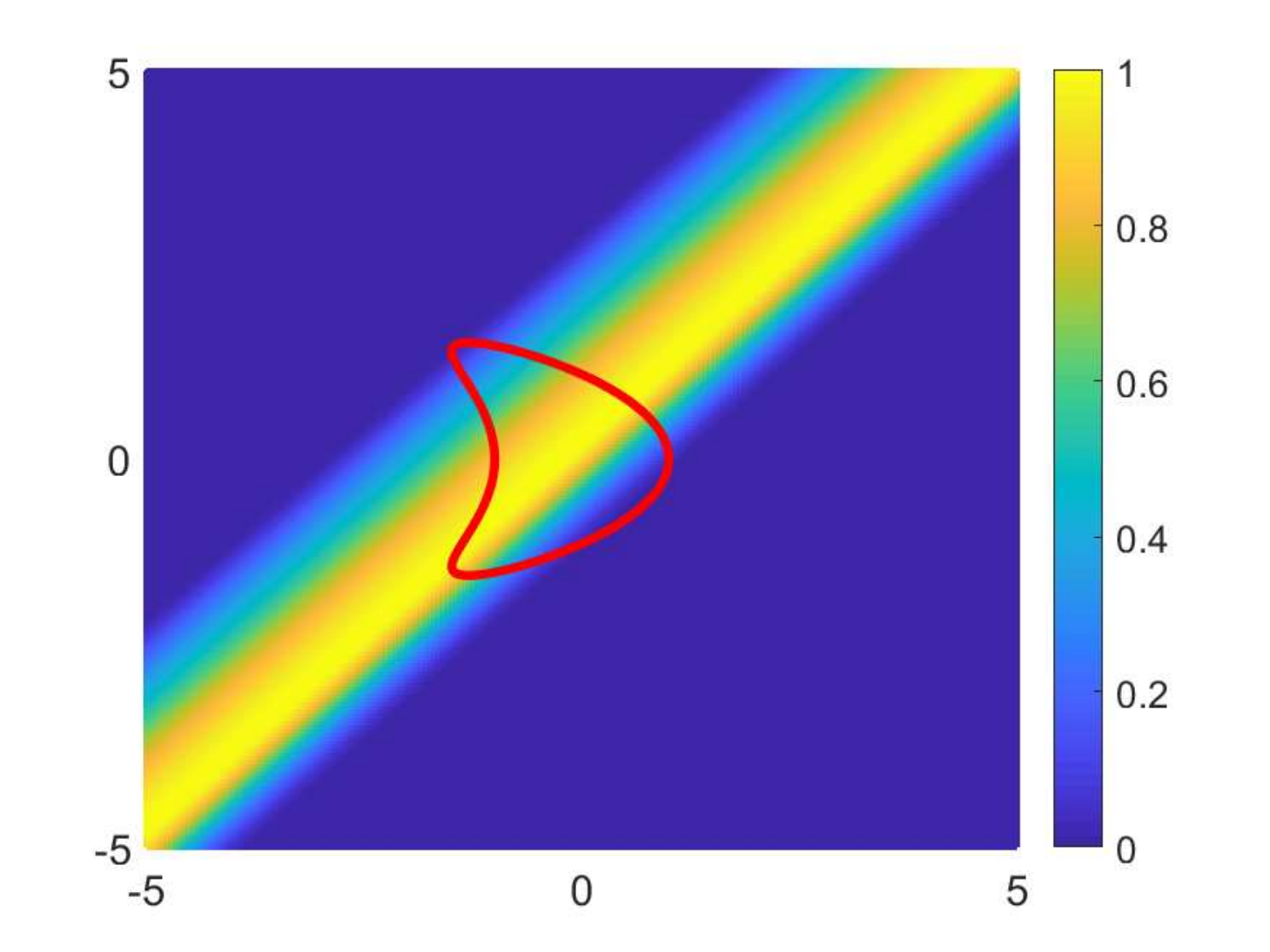}
	}
	\subfigure[$T=3$ ]{
		\includegraphics[scale=0.3]{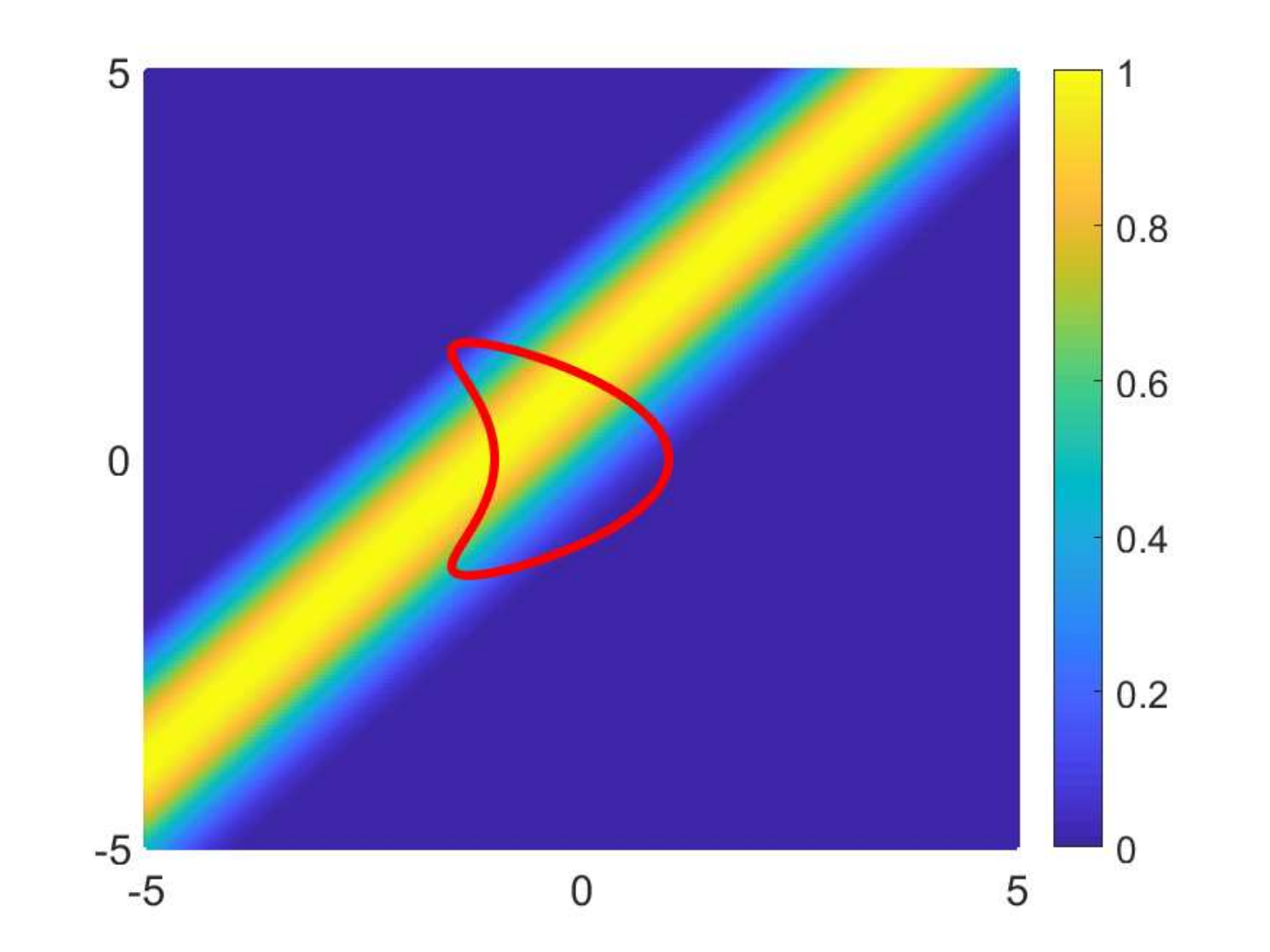}
		
	}
	\subfigure[$T=8$]{
		\includegraphics[scale=0.3]{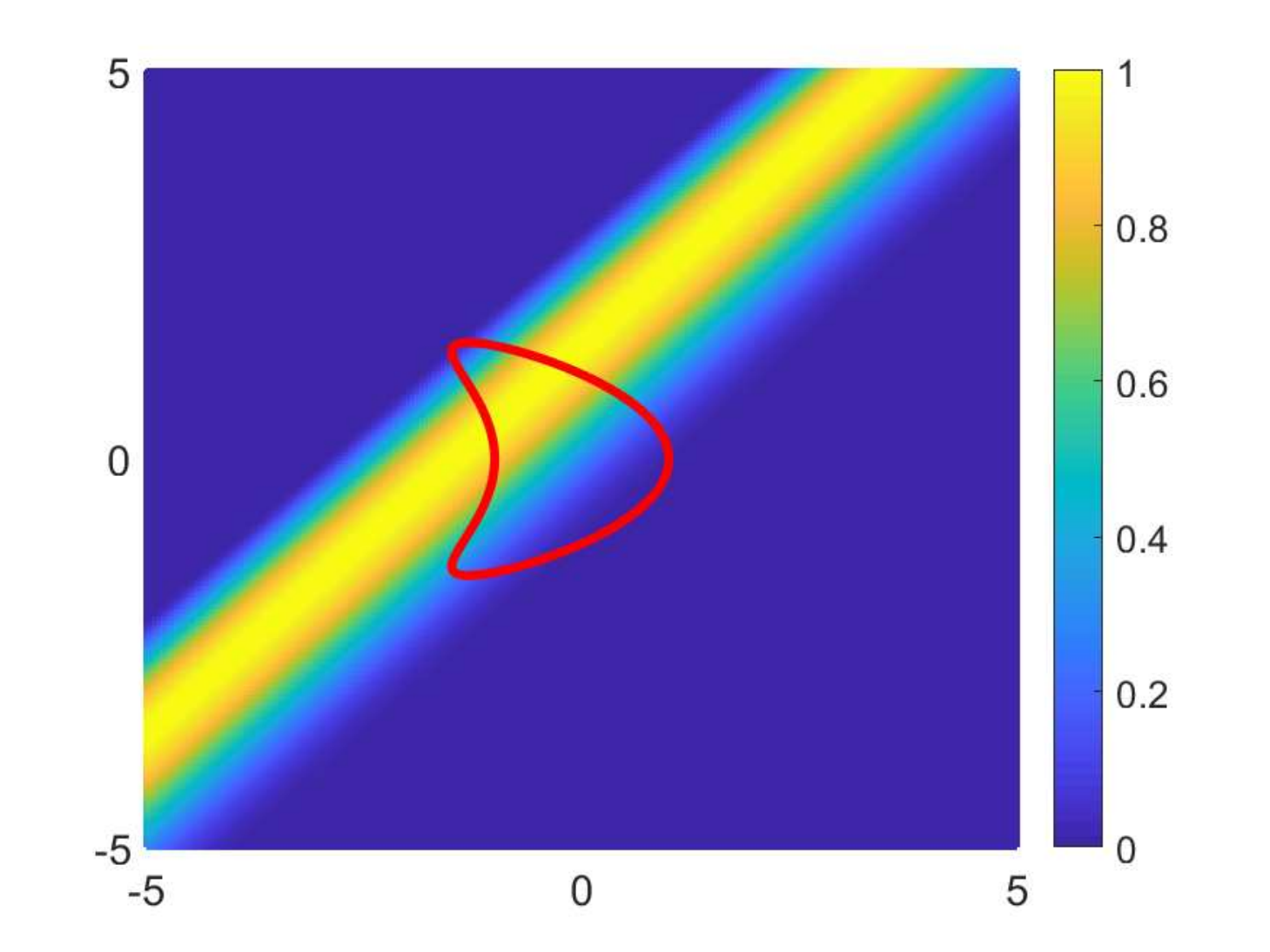}
		
	}

	\caption{Reconstructions of a kite-shaped support with $F(x,t)=(x_1^2+x_2^2+10)t$ and $\theta = 3\pi /4$ with different Fourier transform windows $(t_{\min}, t_{\max})=(0, T)$.
	} \label{fig:1dirT}
\end{figure}

Finally, we consider a support with two disconnected components, a kite and an ellipse, which satisfy the assumption(A) for some observation direction $\hat{x}$. In Fig. \ref{fig:1dir2D} (a), the observation angle is taken as $\theta=\pi/4$, and the strip goes along the direction of $3\pi/4$ , separating the two supports precisely. The kite-ellipse-shaped supports lie within the strips obtained from the observation angle $3\pi/4$, which are visualized in Fig.\ref{fig:1dir2D} (b). Since the distance assumption(A) is not satisfied at the angle $\theta=3\pi/4$, the two supports are not separated but lie within the strip perpendicular to the observation direction. We set the observation angle to $\theta=2\pi$ in Fig.\ref{fig:1dir2D} (c). The recovered strips along the direction $(1, 0)$ precisely separates the two components.

\begin{figure}[H]
	\centering
	\subfigure[$t_{\min}=1$, $T=0.1$]{
		\includegraphics[scale=0.3]{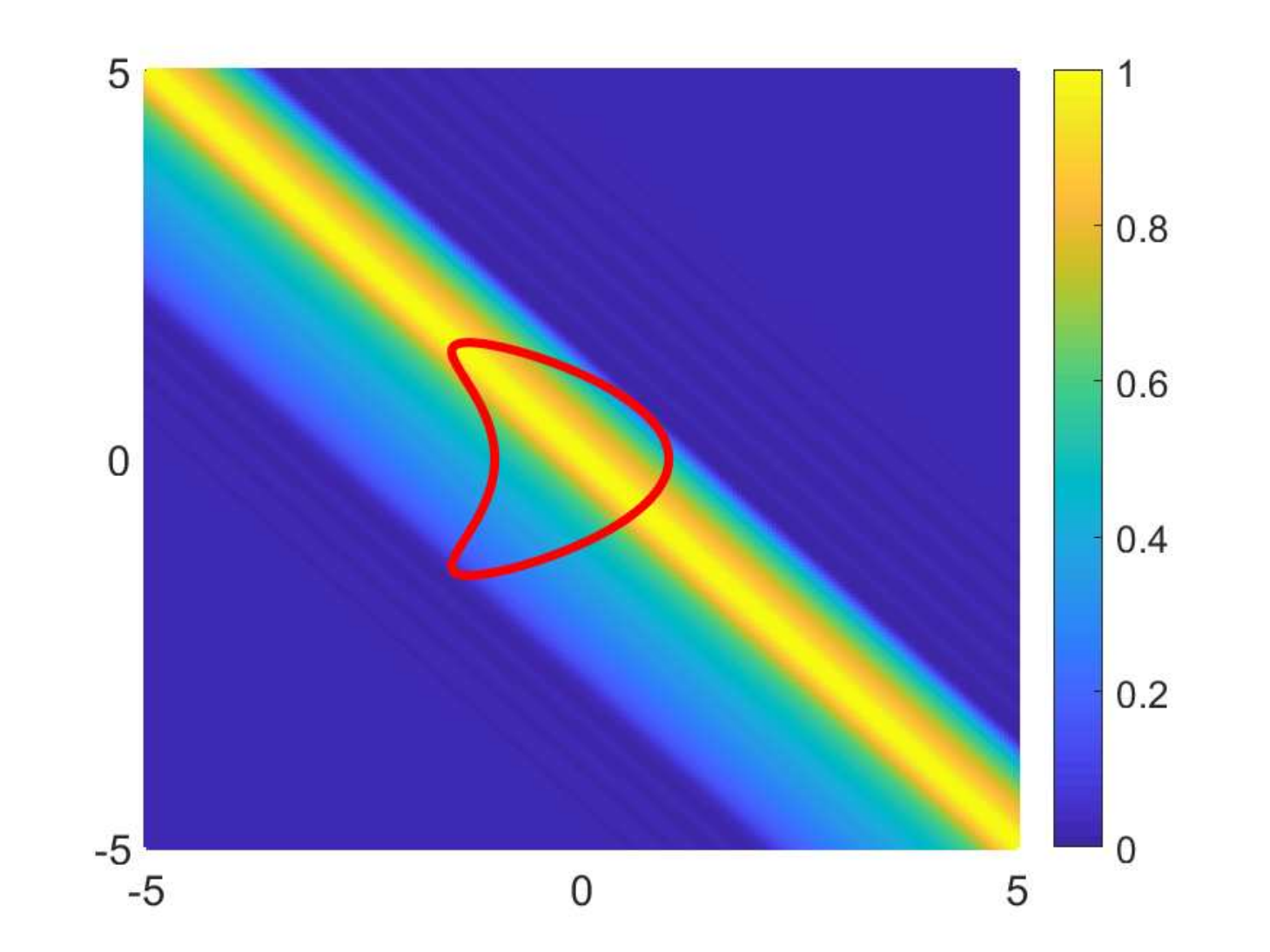}
	}
	\subfigure[$t_{\min}=2$, $T=0.1$ ]{
		\includegraphics[scale=0.3]{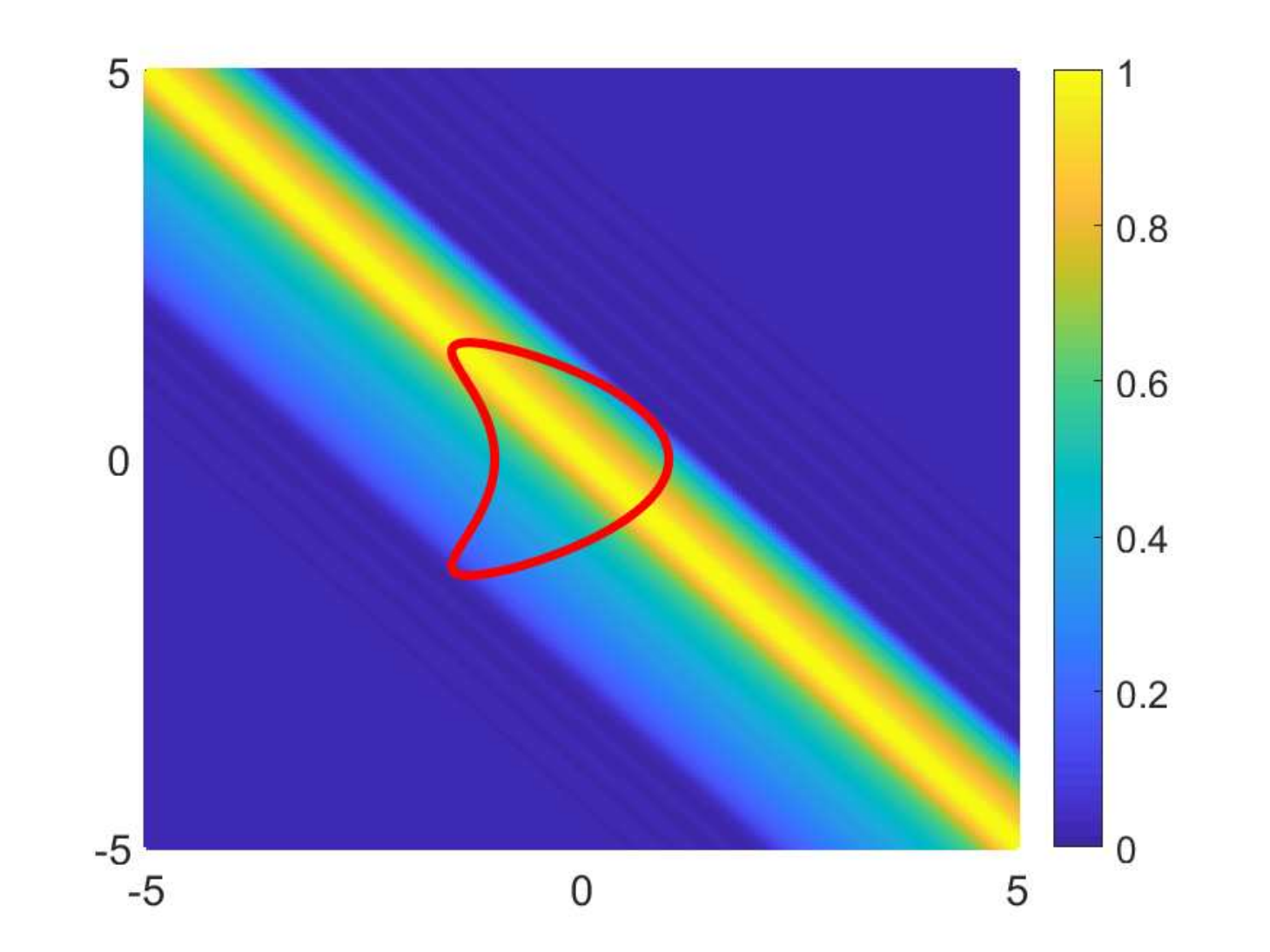}
		
	}
	\subfigure[$t_{\min}=3$, $T=0.1$]{
		\includegraphics[scale=0.3]{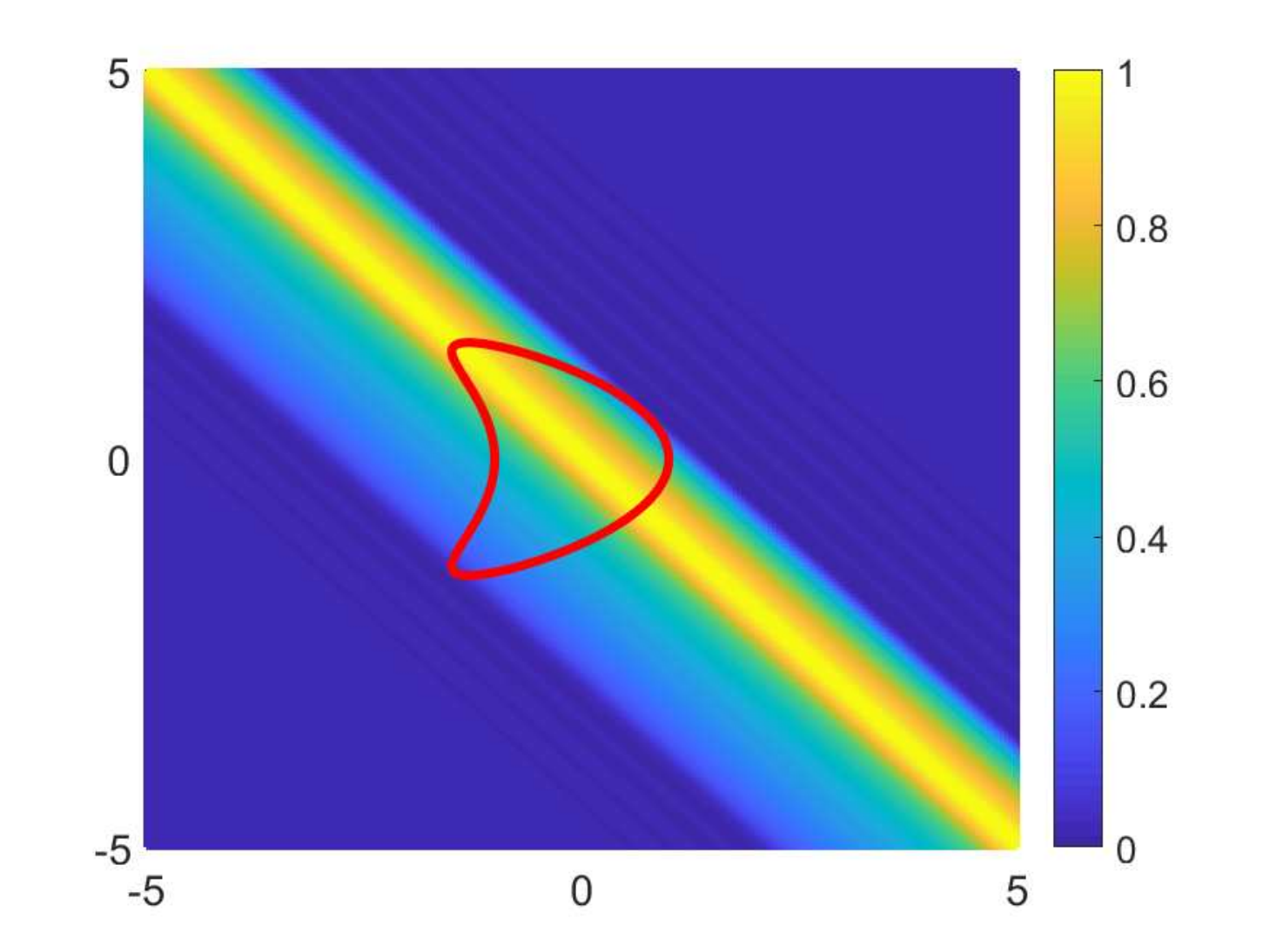}
		
	}

	\caption{Reconstructions of a kite-shaped support with $F(x,t)=(x_1^2+x_2^2+10)t$ and $\theta = \pi /4$ with different Fourier transform windows $(t_{\min}, t_{\min}+T)$.
	} \label{fig:1dirt}
\end{figure}

\begin{figure}[H]
	\centering
	\subfigure[$\theta=\pi/4$]{
		\includegraphics[scale=0.3]{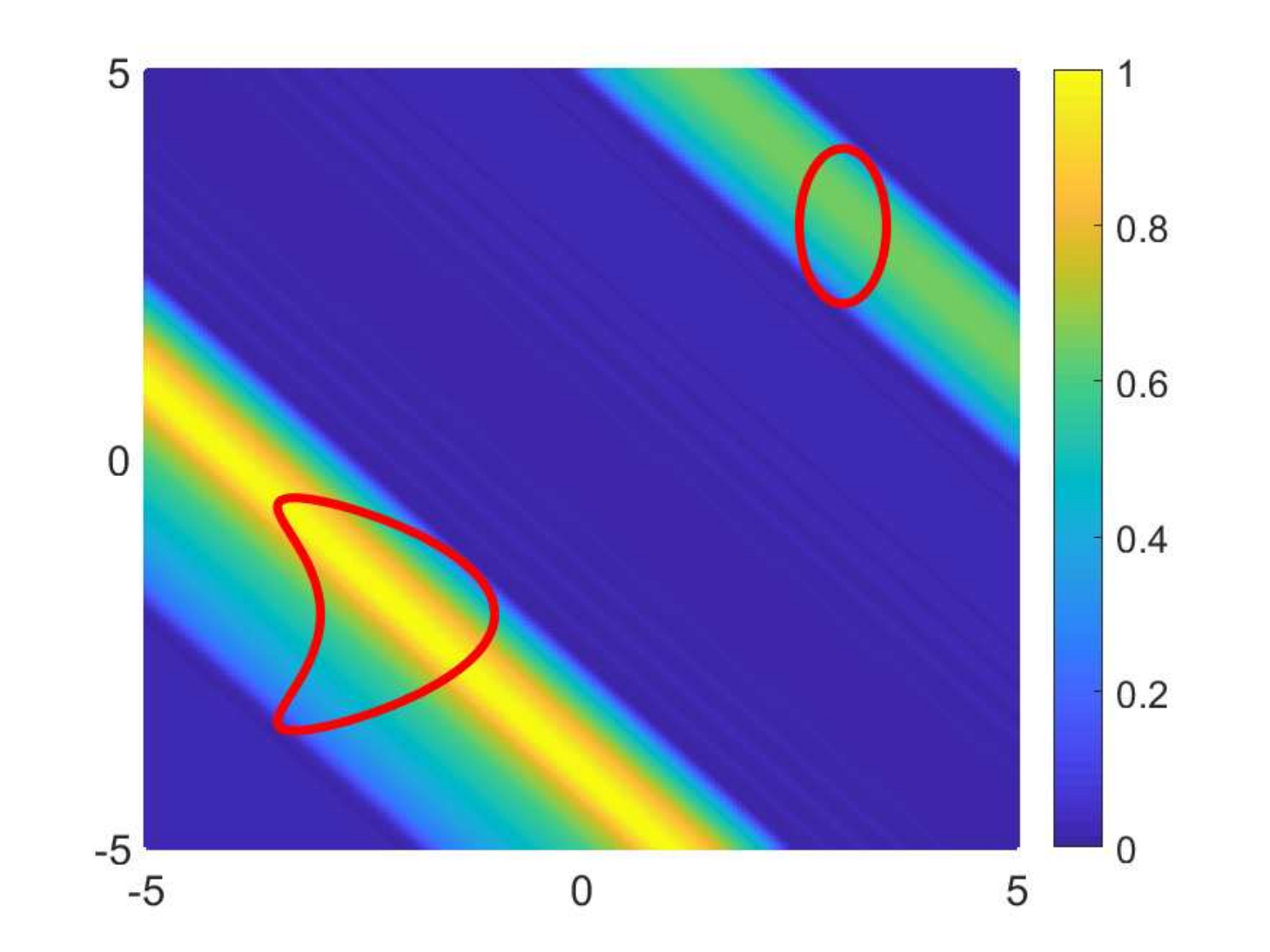}
	}
	\subfigure[$\theta=3\pi/4$]{
		\includegraphics[scale=0.3]{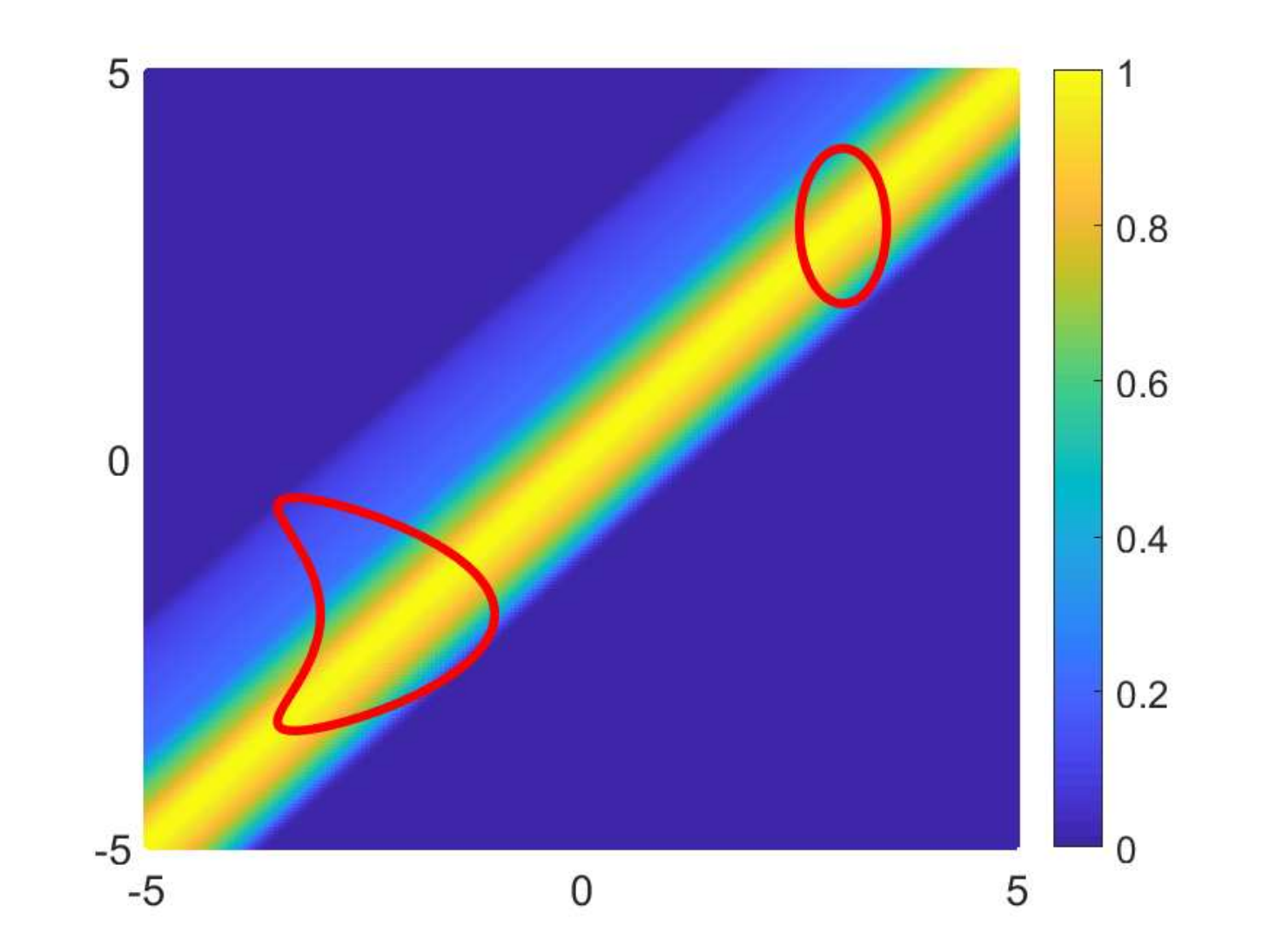}
		
	}
	\subfigure[$\theta=2\pi$]{
		\includegraphics[scale=0.3]{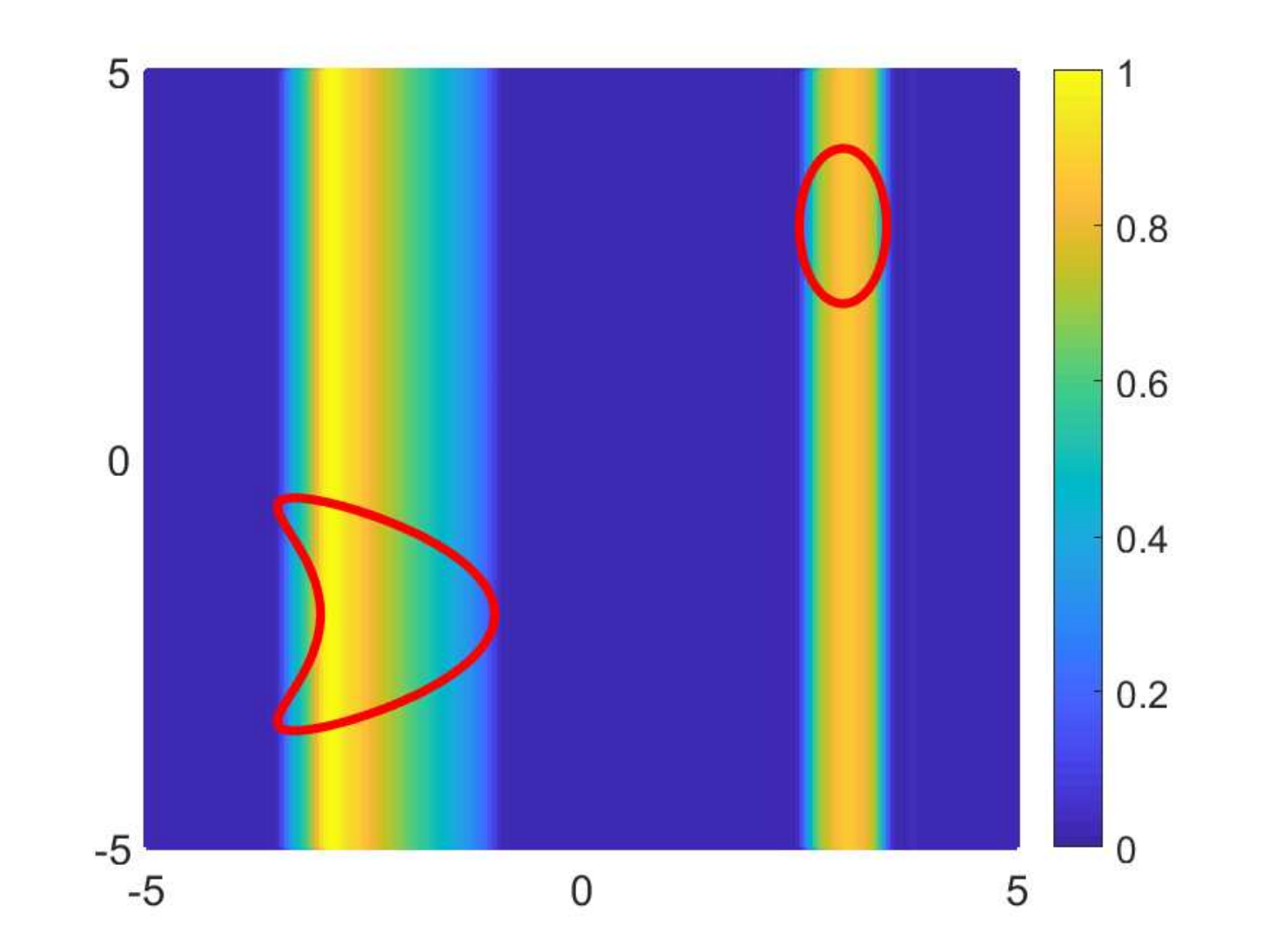}
		
	}

	\caption{Reconstructions of a kite-ellipse-shaped support from different observation angles. Here $F(x,t)=(x_1^2+x_2^2+10)t$ and the Fourier transform window is taken as $(0,0.1)$.
	} \label{fig:1dir2D}
\end{figure}

\subsection{Reconstructions from sparse observation directions}

In this subsection, we test the performance of the multi-frequency sampling method  with $M$ observation directions $\hat{x}_m=(\cos{\theta_m},\ \sin{\theta_m})$, $\theta_m=\dfrac{m-1}{M}\pi,\ m=1,2,\cdots,M $. The experimental results for reconstructing a kite-shaped source support are shown in Fig. \ref{fig:Mdir}, where $F(x,t)=(x_1^2+x_2^2+10)t$. Evidently, the recovery quality has been improved as the number of observation angles increases in Fig.\ref{fig:Mdir}. In the case $M =8$, the shape of the kite can be well restored.

\begin{figure}[H]
	\centering
	\subfigure[$M=2$]{
		\includegraphics[scale=0.3]{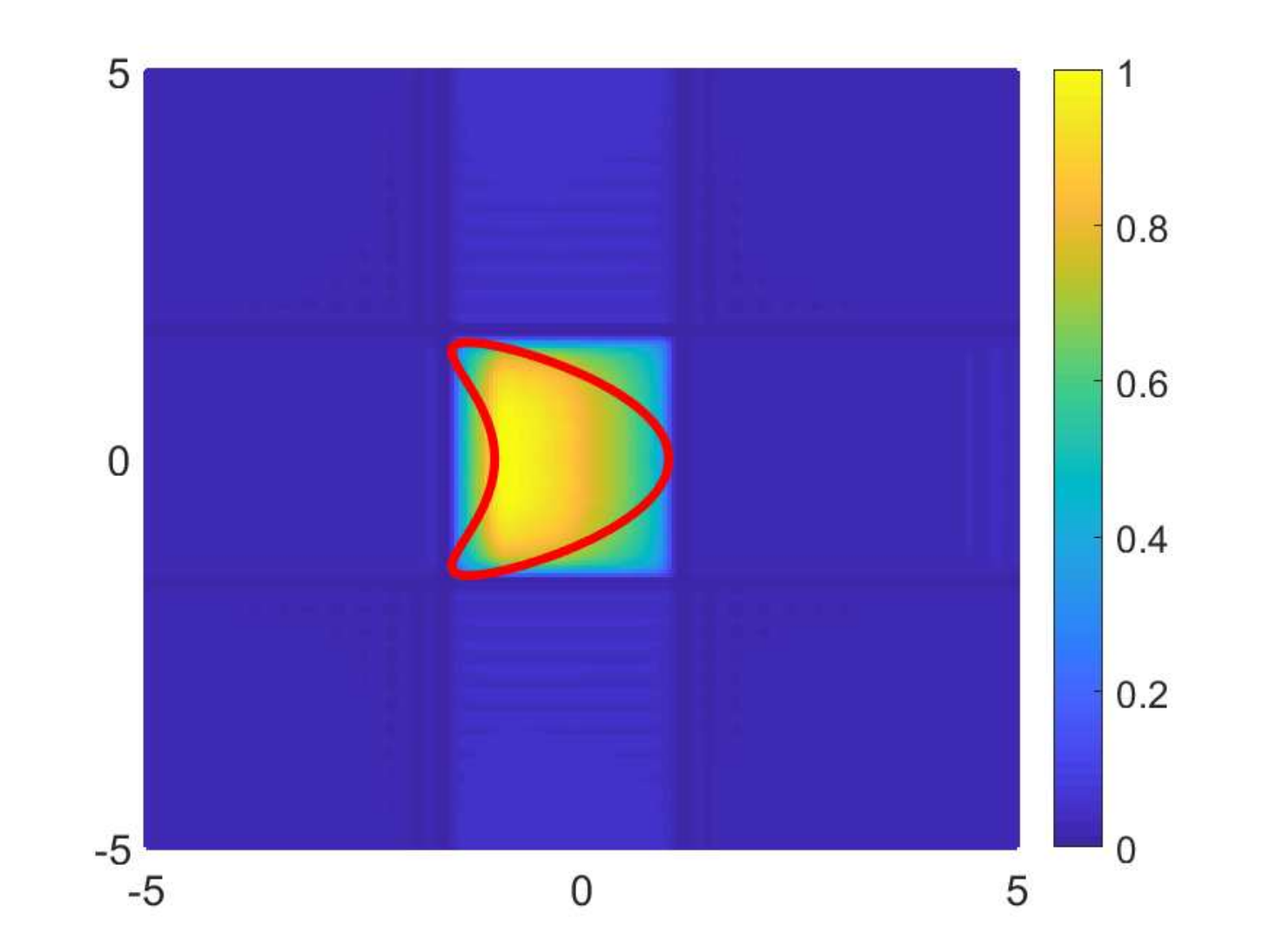}
	}
	\subfigure[$M=4$ ]{
		\includegraphics[scale=0.3]{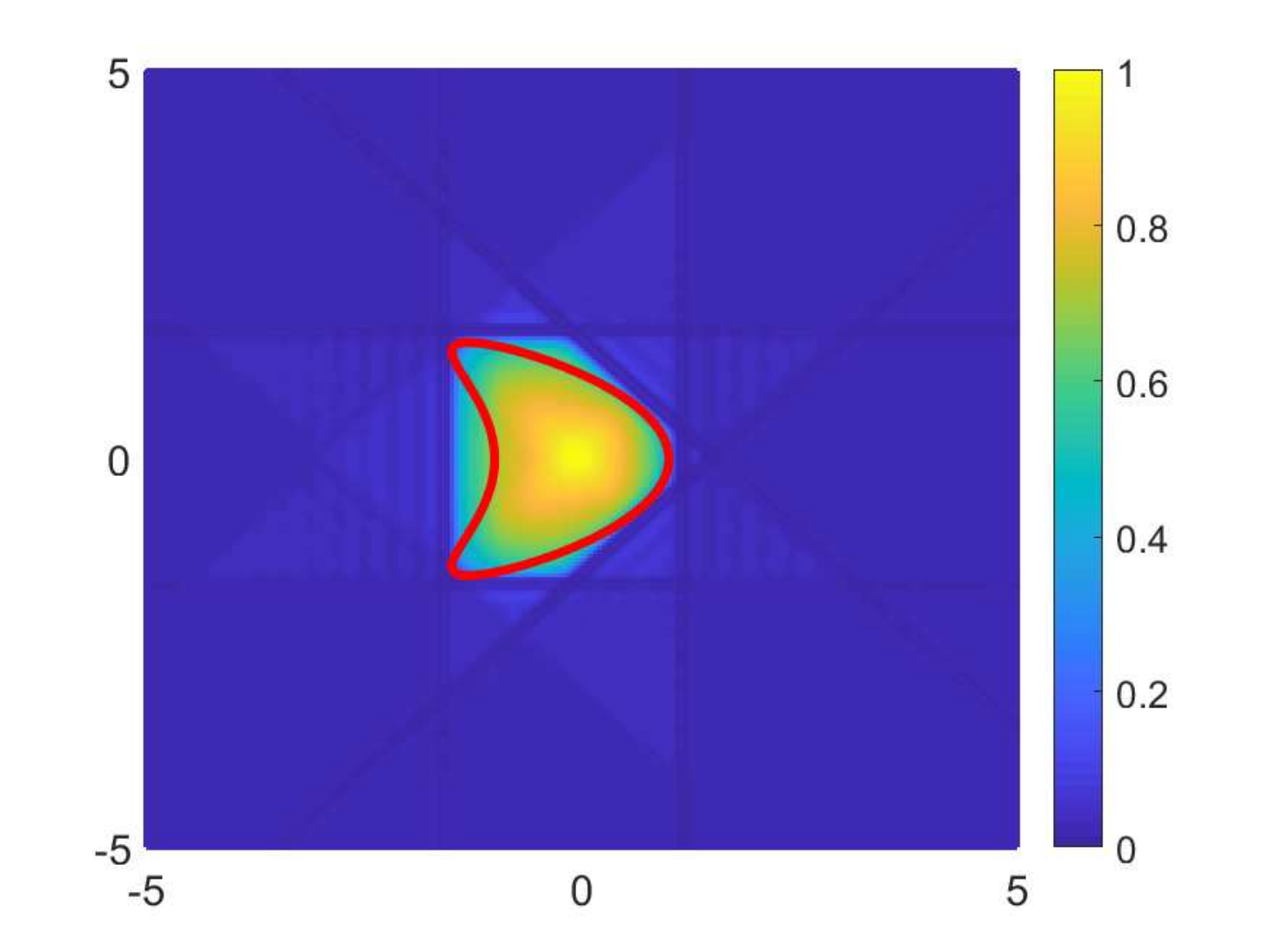}
		
	}
	\subfigure[$M=8$]{
		\includegraphics[scale=0.3]{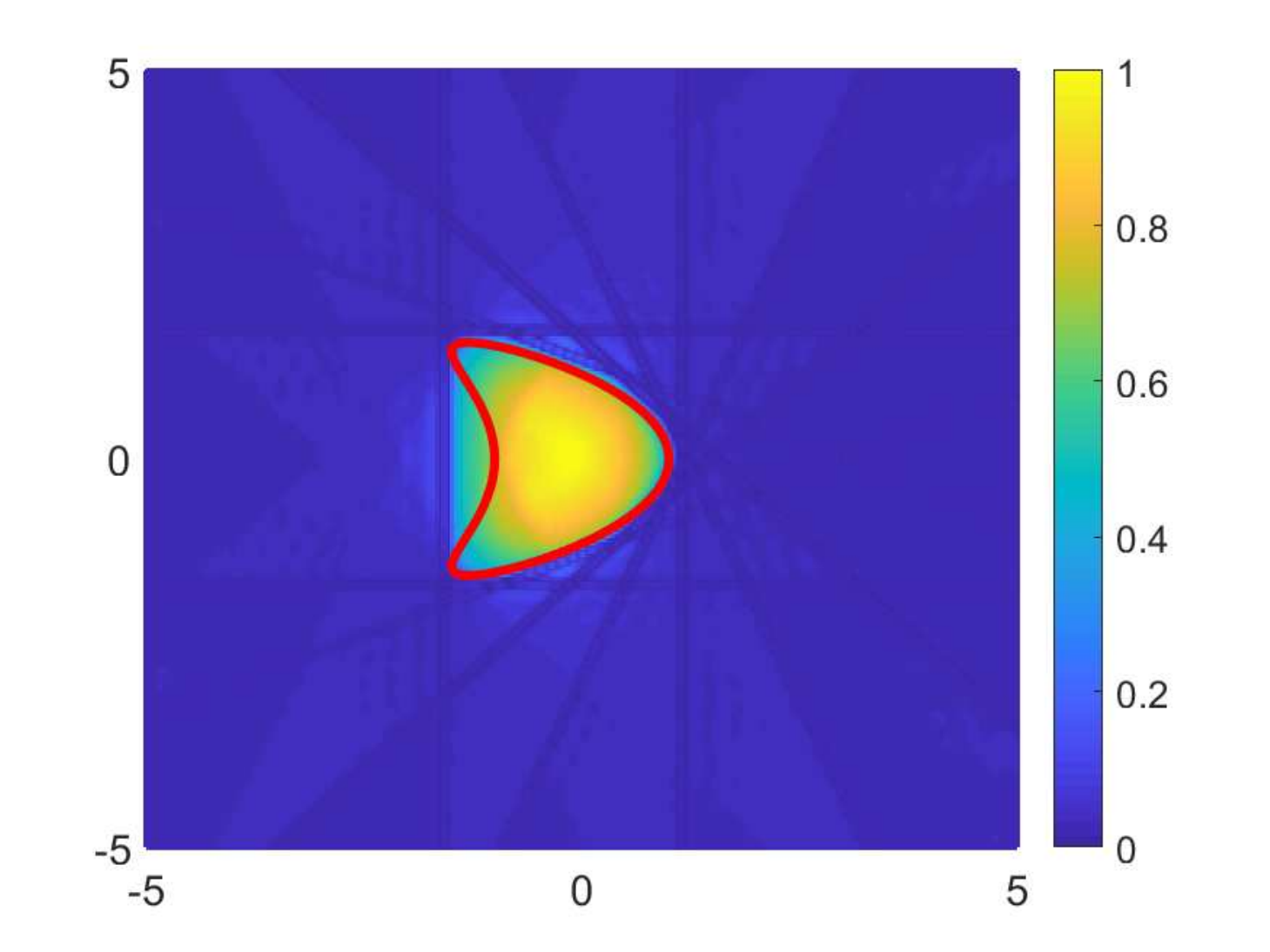}
		
	}
	
	\caption{Reconstructions of two disconnected supports with the source function $F(x,t)=(x_1^2+x_2^2+10)t$, the Fourier transform window $(0, 0.1)$ and using $M$ directions.
	} \label{fig:Mdir}
\end{figure}

As done in the single observation case, we also illustrate the reconstructions with different radiating periods $(0,T)$ of the Fourier transformation in Fig.\ref{fig:MdirT}. Fixing $t_{\min}=0$ and $K=20$, we set $T=t_{\max}-t_{\min}=1$ and $N=200$ in Fig.\ref{fig:MdirT} (a); $T=3$ and $N=200$ in Fig.\ref{fig:MdirT} (b); $T=5$ and $N=300$ in Fig.\ref{fig:MdirT} (c). 
Even for large $T$ we observe that the reconstructed $\Theta$-convex hull reflects the location and shape of the support.

In Fig.\ref{fig:Mdir2}, we focus on the case of two disconnected supports. The kite-ellipse-shaped domain can be restored pretty well by using $16$ observations. We set the source function $F(x,t)=(x_1+10)t$ supported on the  kite centered at $(0,-2)$ and the ellipse centered at $(0,3)$ in Fig.\ref{fig:Mdir2}(a); $F(x,t)=(x_1^2+x_2^2+10)t$ supported on the kite centered at $(-2,-2)$ and the ellipse centered at $(3,3)$ in Fig.\ref{fig:Mdir2}(b); $F(x,t)=(x_2+10)t$ with the kite-center at $(-2,0)$ and the ellipse center at $(3,0)$ in Fig.\ref{fig:Mdir2}(c). These figures are truncated by a threshold $\varepsilon=0.30$ in Figures \ref{fig:Mdir2}(d), (e) and (f).

\begin{figure}[H]
	\centering
	\subfigure[$T=1$]{
		\includegraphics[scale=0.3]{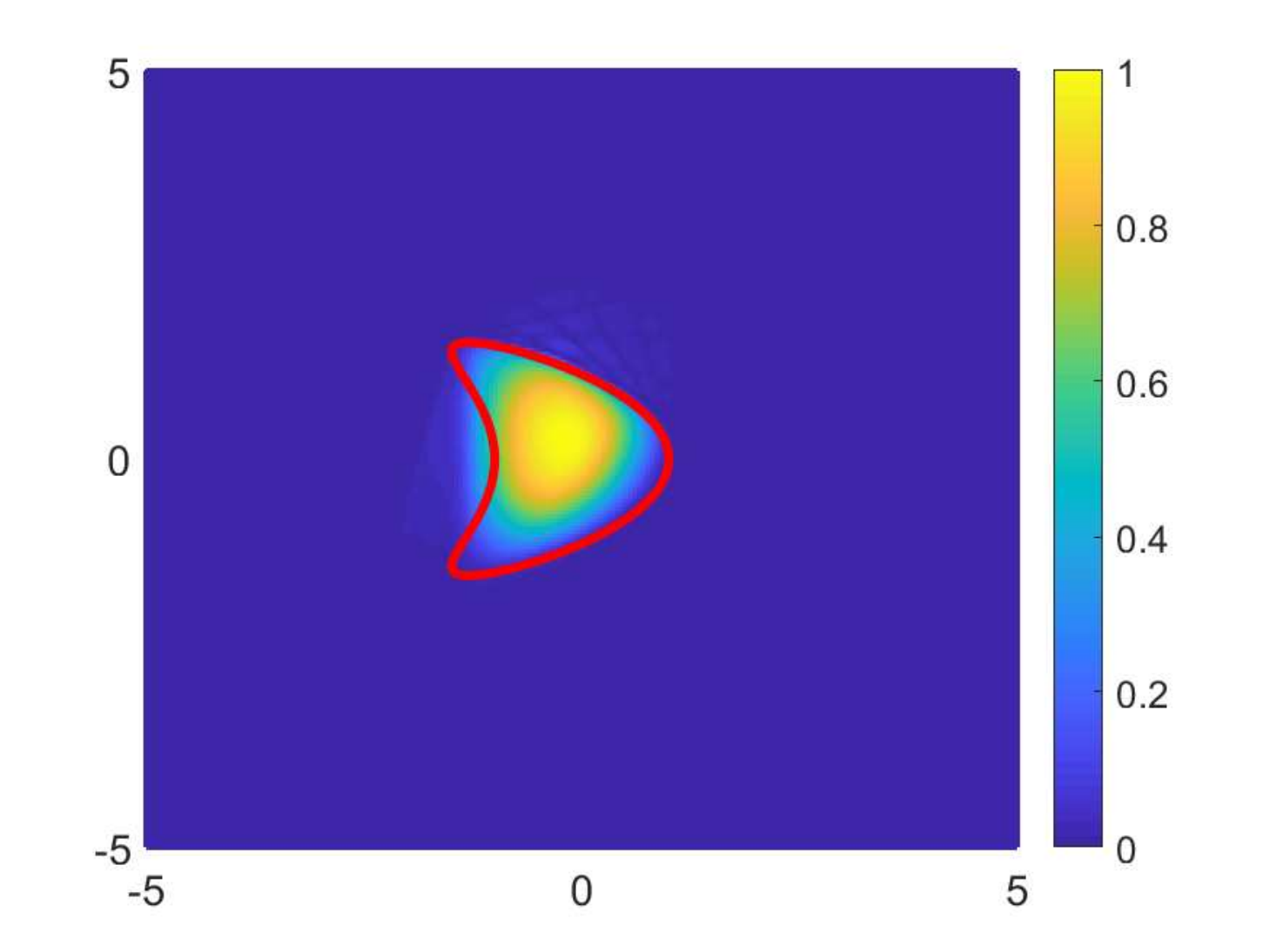}
	}
	\subfigure[$T=3$ ]{
		\includegraphics[scale=0.3]{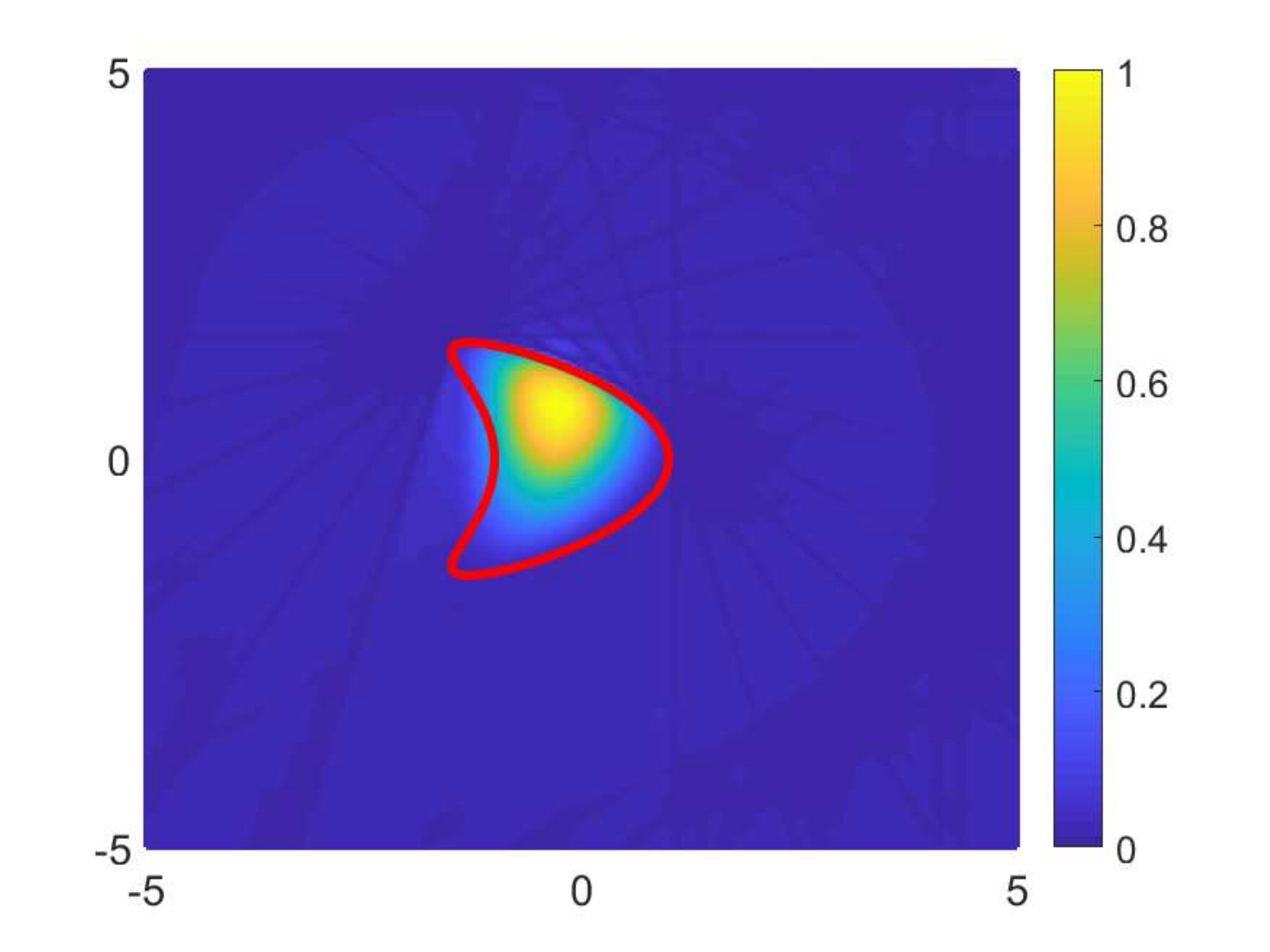}
		
	}
	\subfigure[$T=5$]{
		\includegraphics[scale=0.3]{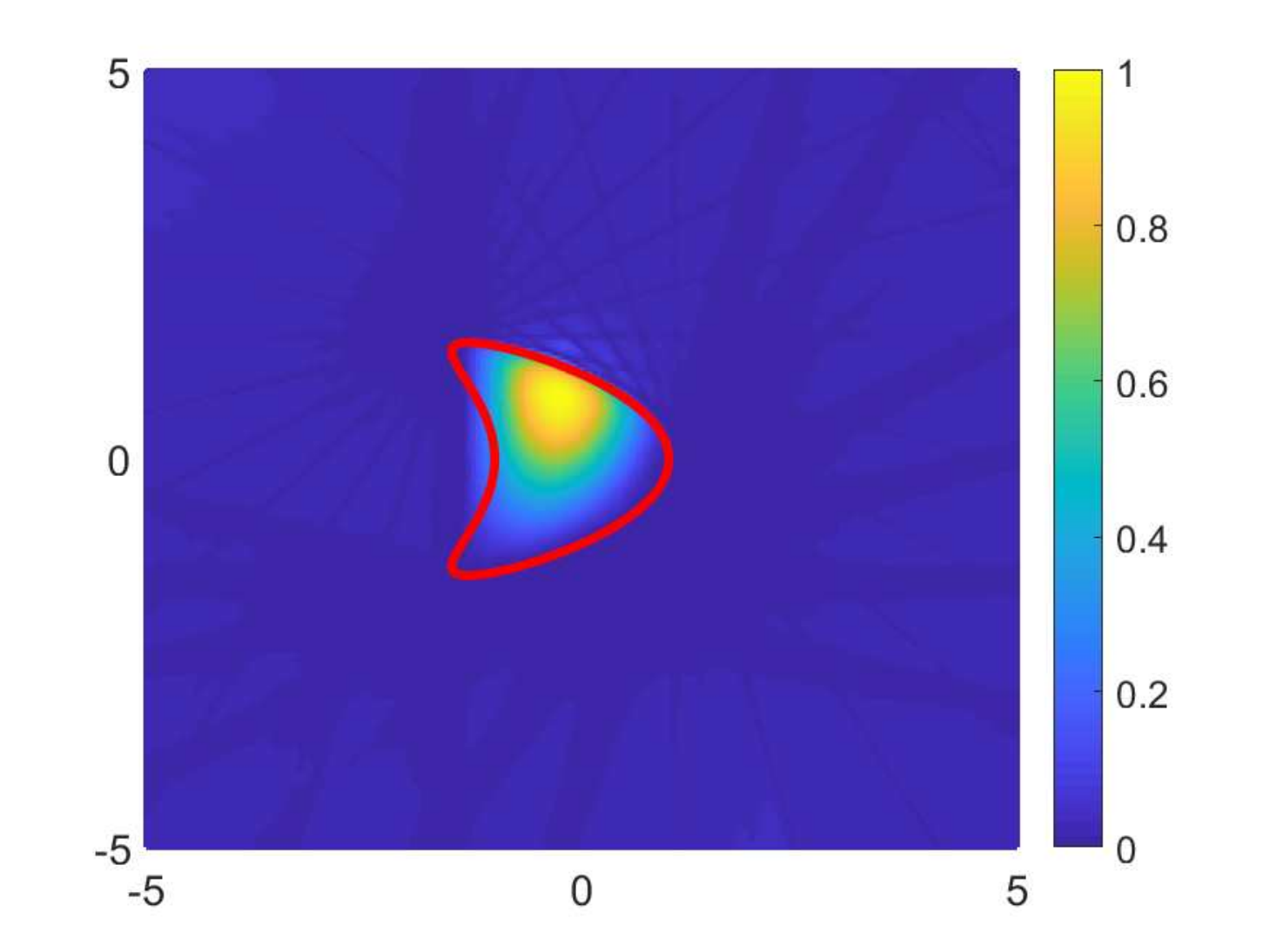}
		
	}

	\caption{Reconstructions of a kite-shaped support source $F(x,t)=(x_1^2+x_2^2+10)t$ with different Fourier transform windows $(0, T)$ by using $M=12$ observations.
	} \label{fig:MdirT}
\end{figure}

\begin{figure}[H]
	\centering
	\subfigure[$F=(x_1+10)t$]{
		\includegraphics[scale=0.3]{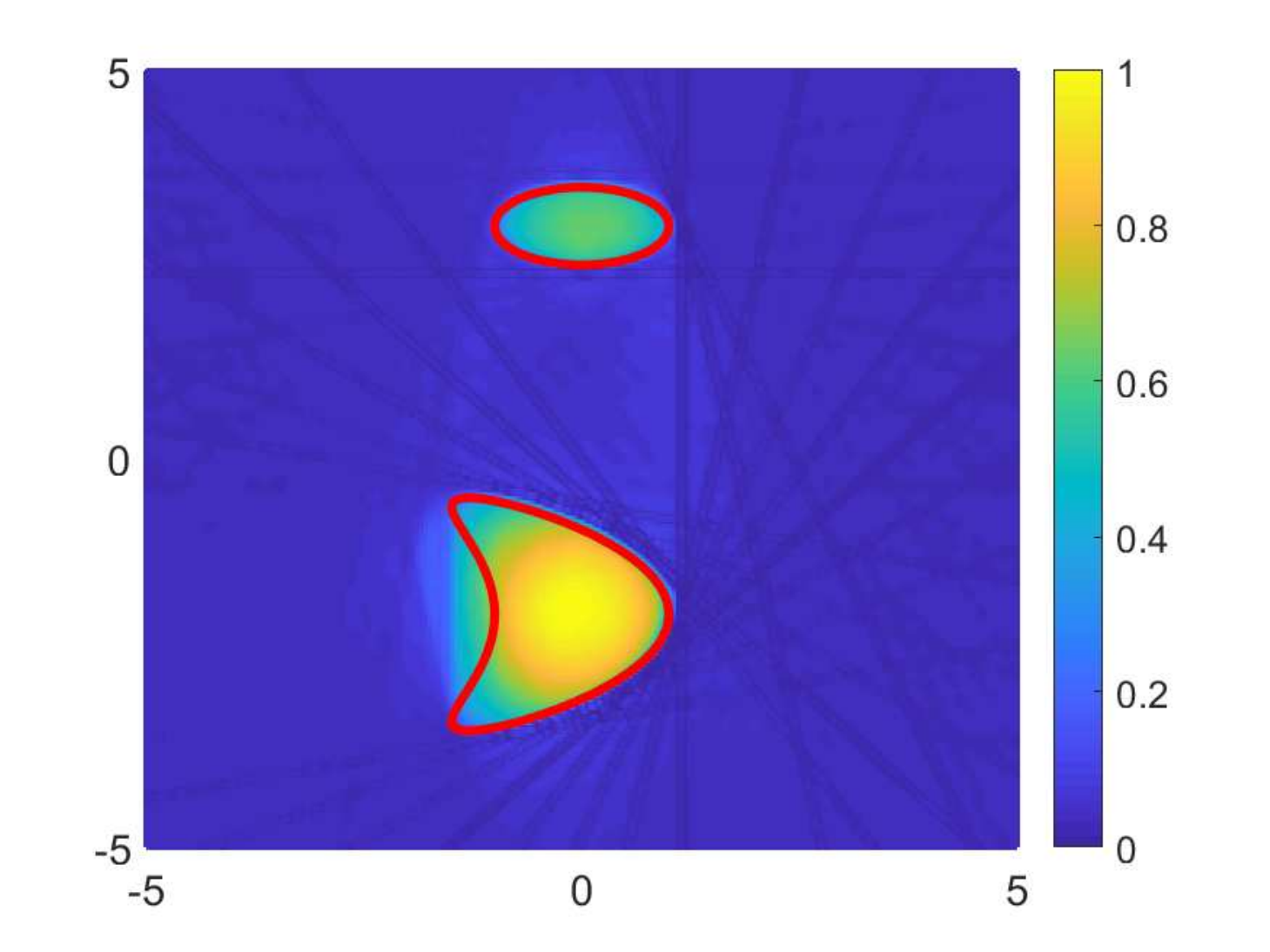}
	}
	\subfigure[$F=(x_1^2+x_2^2+10)t$ ]{
		\includegraphics[scale=0.3]{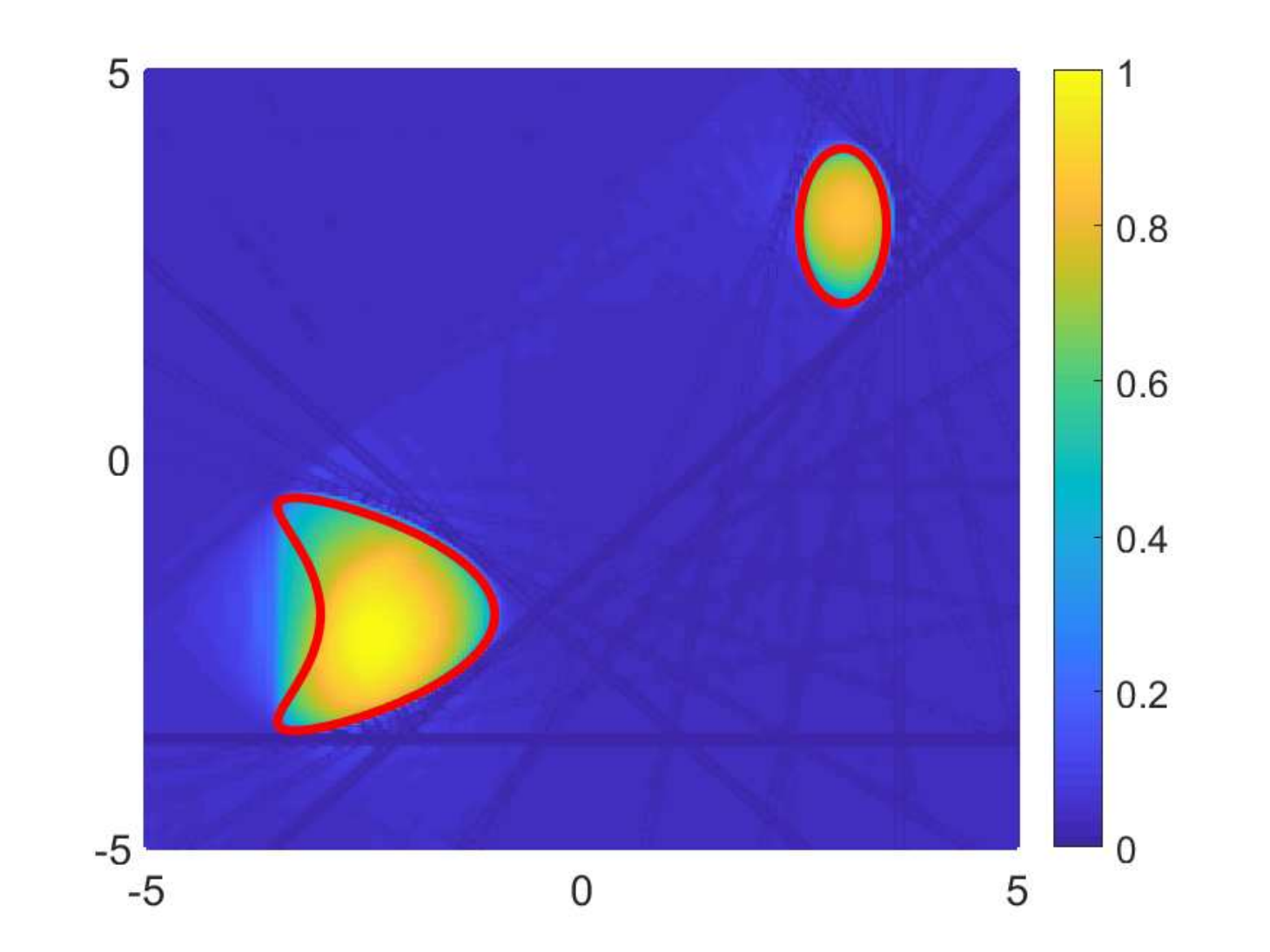}
		
	}
	\subfigure[$F=(x_2+10)t$]{
		\includegraphics[scale=0.3]{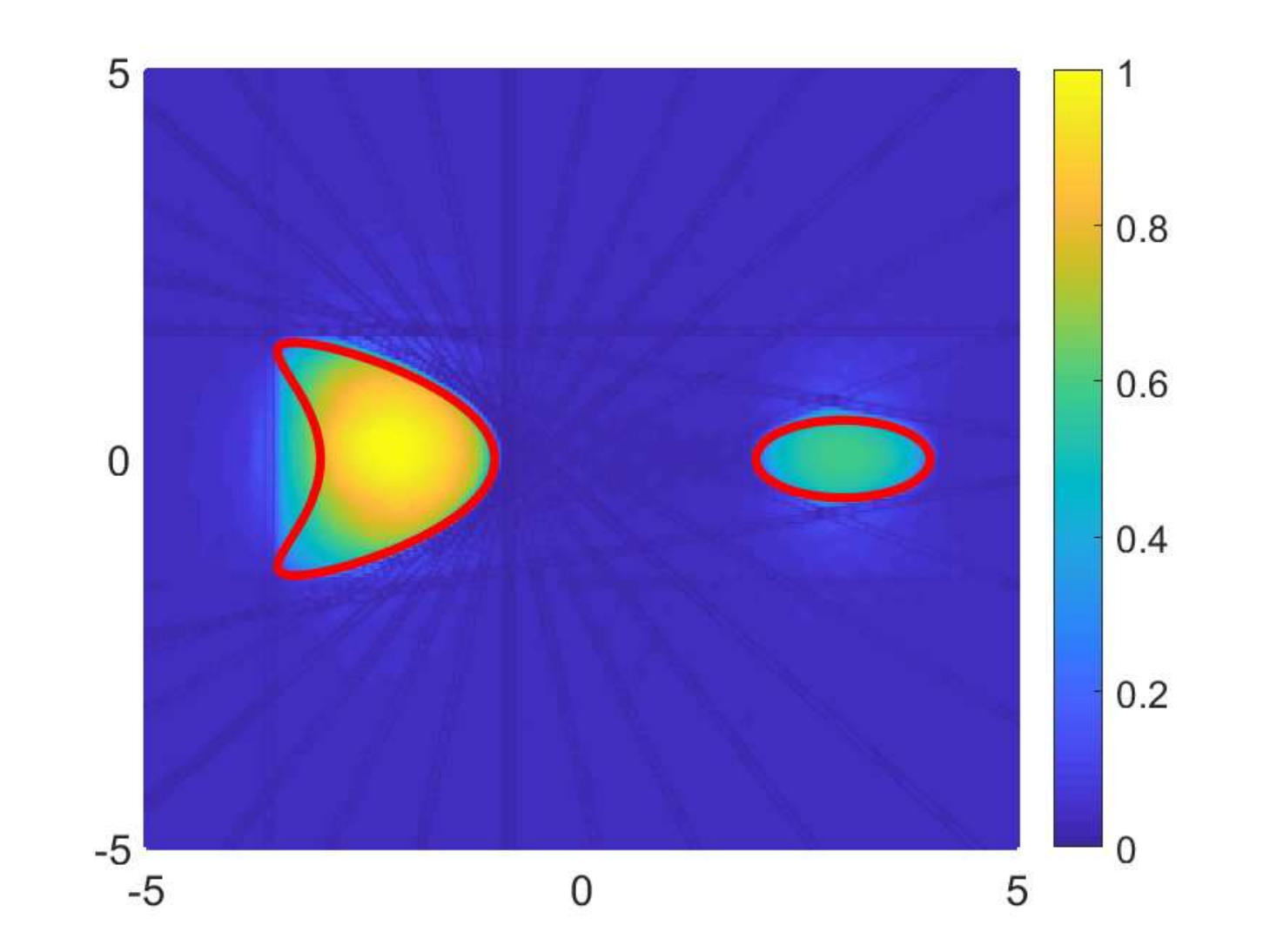}
		
	}
	
	\subfigure[$F=(x_1+10)t$, $\varepsilon=0.30$]{
		\includegraphics[scale=0.3]{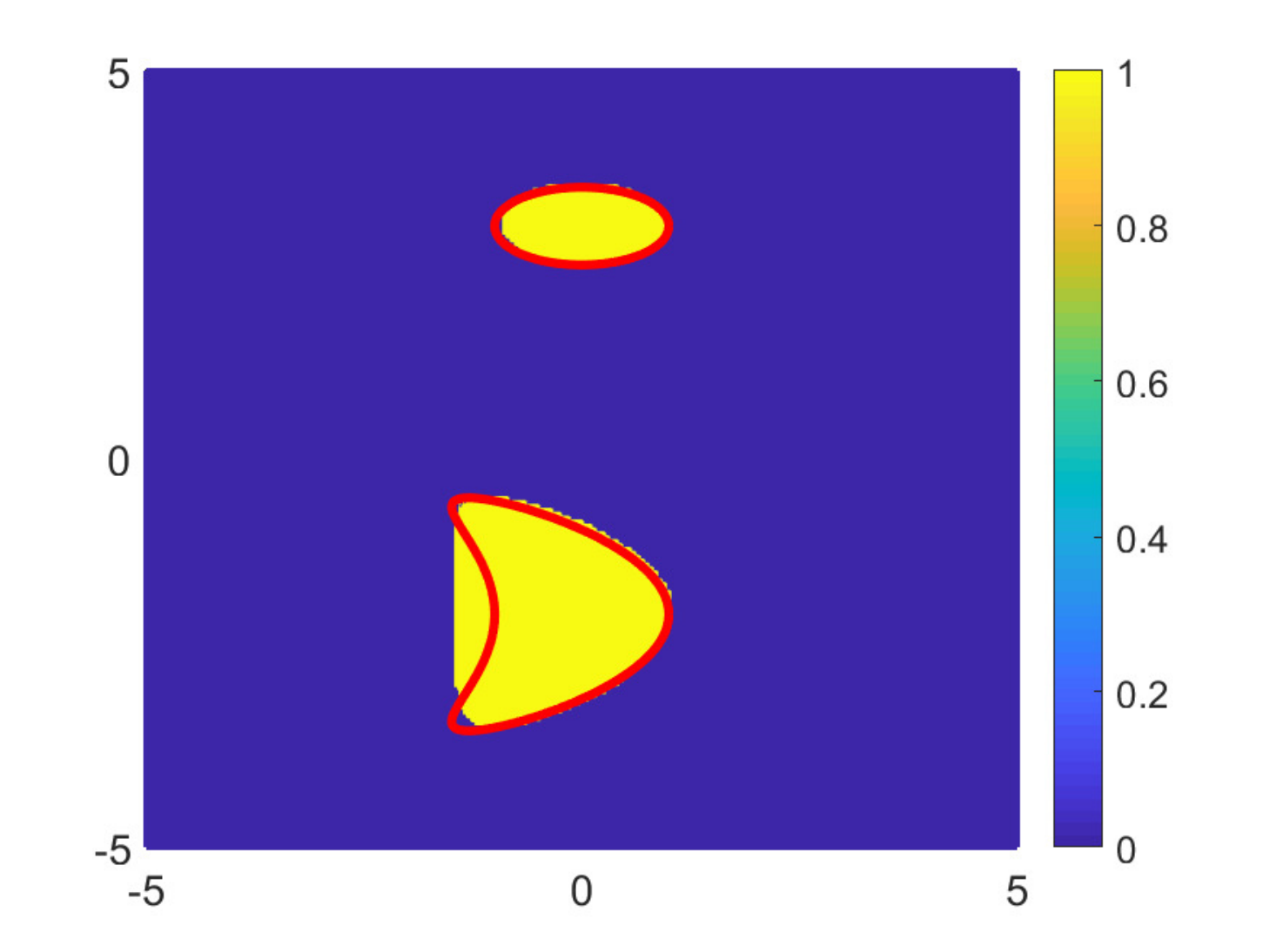}
	}
	\subfigure[$F=(x_1^2+x_2^2+10)t$, $\varepsilon=0.30$ ]{
		\includegraphics[scale=0.3]{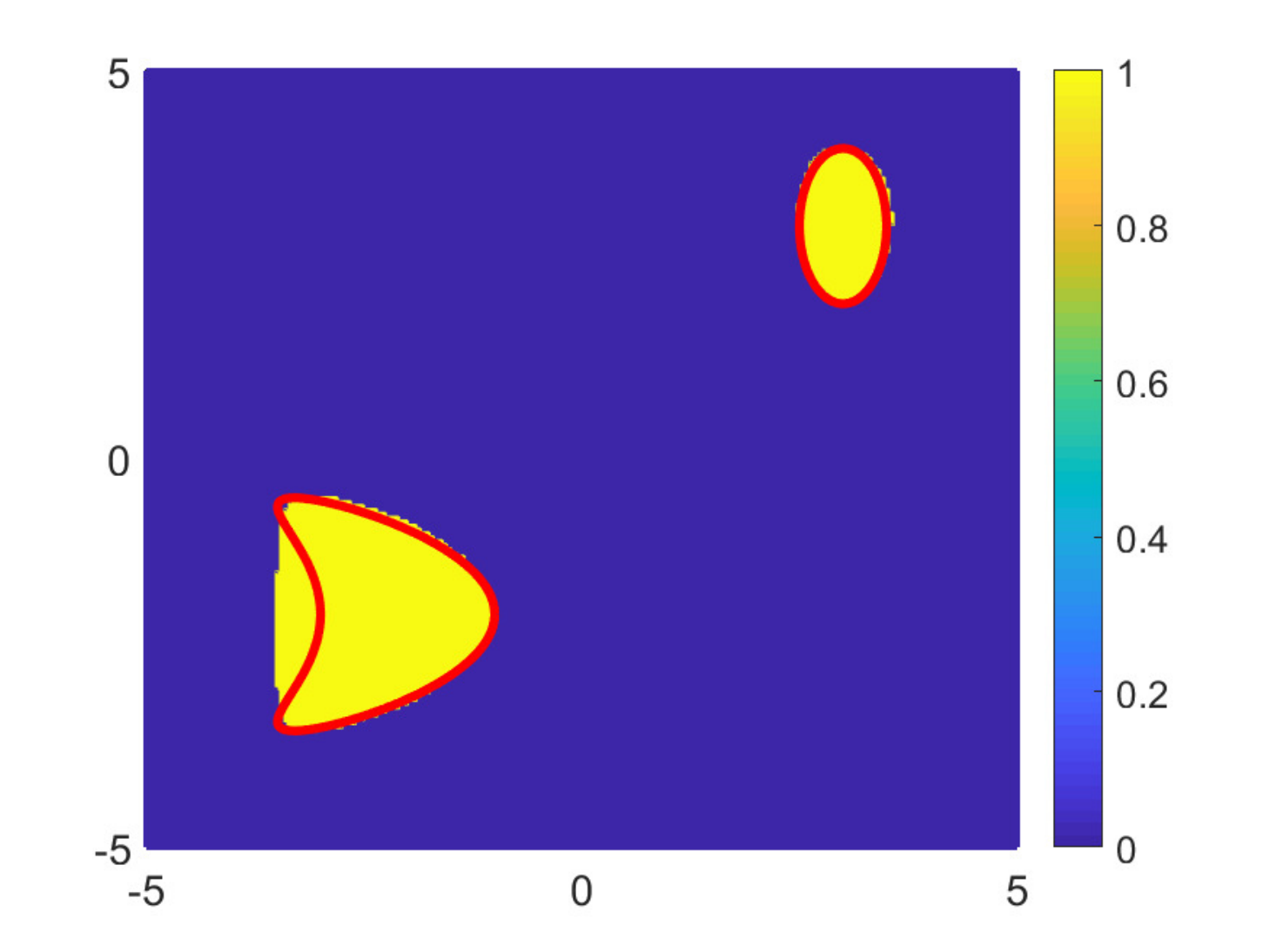}
		
	}
	\subfigure[$F=(x_2+10)t$, $\varepsilon=0.30$]{
		\includegraphics[scale=0.3]{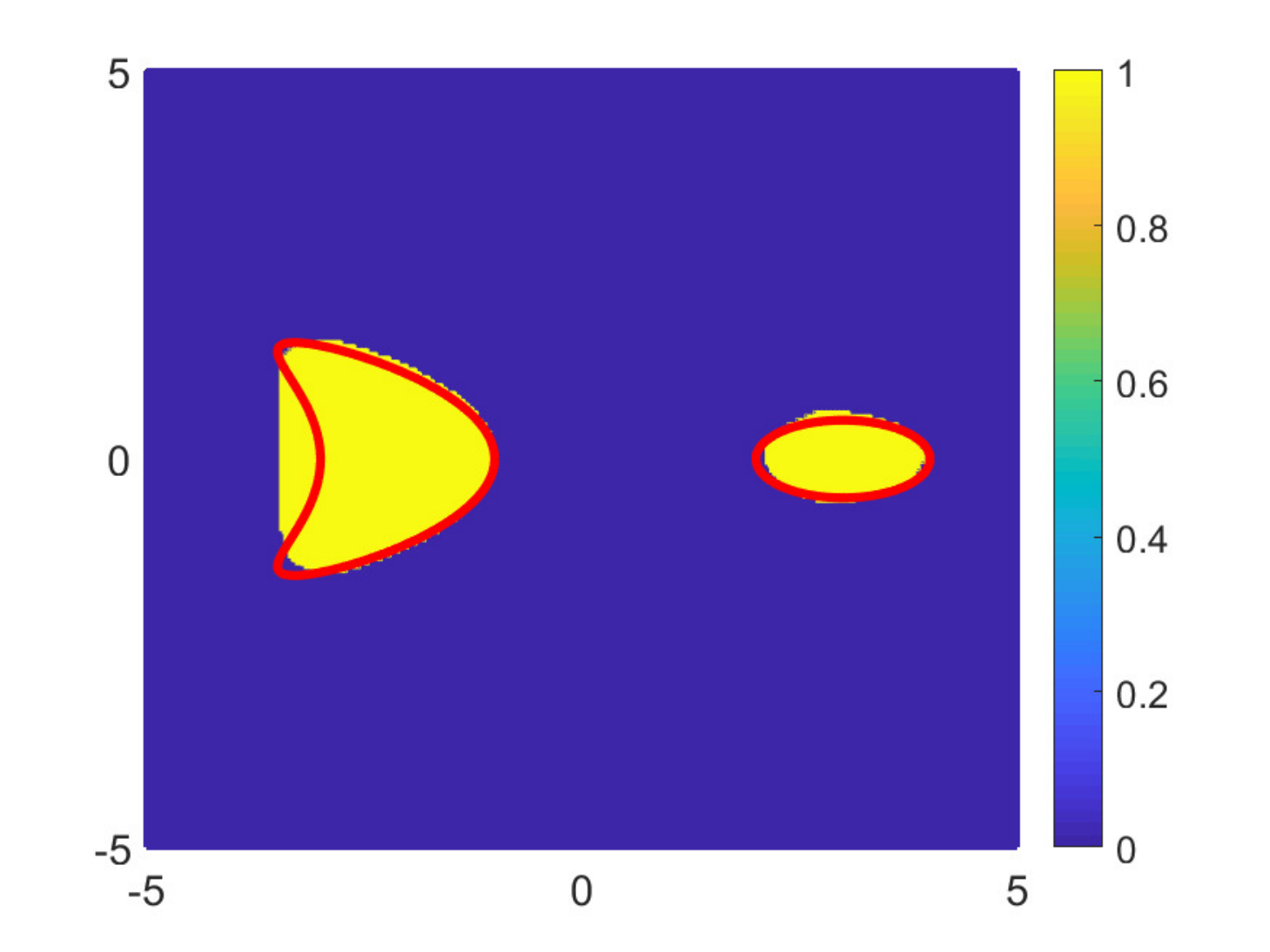}
			}
	
	\caption{Reconstructions of a kite-ellipse-shaped support with different locations and sources by using $16$ observations. The Fourier transform windows is $(0, 0.1)$.
	} \label{fig:Mdir2}
\end{figure}

We show the experimental examples in Fig. \ref{fig:Mdir2T} by using different Fourier transform windows $(0,T)$. The reconstructions of two disconnected supports are displayed in Fig.\ref{fig:Mdir2T}, where we fix $t_{\min}=0$ and $K=20$. We choose $T=t_{\max}-t_{\min}=1$ and $N=200$ in Fig.\ref{fig:Mdir2T} (a), $T=3$ and $N=300$ in Fig.\ref{fig:Mdir2T} (b), $T=7$ and $N=300$ in Fig.\ref{fig:Mdir2T} (c). In the last case of $T=7$,
the reconstruction is distorted because the assumption (A) is not satisfied. Here we choose $M=16$ observation directions.

Finally, we examine the sensitivity of the direct sampling algorithm to noise.  The synthetic data is polluted by noise via the following formula
\be
u_{\delta}^{\infty}(\hat{x},k)= u^{\infty}(\hat{x},k)+\delta \mathbf{R}\circ \, u^{\infty}(\hat{x},k),
\en
where $\delta$ is the noise level, $\mathbf{R}\in\R^{W\times N}$ is a uniformly distributed random matrix with the random variable ranging from $-1$ to $1$ and $\circ$ represents the Hadamard product. We test three different noise levels $\delta=10\%$, $50\%$, and $100\%$. The recovered results are displayed in Fig.\ref{fig:Mdir2noise}, from which one can conclude that the inversion algorithm is rather robust against noise. This phenomenon has also been reported in other existing literatures (see e.g., \cite{CCH13,L17}). In our tests, we find that boundaries of the strips $\mathcal{S}^{(\hat{x}_m)}$ are also visible at high noise levels.
\begin{figure}[H]
	\centering
	\subfigure[$T=1$]{
		\includegraphics[scale=0.3]{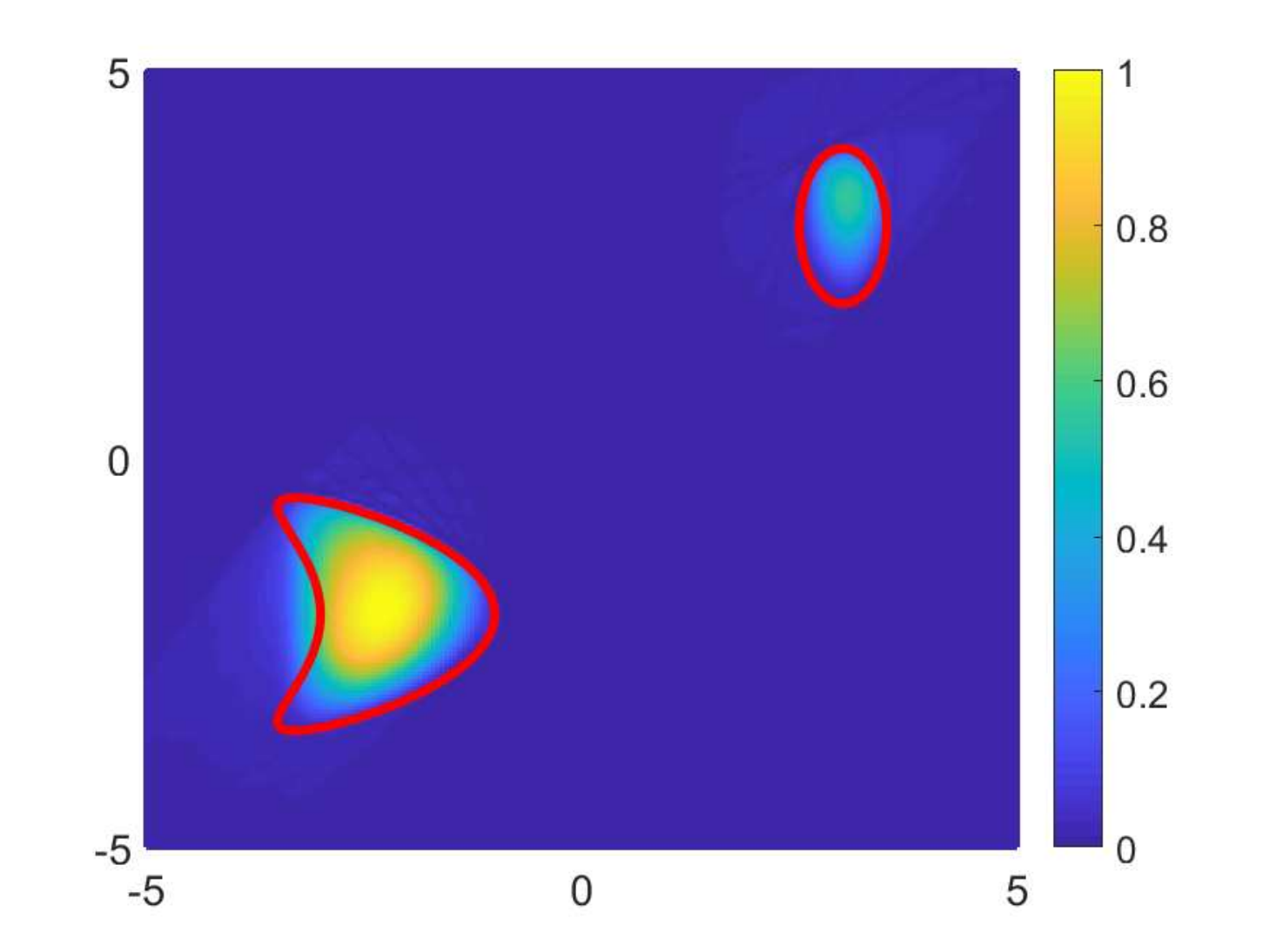}
	}
	\subfigure[$T=3$ ]{
		\includegraphics[scale=0.3]{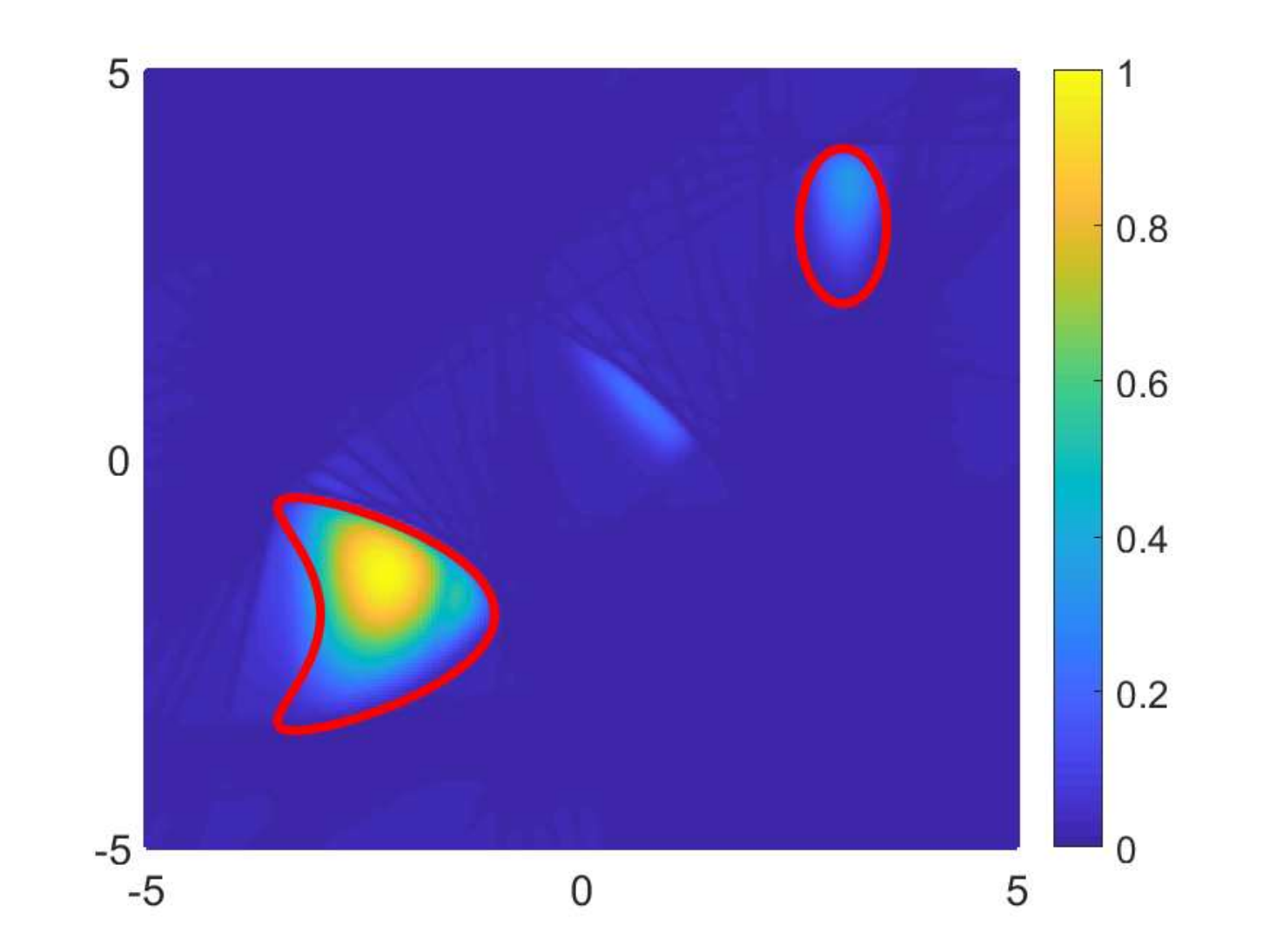}
		
	}
	\subfigure[$T=7$]{
		\includegraphics[scale=0.3]{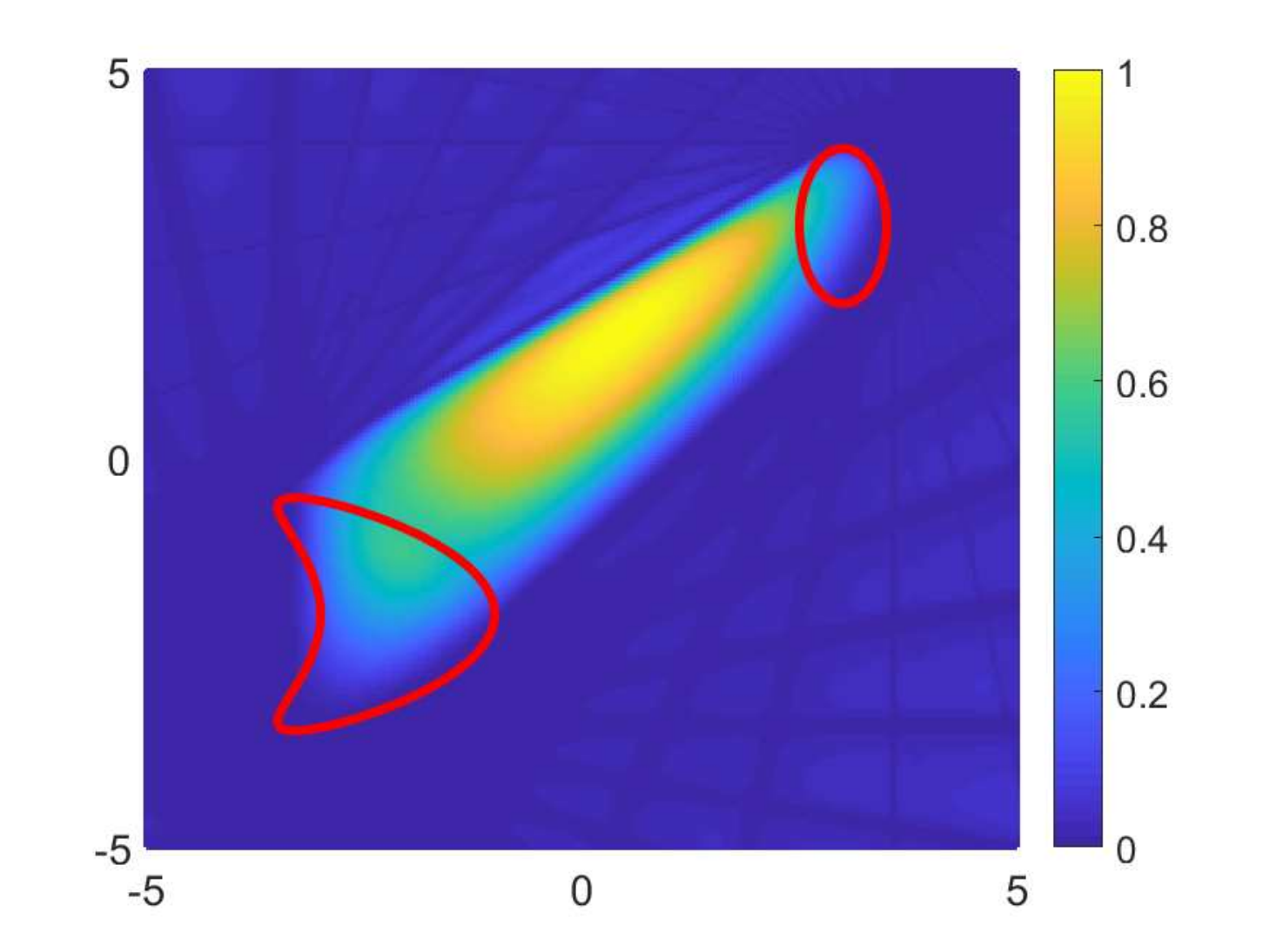}
		
	}
	\caption{Reconstructions of a kite-shaped support with $F(x,t)=(x_1^2+x_2^2+10)t$ with different Fourier transform windows $(0, T)$.
	} \label{fig:Mdir2T}
\end{figure}

\begin{figure}[H]
	\centering
	\subfigure[$\delta=10\%$]{
		\includegraphics[scale=0.3]{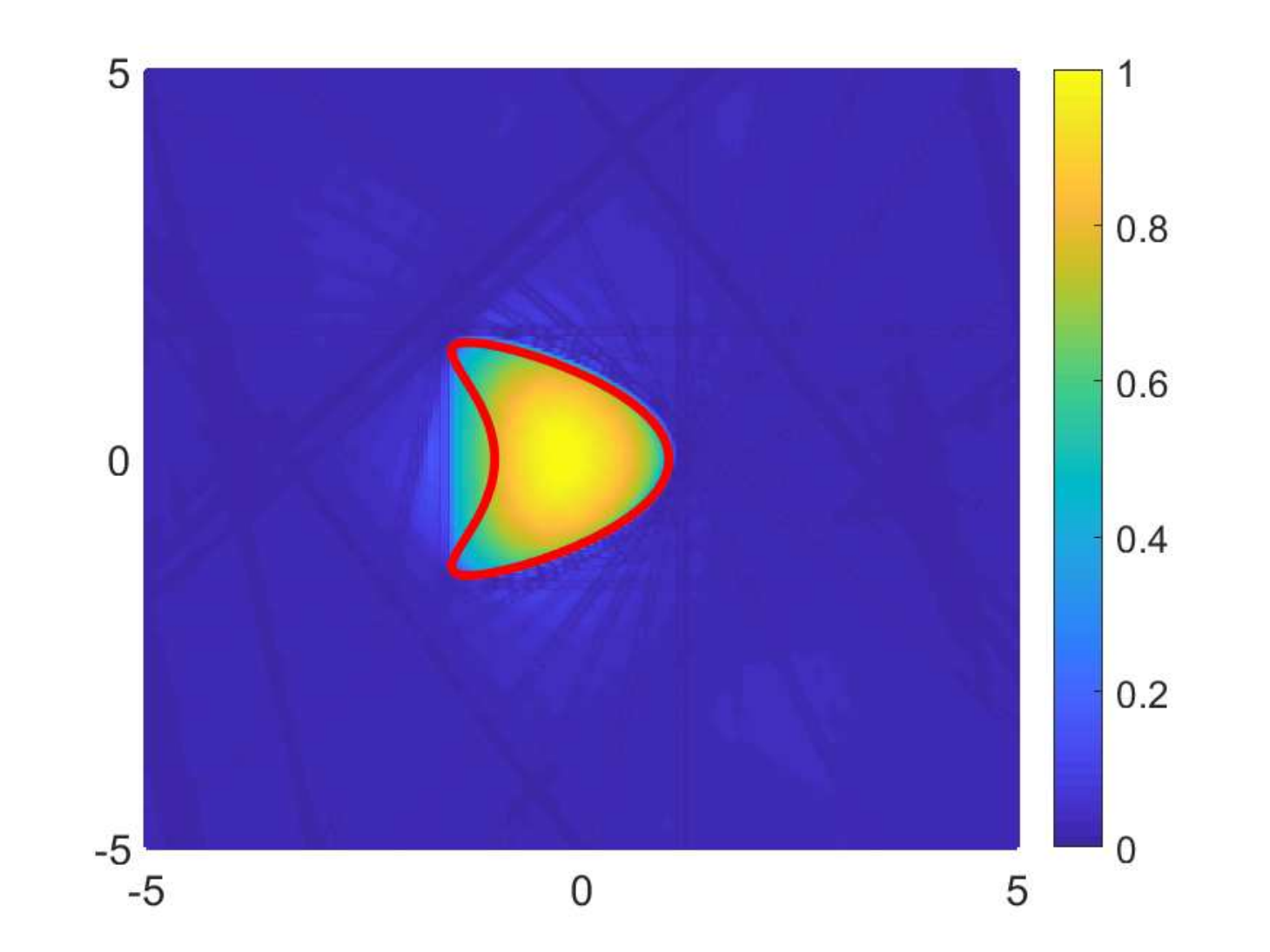}
	}
	\subfigure[$\delta=50\%$ ]{
		\includegraphics[scale=0.3]{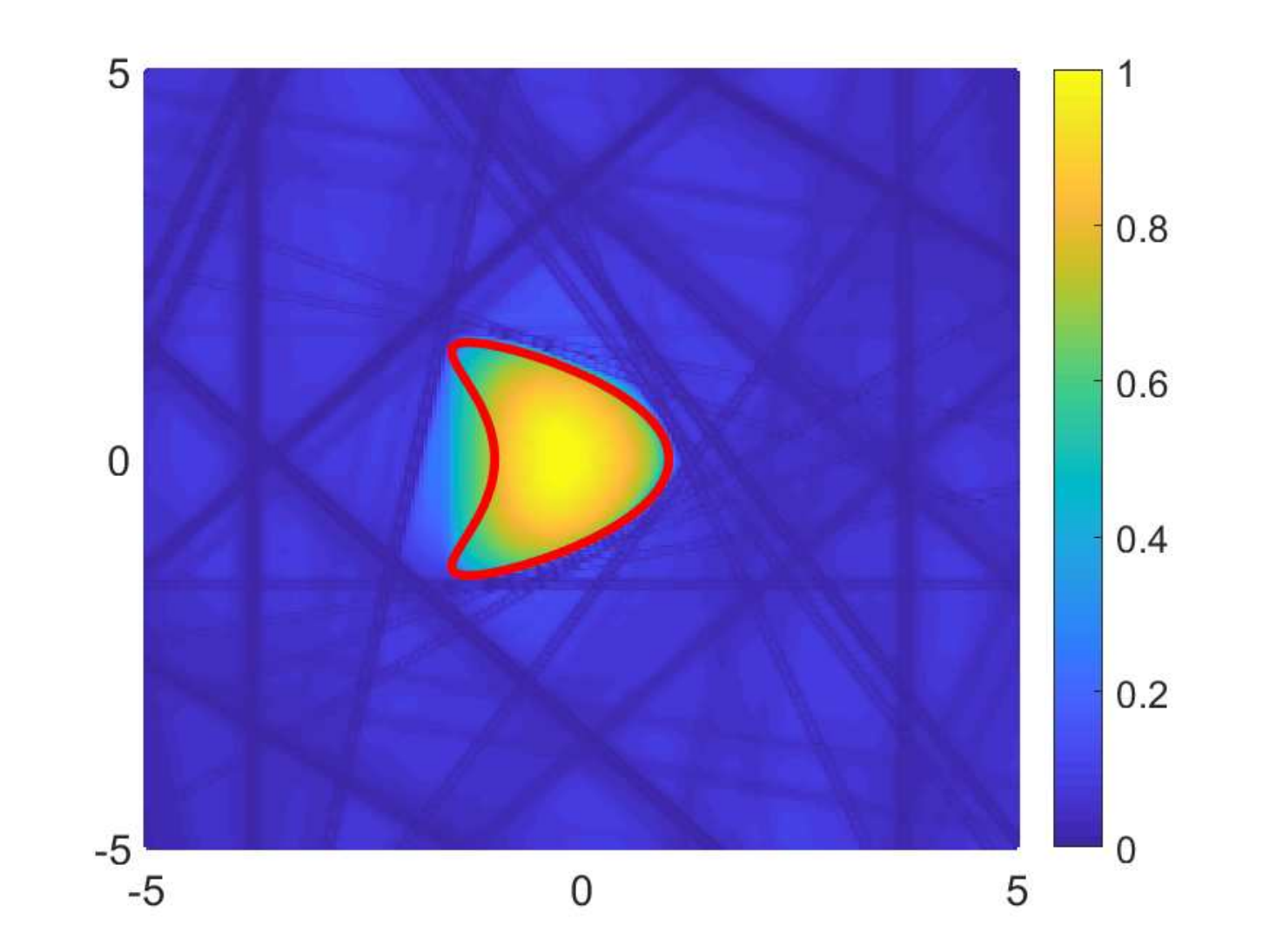}
		
	}
	\subfigure[$\delta=100\%$]{
		\includegraphics[scale=0.3]{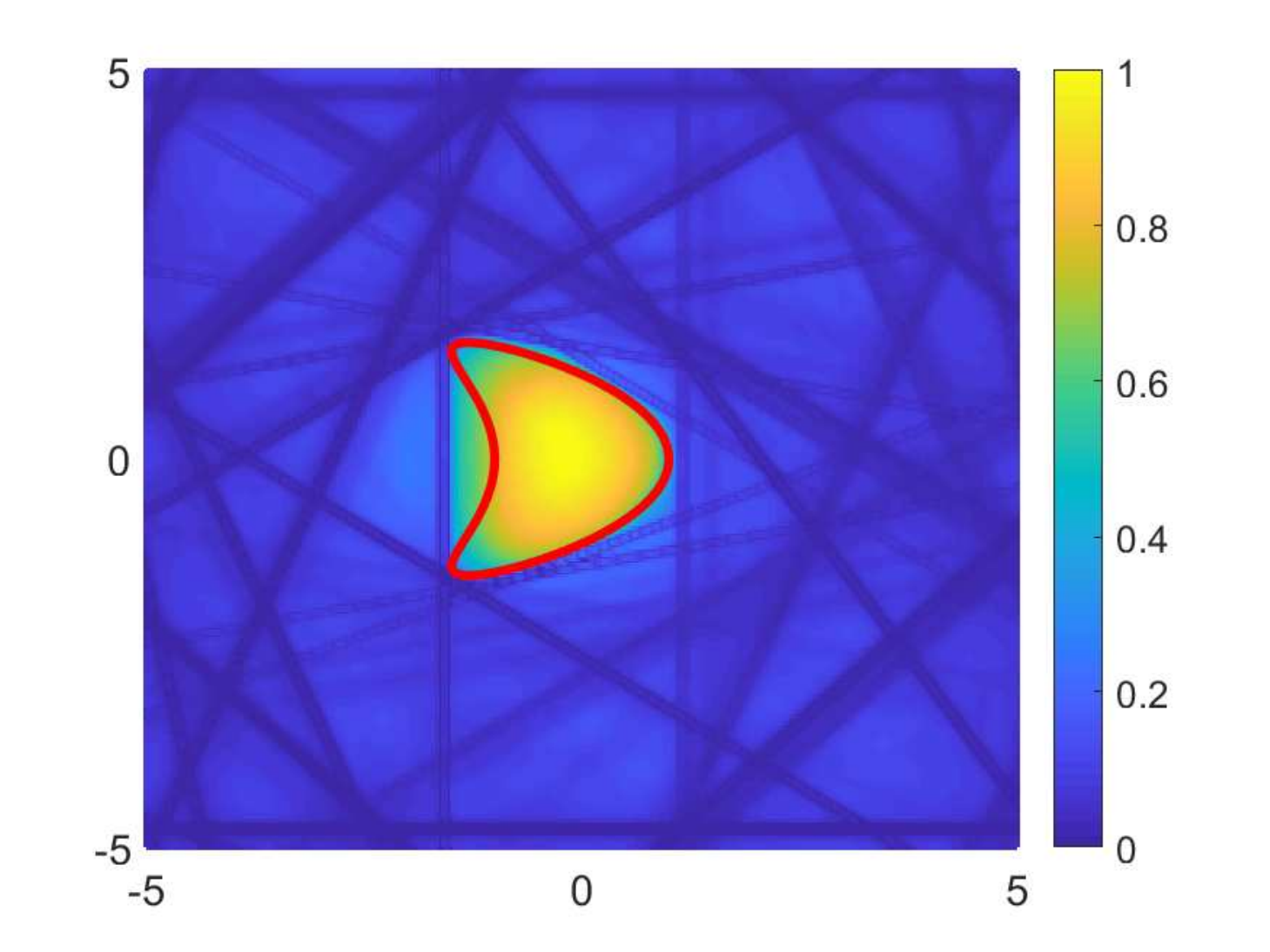}
		
	}
	\caption{Reconstructions of a kite-shaped support with $F(x,t)=(x_1^2+x_2^2+10)t$ with different noise levels $\delta$.
	} \label{fig:Mdir2noise}
\end{figure}

\section{Reconstructions from near-field measurements in $\R^3$}\label{sec:5}
In this section, we suppose that $D\subset B_R\subset \R^3$ for some $R>0$ and take the near-field measurement data on sparse observation points lying on $|x|=R$. We remark that it seems difficult to extend the results of this section to two dimensions, perhaps due to the lack of Huygens' principle in 2D. 
For simplicity we supposed that $D$ is connected. For source supports with multiple components, one can proceed with analogous arguments to the far-field case in two dimensions.
\subsection{Indicator and test functions}
By \eqref{expression-w}, the near-field data can be expressed as
\ben
u(x,k)&=&\int _{D} \frac{e^{ik|x-y|}}{4\pi |x-y|}\int_{t_{\min}}^{t_{\max}}F(y,t)e^{-ikt} dtdy,\quad x\in \R^3.
\enn
The supporting interval of the Fourier transform of the near-field data with respect to frequencies is stated below.
\begin{lem}\label{lem5} Let the assumption $(\ref{F})$ hold true and fix some $x_0\in \partial B_R$.	 The supporting interval of $\mathcal{F}^{-1}[u(x_0,k)]$ is $$H_{0}:=\big(t_{\min}-\sup_{z\in D}|x_0-z|,\ t_{\max}-\inf_{z\in D}|x_0-z|\big).$$ Moreover, the function $t\longmapsto (\mathcal{F}^{-1}u)(t)$ is positive in $H_{0}$.
\end{lem}

\begin{proof} For $t\in \R$,
define $$\Gamma(t):=\left\lbrace y\in D:|x_0-y|=t\right\rbrace.$$
 Like the far-field case, we reformulate the near-field expression as the Fourier transform by
	\be
	u(x_0,k)
	&=&\int_{t_{\min}}^{t_{\max}} \int_{D} F(y,t) \frac{e^{-ik(t-|x_0-y|)}}{4\pi|x_0-y|}dy dt \qquad \nonumber \\ \nonumber
	&=&\int_{t_{\min}}^{t_{\max}} \int_{\R}\int_{\Gamma(t-\xi)} F(y,t)   \frac{e^{-ik\xi}}{4\pi(t-\xi)}ds(y)d\xi dt  \\ \label{wh}
	&=&\int_{\R} e^{-ik\xi} h(\xi) d\xi =(\mathcal{F}h)(k),
	\en
	where
	\be\label{h}
	h(\xi):=\int_{t_{\min}-\xi}^{t_{\max}-\xi} \int_{\Gamma(t)}\frac{F(y,t+\xi)}{4\pi t}ds(y)dt,\quad \xi\in \R.
	\en
Since $\Gamma(t)=\emptyset$ for $t<\inf_{z\in D}|x_0-z|$ or $t>\sup_{z\in D}|x_0-z|$,
we have $h(\xi)=0$ for $\xi<t_{\min}-\sup_{z\in D}|x_0-z|$ or $\xi>t_{\max}-\inf_{z\in D}|x_0-z|$. Thus, $$ \supp\,(\mathcal{F}^{-1}u(x_0, k))=\supp \,(h)\subset H_0.$$

Next we show that $h$ is positive for $t\in H_0$. For $\xi\in H_0$, we suppose that $$\xi=t_{\min}-\sup_{z\in D}|x-z|+\varepsilon\quad\mbox{for some}\quad\varepsilon \in (0,\Lambda),$$
with the notations $$\Lambda:=T+L, \quad L:=\sup_{z\in D}|x_0-z|-\inf_{z\in D}|x_0-z|.$$ The number $\Lambda>0$ represents the length of the interval $H_0$.
By the assumption $(\ref{F})$, we deduce from \eqref{h} that
	\be
	h(\xi)=\int_{\inf_{z\in D}|x_0-z|}^{\sup_{z\in D}|x_0-z|-\varepsilon+T}\int_{\Gamma(t)}\frac{F(y,t+\xi)}{4\pi t}ds(y)dt.
	\en
We observe that $$\xi+t\in \big(t_{\min}-L+\varepsilon, t_{\max}\big)\qquad \mbox{if}\quad \  t\in\big(\inf_{z\in D}|x_0-z|),\;\sup_{z\in D}|x_0-z|-\varepsilon+T\big),$$
and that for any $\epsilon\in(0, \Lambda)$,
\ben
\big(\inf_{z\in D}|x_0-z|),\;\sup_{z\in D}|x_0-z|-\varepsilon+T\big)\,\cap\,
\big(\inf_{z\in D}|x_0-z|),\;\sup_{z\in D}|x_0-z|\big)\neq \emptyset, \\
\big(t_{\min}-L+\varepsilon, t_{\max}\big)\cap
\big(t_{\min}, t_{\max}\big)\neq \emptyset.
\enn
This together with the positivity of $F$ proves $h>0$ in $H_0$.
\end{proof}



\noindent For some fixed $|x|=R$, we introduce two test functions
\ben
\psi^{(x)}_{1}(y,k):=e^{ik(|x-y|-t_{\min})},\quad
\psi^{(x)}_{2}(y,k):=e^{ik(|x-y|-t_{\max})},
\enn
and two auxiliary indicator functions
\be \label{near-Ij}
I^{(x)}_j(y):=\int_{\R}u(x,k) \ \overline{\psi_j^{(x)}(y,k)}\ dk, \quad j=1,2.
\en
From the multi-frequency near-field data $u(x, k)$ at the observation point $x\in\partial B_R$, we want to image the annular domain
\ben \label{annu}
\mathcal{A}^{(x)}:=\{y\in \R^3: \inf_{z\in D}|x-z|<|x-y|<\sup_{z\in D}|x-z|\}.
\enn
For this purpose, we design the indicator function
\be  \label{Indicator-near}
I^{(x)}(y):=\left[\frac{1}{I_1^{(x)}(y)}+  \frac{1}{I_2^{(x)}(y)}   \right]^{-1}=\frac{I_1^{(x)}(y)\; I_2^{(x)}(y)}{I_1^{(x)}(y)+I_2^{(x)}(y)},\quad y\in \R^3.
\en

\begin{thm}\label{ID2} Let $D\subset B_R$ for some $R>0$ and we fix some $|x|=R$. Then it holds that
\ben
I^{(x)}(y)
=\left\{\begin{array}{lll}
0\qquad\quad&&\mbox{for}\quad y\notin\mathcal{A}^{(x)},\\
{\rm finite\; positive\; number} \qquad&&\mbox{for}\quad y
\in \mathcal{A}^{(x)}.
\end{array}\right.
\enn
\end{thm}
\begin{proof} Arguing analogously to the proof of \eqref{test}, we see
\ben
\psi^{(x)}_1(y,k)= [\mathcal{F}(\delta(\xi+|x-y|-t_{\min})](k),\\
\psi^{(x)}_2(y,k)= [\mathcal{F}(\delta(\xi+|x-y|-t_{\max})](k).
\enn
In view of \eqref{wh} and \eqref{h}, we get
\be \label{Inear}
	I^{(x)}_1(y)= h(t_{\min}-|x-y|),\quad
I^{(x)}_2(y)= h(t_{\max}-|x-y|).
\en
Hence, using the results of Lemma \ref{lem5} yields

\ben\label{In3}
I^{(x)}_1(y)
=\left\{\begin{array}{lll}
0\qquad\quad&&\mbox{for}\quad |x-y| \notin \big(\inf_{z\in D}|x-z|-T, \sup_{z\in D}|x-z|\big),\\
 {\rm finite\; positive\; number} \qquad&&\mbox{for}\quad |x-y| \in \big(\inf_{z\in D}|x-z|-T, \sup_{z\in D}|x-z|\big);
\end{array}\right. \\ \label{In4}
I^{(x)}_2(y)=\left\{\begin{array}{lll}
 0\qquad\quad&&\mbox{for}\quad
 |x-y| \notin \big(\inf_{z\in D}|x-z|, \sup_{z\in D}|x-z|+T\big),\\
 {\rm finite\; positive\; number}\qquad&&\mbox{for}\quad
 |x-y| \in \big(\inf_{z\in D}|x-z|, \sup_{z\in D}|x-z|+T\big).
 \end{array}\right.
 \enn
Now, the indicating behavior of $I^{(x)}$ follows from the fact $\supp I^{(x)}:=\supp \, I^{(x)}_1 \cap \supp \, I^{(x)}_2=(\inf_{z\in D}|x-z|, \sup_{z\in D}|x-z|)$.
\end{proof}
Using the multi-frequency data taking at sparse observation positions $x_m\in\partial B_R$ ($m=1,2,\cdots, M$), one can design a new indicator function via superposition for imaging the intersection $\bigcap_{m=1}^M \mathcal{A}^{(x_m)}$. We omit the details for brevity, since it can be done analogously to the far-field case in 2D.

\subsection{Numerical tests using near-field measurements in $\R^3$}
In this section, we present numerical reconstructions of the source support  $D\subset\R^3$ from the multi-frequency data $\left\lbrace u(x_m,  k): x_m\in \pa B_R,k\in(0,K),m=1,2, \cdots , M \right\rbrace $ at sparse observation positions. Here $K>0$ is a truncated number.
In view of Theorem \ref{ID2},  the smallest annular domain containing  the support and centered at $x$ can be restored using the indicator $I^{(x)}$ (\ref{Indicator-near}) from
the near-field data $u(x,k)$, $k\in(0,K)$. To discretize the integral with respect to wavenumbers, we take $$ k_n:=n \Delta k\;\;\mbox{with}\;\; \Delta k:= \frac{K}{N}>0.$$ The indicator functions $I_j^{(x)}(y)$ defined by (\ref{near-Ij}) can be approximated by
\ben
I_j^{(x)}(y) \sim 2\mbox{Re}\left \{ \int_0^K u(x,k) \ov {\psi^{(x)}_{j}(y,k)}\,dk\right\},\; j=1,2.
\enn
Then the source support $D$ can be imaged by plotting the indicator
\be \label{ID-near}
 I(y)=\left[\sum_{m=1}^{M}\dfrac{1}{I^{(x_m)}(y)}\right]^{-1}=\left[\sum_{m=1}^{M}
 \frac{I_1^{(x_m)}(y)+I_2^{(x_m)}(y)}{I_1^{(x_m)}(y)\; I_2^{(x_m)}(y)}\right]^{-1},\quad y\in \R^3.
  \en
Similarly to the far-field case (see Theorem \ref{ID2}), we have the following result with sparse observation positions.

\begin{thm}
Suppose that the positivity condition \eqref{F} holds on the domain $D\subset B_R$. Then it holds that
\be
I(y)=\left\{\begin{array}{lll}
0\qquad\quad&&\mbox{for}\quad y\notin \bigcap_{m=1}^M \mathcal{A}^{(x_m)},\\
{\rm finite\; positive\; number} \qquad&&\mbox{for}\quad y
\in \bigcap_{m=1}^M \mathcal{A}^{(x_m)}.
\end{array}\right.
\en
\end{thm}

Unless otherwise stated, we take $K=20$ and $N=100$ as default in the following experiments.

\textbf{Example 1}
In this example, we illustrate the reconstructions of the annular $\mathcal A ^{ (x)} $ for a cube by plotting the indicator $I^{(x)}(y)$ in (\ref{Indicator-near}). We set the time dependent source function to be  $F(x,t)=(x_1^2+x_2^2+x_3^2+1)(t+1)$  which satisfies the positivity condition (\ref{F}). We assume that  $F$ is supported in $D\times(t_{\min},t_{\max})$ with $t_{\min}=0$ and $t_{\max}=0.1$. The cube $D$ is defined by (see Figure \ref{fig:3d-cube} (a))
\ben
D=\{(x_1,x_2,x_3): |x_1|<1, |x_2|<1\; \mbox{and}\; |x_3|<1 \}.
\enn
We take the observation point at $x=(3,0,0)$ and plot the indicator in the search domain  $[-3, 3]^3$. Figures \ref{fig:3d-cube} (b) and (c) show visualizations of the indicator function (\ref{Indicator-near}) in the search area.    Fig.\ref{fig:3d-cube} (b) shows a slice of the reconstruction at $y_2=0$. We see that the cross of the plane $y_2=0$ with the smallest annular containing the square (in red) and centered at $x=(3,0, 0)$ is nicely reconstructed.  Fig.\ref{fig:3d-cube} (c) illustrates an iso-surface of the reconstruction at the iso-level $1.8\times10^{-2}$.  The iso-surfaces perfectly enclose the cube-shaped support.
\begin{figure}[H]
	\centering
	\subfigure[Geometry of source support]{
		\includegraphics[scale=0.22]{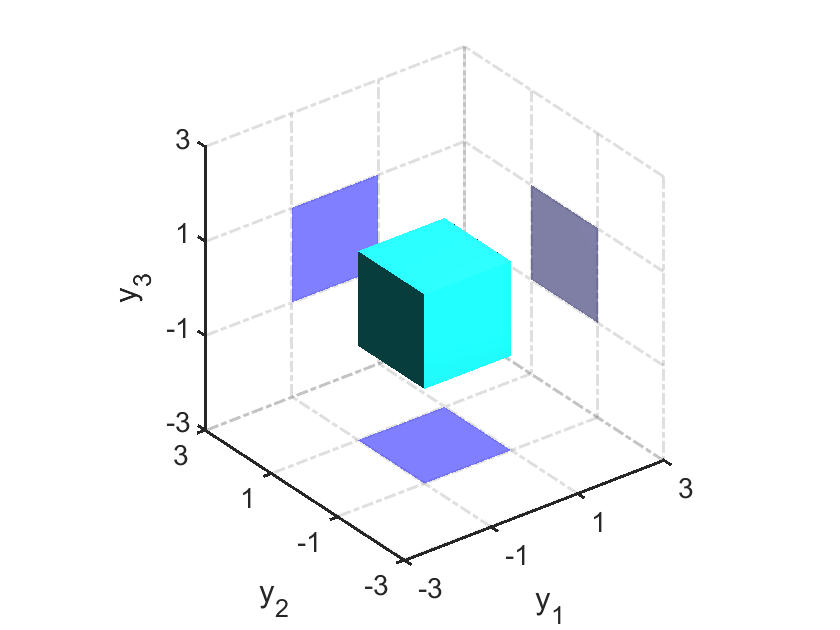}
	}
	\subfigure[Slice at $y_2=0$]{
		\includegraphics[scale=0.22]{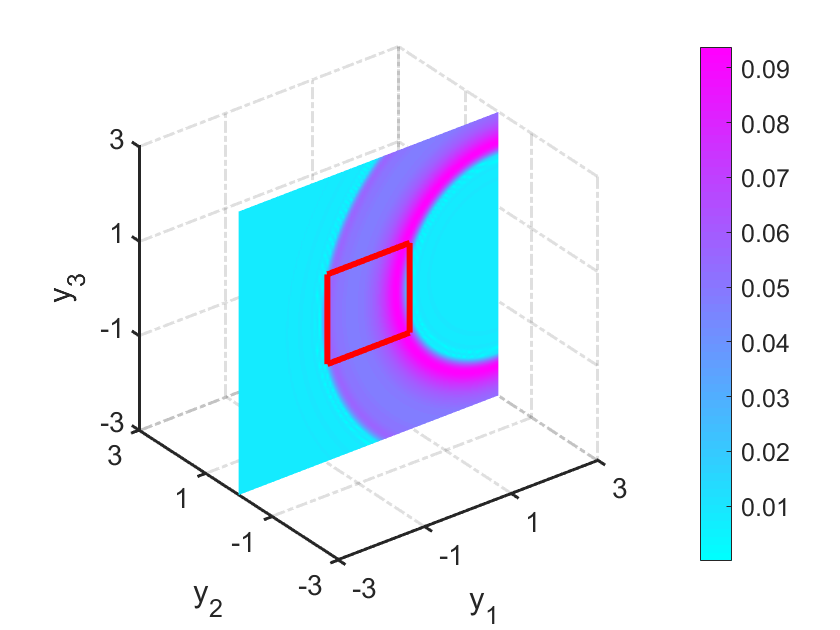}
		
	}
	\subfigure[Iso-surface level $=1.8\times 10^{-2}$]{
		\includegraphics[scale=0.22]{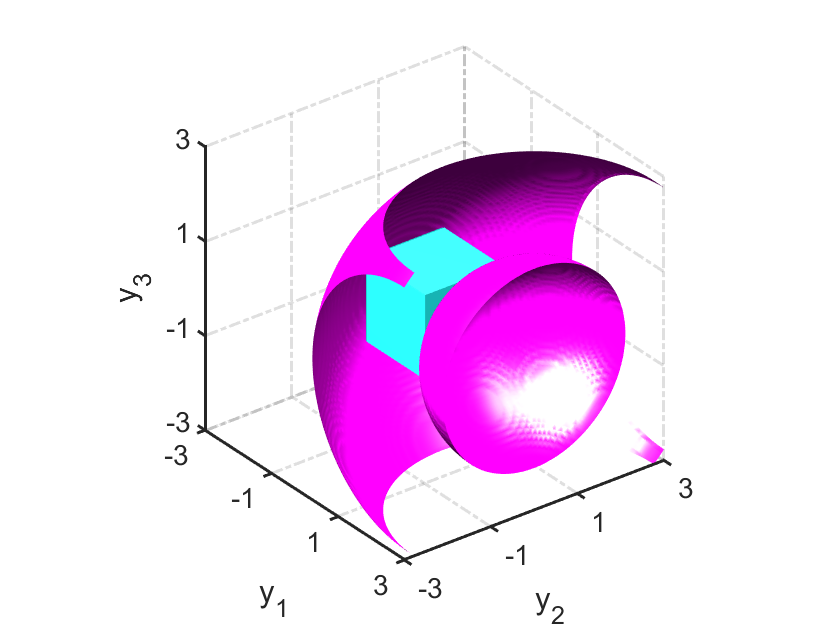}
		
	}
	\caption{ Reconstruction of a cube-shaped support using multi-frequency data from a single observation point.
	} \label{fig:3d-cube}
\end{figure}

\textbf{Example 2}.
We continue \textbf{Example 1} with multiple observation points. A visualization of the indicator function  (\ref{ID-near}) is shown in Fig.\ref{fig:3d-rec} with $M$ observation points. The multi-frequency data at two observation points $\{ (3,0,0), (-3,0,0)\}$ are utilized in Fig.\ref{fig:3d-rec} (a), and the data at four observation points $\{ (3,0,0), (-3,0,0), (0,3,0), (0,-3,0)\}$ in Fig.\ref{fig:3d-rec} (b), where the position of the source is nicely reconstructed. Using six observation points $\{ (3,0,0), (-3,0,0), (0,3,0), (0,-3,0), (0,0,3), (0,0,-3)\}$ in Fig.\ref{fig:3d-rec} (c), we can see that both the shape and location of the source support are well reconstructed.

\begin{figure}[H]
	\centering
	\subfigure[$M=2$]{
		\includegraphics[scale=0.22]{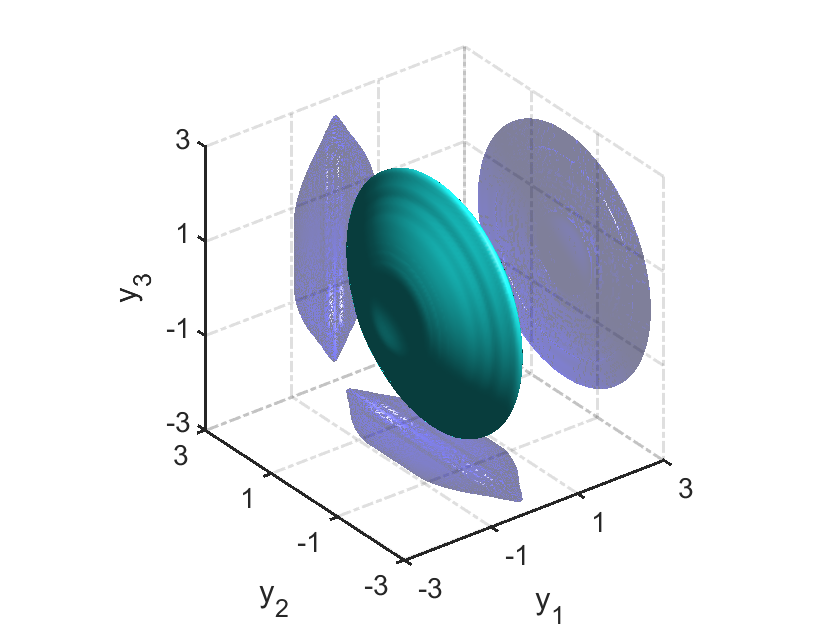}
	}
	\subfigure[$M=4$]{
		\includegraphics[scale=0.22]{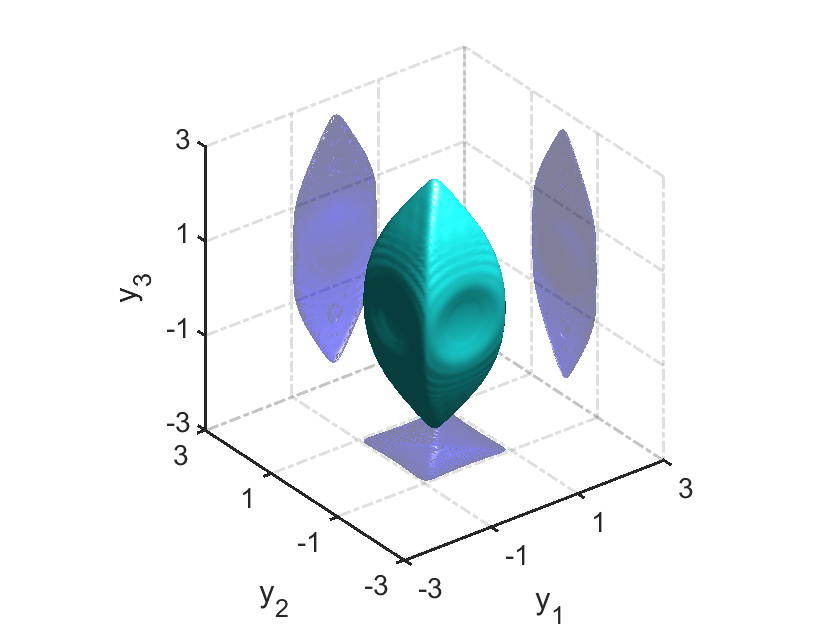}
		
	}
	\subfigure[$M=6$]{
		\includegraphics[scale=0.22]{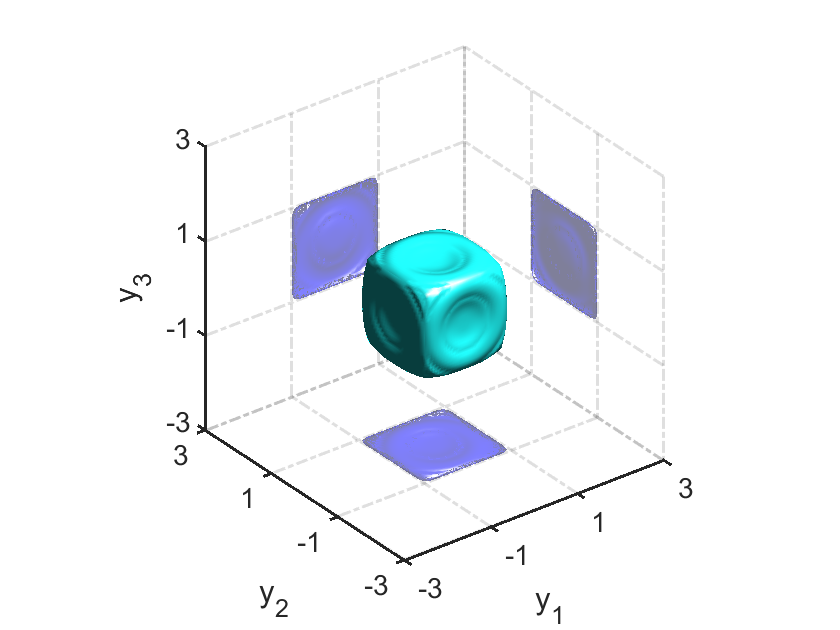}
		
	}
	\caption{Reconstruction of a cube-shaped support using multi-frequency data from multiple observation points. The number of observation points is $M=2$, $4$, $6$  and  the corresponding iso-surface levels are $2.5\times10^{-2}$, $1.2\times10^{-2}$ and $7\times10^{-3}$, respectively.
	} \label{fig:3d-rec}
\end{figure}
To clearly illustrate the reconstructions in Fig.\ref{fig:3d-rec}, we also plot the projections of the images onto the $oy_1 y_2$, $oy_1 y_3$ and $oy_2 y_3$ planes. From these 2D visulizations one sees that the projections at $y_1-$axis are $[-1,1]$ in Fig.\ref{fig:3d-rec} (a), the projections on $0y_1y_2$ is square $[-1,1]^2$ in Fig.\ref{fig:3d-rec} (b) and the projections are all squares $[-1,1]^2$ in Fig.\ref{fig:3d-rec} (c), which prove the accuracy of the 3D reconstructions.
Fig.\ref{fig:3d-slice} shows  slices of the reconstruction at the planes $y_1=0$, $y_2=0$ and $y_3=0$ using the data at six observation points (see Fig.\ref{fig:3d-rec} (c)). For comparison we also demonstrate the boundary of the source support's slice with the red solid line. These slices also confirm the accuracy of our algorithm.

\begin{figure}[htp]
	\centering
	\subfigure[Slice at $y_3=0$]{
		\includegraphics[scale=0.22]{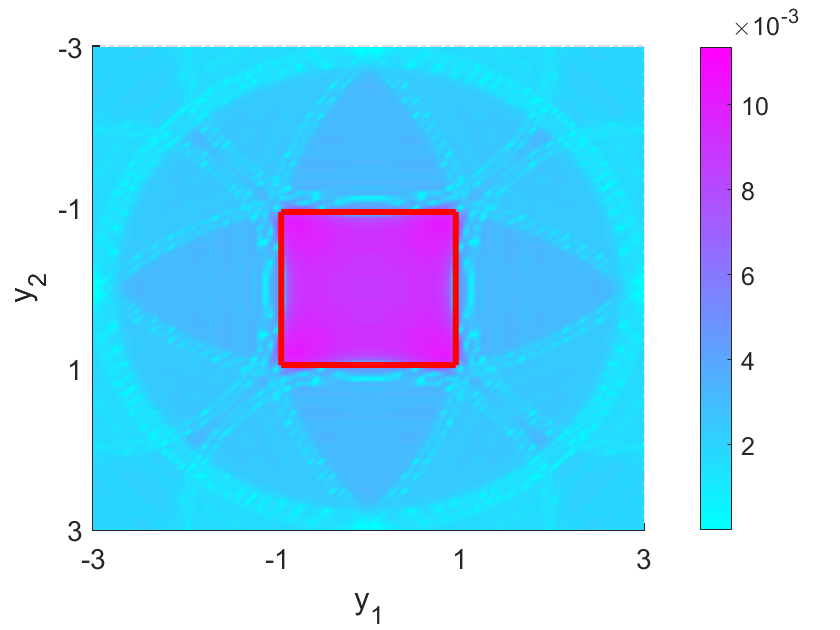}
	}
	\subfigure[Slice at $y_2=0$]{
		\includegraphics[scale=0.22]{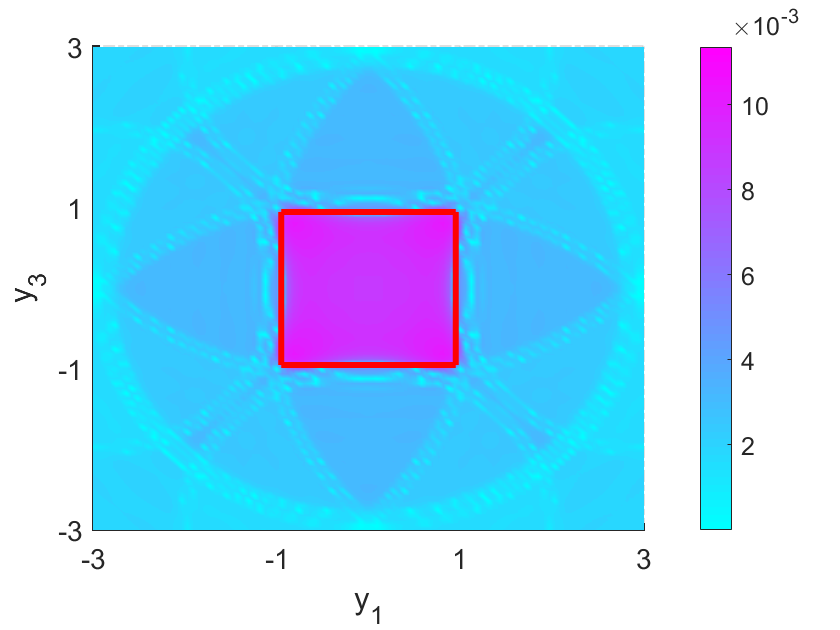}
		
	}
	\subfigure[Slice  at $y_1=0$]{
		\includegraphics[scale=0.22]{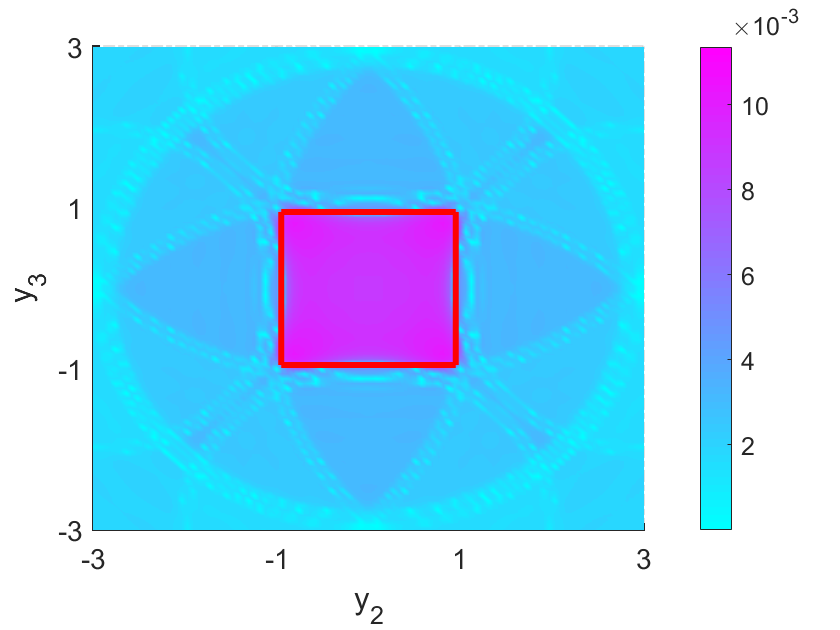}
		
	}
	\caption{Slices of the reconstruction from six observation points at the planes $y_3=0$, $y_2=0$ and $y_1=0$.
	} \label{fig:3d-slice}
\end{figure}

In Figs. \ref{fig:3d-LT} and \ref{fig:3d-LT1}, we show slices of the reconstructions of a cubic source with a longer radiating period $(t_{\min}, t_{\max})$. Different Fourier transform windows from the data measured at six observation points (like the case in Fig.\ref{fig:3d-rec} (c)) are used. The radiating period (resp. Fourier transform window) is taken as $(0,0.5)$ in Fig.\ref{fig:3d-LT} and  $(0,1)$ in Fig. \ref{fig:3d-LT1}. It can be observed that, even for a long duration $T=t_{\max}-t_{\min}$, satisfactory inversions for capturing the shape and location of the source can be achieved.
However, the size of the support becomes distorted as $T$ increases. The quality on reconstructing the size can be improved by increasing the number of frequencies. The longer  the time-dependent source lasts, the more the number of frequencies is needed. Of course the number of observation points affects the resolution of the reconstructions as well. In Fig.\ref{fig:3d-LT1N}, we use more multi-frequency data to improve the quality of reconstruction in Fig.\ref{fig:3d-LT1}. Here we take $N=400$ and $K=20$. It is observed that the cube is nicely restored with more multi-frequency data.

\begin{figure}[H]
	\centering
	\subfigure[Iso-surface level $=2.9\times10^{-2}$]{
		\includegraphics[scale=0.22]{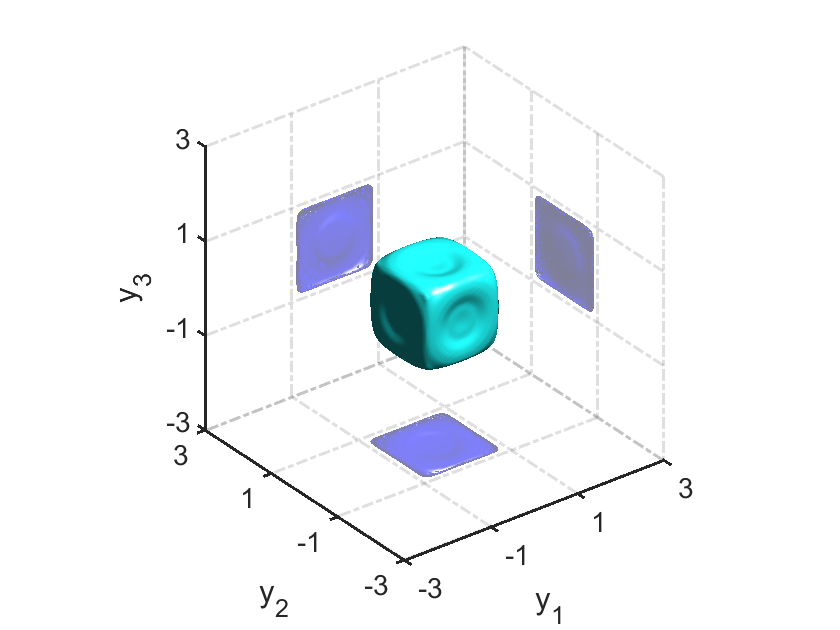}
	}
	\subfigure[Slice at $y_2=0$]{
		\includegraphics[scale=0.22]{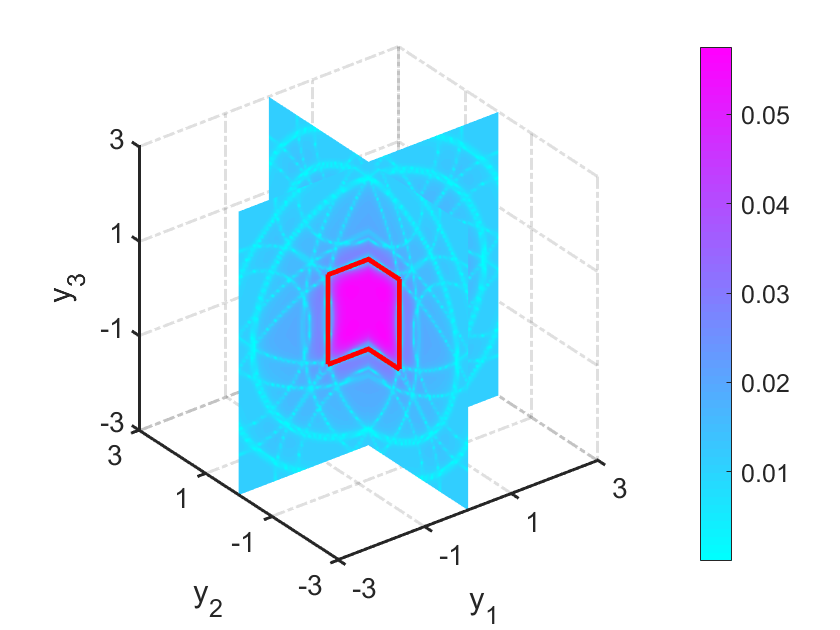}
		
	}
	\subfigure[Slice at $y_1=0$]{
		\includegraphics[scale=0.22]{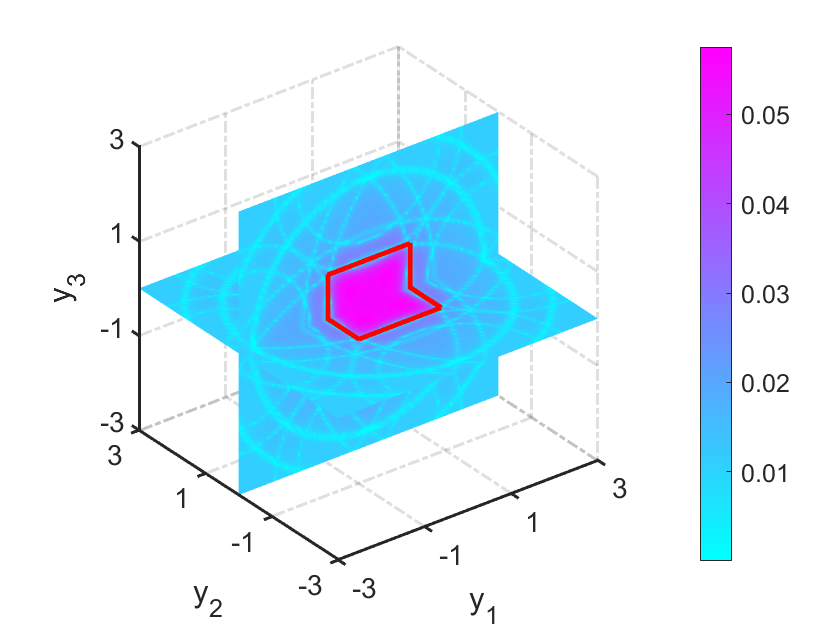}
		
	}
	\caption{Reconstruction of a cube-shaped support using multi-frequency data from six observation points. The Fourier transform window is $(0,0.5)$. } \label{fig:3d-LT}
\end{figure}
%
%
%

\begin{figure}[H]
	\centering
	\subfigure[Iso-surface level $=6\times10^{-2}$]{
		\includegraphics[scale=0.22]{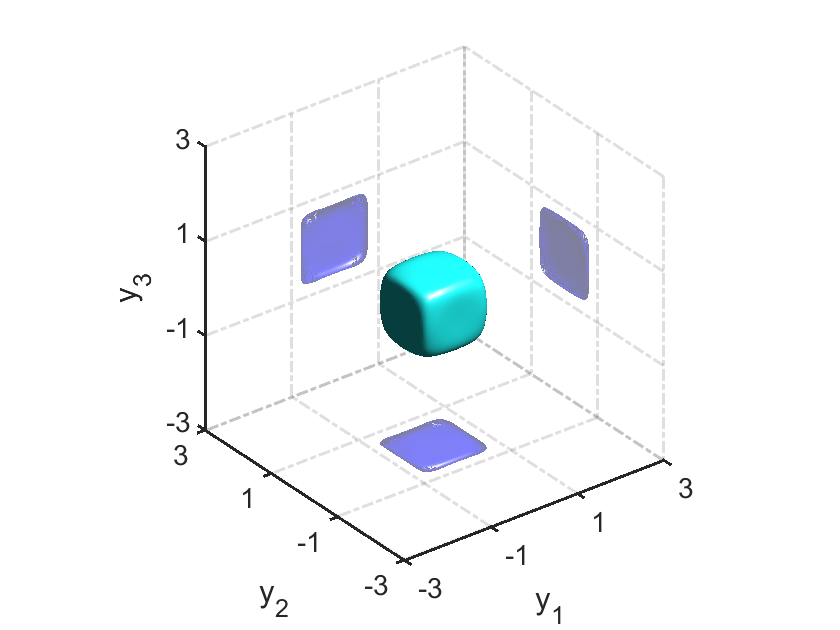}
	}
	\subfigure[Slice at $y_2=0$]{
		\includegraphics[scale=0.22]{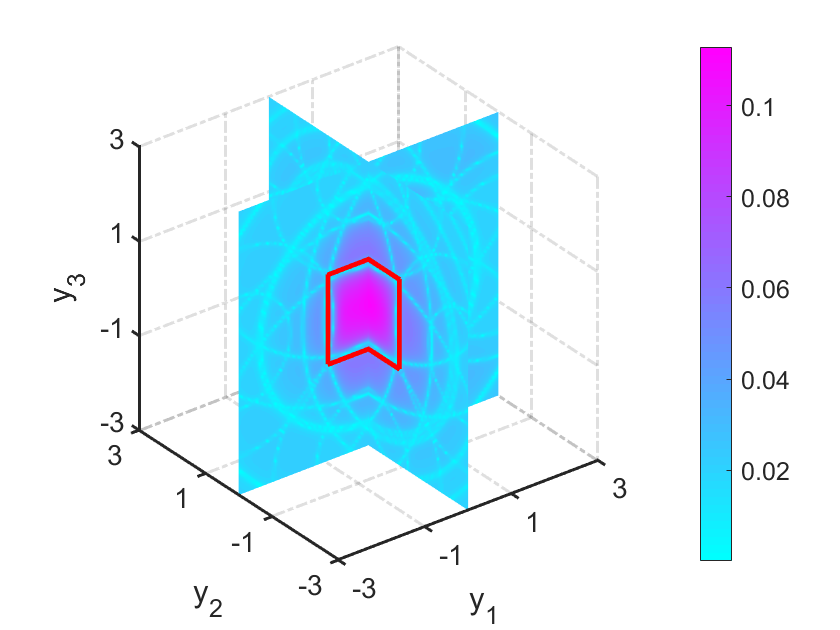}
		
	}
	\subfigure[Slice at $y_1=0$]{
		\includegraphics[scale=0.22]{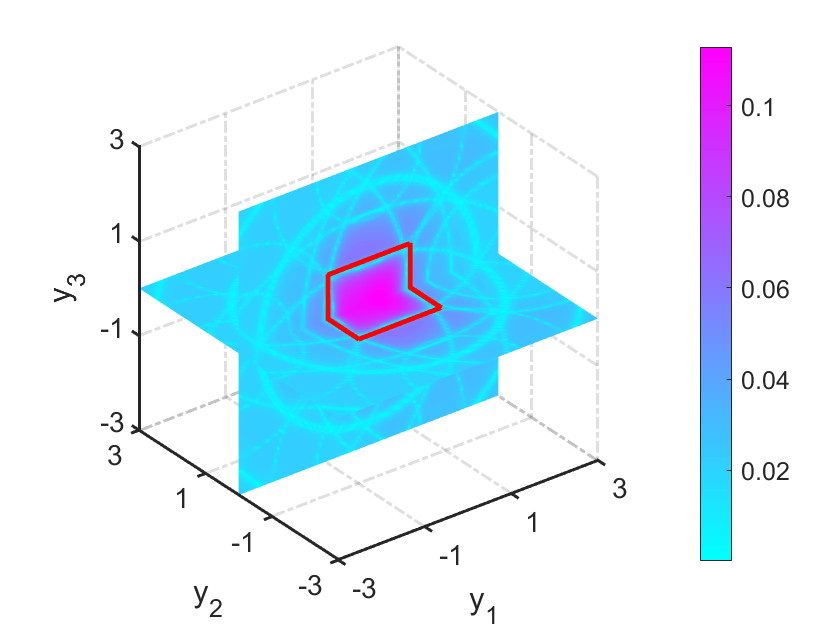}
		
	}
	\caption{Reconstruction of a cube-shaped support using multi-frequency data from six observation points. The Fourier transform window is $(0,1)$. } \label{fig:3d-LT1}
\end{figure}

\begin{figure}[H]
	\centering
	\subfigure[Iso-surface level $=2.68\times10^{-2}$]{
		\includegraphics[scale=0.22]{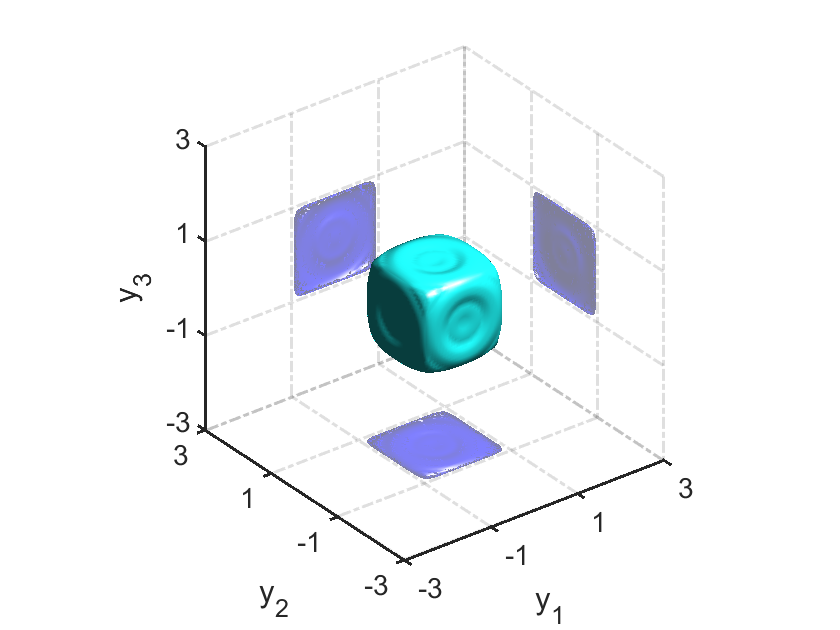}
	}
	\subfigure[Slice at $y_2=0$]{
		\includegraphics[scale=0.22]{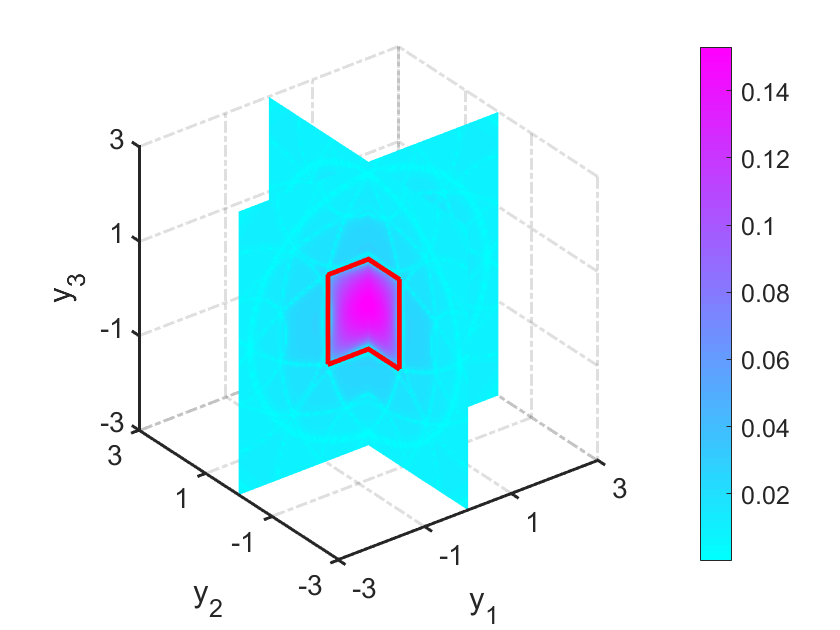}
		
	}
	\subfigure[Slice at $y_1=0$]{
		\includegraphics[scale=0.22]{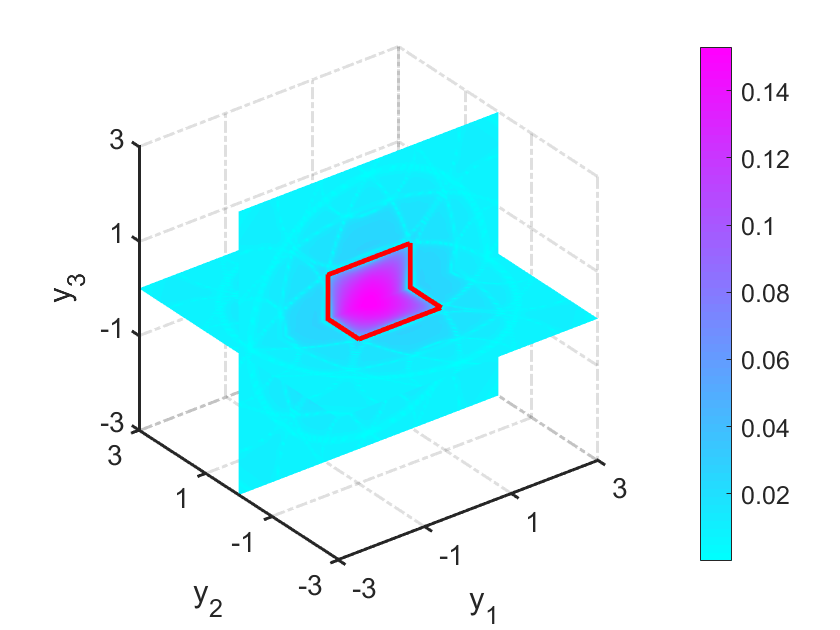}
		
	}
	\caption{Reconstruction of a cube-shaped support using more multi-frequency data from six observation points. The Fourier transform window is $(0,1)$. Here we take $N=400$.} \label{fig:3d-LT1N}
\end{figure}



\section*{Acknowledgements}
G. Hu is partially supported by the National Natural Science Foundation of China (No. 12071236) and the Fundamental Research Funds for Central Universities in China (No. 63213025).


\begin{thebibliography}{00}

\bibitem{AKS} C. Alves, R. Kress, and P. Serranho, Iterative and range test methods for an inverse source problem
for acoustic waves, Inverse Problems, 25 (2009): 055005.


\bibitem{AHLS} A. Alzaalig, G. Hu, X. Liu and J. Sun, Fast acoustic source imaging using multi-frequency sparse data, Inverse Problems, 36 (2020): 025009.


\bibitem{BLT10} G. Bao, J. Lin, and F. Triki, A multi-frequency inverse source problem, J. Differential Equations, 249 (2010): 3443--3465.


\bibitem{BC} N. Bleistein and J. Cohen, Nonuniqueness  in  the  inverse  source  problem  in  acoustics  and electromagnetics, J. Math. Phys., 18 (1977): 194--201.

\bibitem{CCH13} J. Chen, Z. Chen and G. Huang, Reverse time migration for extended obstacles: acoustic
waves, Inverse Problems, 29 (2013): 085005.

\bibitem{CIL} J. Cheng, V. Isakov and S. Lu,  Increasing stability in the inverse source problem with many frequencies, J. Differential Equations, 260 (2016): 4786-4804.


\bibitem{CK} D. Colton,  R.  Kress, \emph{Inverse Acoustic and Electromagnetic Scattering Theory}, Springer Nature, Berlin, 2019.


\bibitem{EH19} J. Elschner and G. Hu, Uniqueness and factorization method for inverse elastic scattering with a single incoming wave, Inverse Problems,  35 (2019): 094002 .


\bibitem{EV09}M. Eller and N. Valdivia, Acoustic source identification using multiple frequency information, Inverse Problems, 25 (2009): 115005.



\bibitem{GS} R. Griesmaire and C. Schmiedecke, A Factorization method for multifrequency inverse source problem with sparse far-field measurements, SIAM J. Imaging Sciences, 10 (2017): 2119-2139.

\bibitem{G11} R. Griesmaier, Multi-frequency orthogonality sampling for inverse obstacle scattering problems, Inverse Problems, 27 (2011):  085005.





\bibitem{GGH} R. Griesmaier, H. Guo and G. Hu, Inverse wave-number-dependent source problems, 2022.

\bibitem{HL2020} G. Hu and J. Li, Uniqueness to inverse source problems in an inhomogeneous medium with a single far-field pattern, SIAM J. Math. Anal., 52 (2020): 5213-5231.


\bibitem{IJZ}K. Ito, B. Jin and J. Zou, A direct sampling method to an inverse medium scattering problem,
Inverse Problems, 28 (2012): 025003.

\bibitem{Ik99} M. Ikehata, Reconstruction of a source domain from the Cauchy data, Inverse Problems, 15
(1999):  637--645.



\bibitem{JLZ2019} X.  Ji, X. Liu and B. Zhang,  Phaseless inverse source scattering problem: phase retrieval, uniqueness and direct sampling methods, J. Comput. Phys. X 1 (2019): 100003.

\bibitem{KG08} A. Kirsch and N. Grinberg, The Factorization Method for Inverse Problems, Oxford University Press,
Oxford, UK, 2008.


\bibitem{KS03} S. Kusiak and J. Sylvester, The scattering support, Comm. Pure Appl. Math., 56 (2003): 1525-1548.

\bibitem{L17} X. Liu, A novel sampling method for multiple multiscale targets from scattering amplitudes at a  fixed frequency,
Inverse Problems, 33 (2017):  085011.

\bibitem{LMZ22} X. Liu, S. Meng and B. Zhang,  Modified sampling method with near field measurements,  SIAM J. Appl. Math., 82 (2022): 244-266.

\bibitem{GH22} G. Ma and G. Hu,  Factorization method with one plane wave: from model-driven and data-driven perspectives. Inverse Problems, 38 (2022):  015003.


\bibitem{NP}G. Nakamura and R. Potthast,  Inverse Modeling-An Introduction to the Theory and Methods of
Inverse Problems and Data Assimilation, Bristol, IOP Publishing, 2015.

\bibitem{P2001}
R. Potthast, Point-sources and Multipoles in Inverse Scattering Theory, Chapman $\&$ Hall/CRC
Research Notes in Mathematics vol 427, Boca Raton, FL: CRC, 2001.








\bibitem{P10} R. Potthast, A study on orthogonality sampling, Inverse Problems, 26 (2010): 074015.







\bibitem{SK05} J. Sylvester and J. Kelly, A scattering support for broadband sparse far field measurements, Inverse
Problems, 21 (2005): 759-771.

\bibitem{ZG} D. Zhang and Y. Guo,  Fourier method for solving the multi-frequency inverse source problem for the Helmholtz equation, Inverse Problems, 31 (2015): 035007.




\end{thebibliography}
\end{document}